\documentclass[12pt]{amsart}
\usepackage[margin=1.25in]{geometry}
\usepackage[utf8]{inputenc}
\usepackage[usenames, dvipsnames]{color}
\usepackage{bm}
\usepackage{amssymb}
\input xy
\usepackage[all]{xy}
\usepackage{tikz-cd}
\usepackage{braids}
\usepackage{placeins}
\usepackage{lscape}
\newcommand{\unit}{\text{\textbf{1}}}

\newcommand{\Inv}{\operatorname{Inv}}

\renewcommand{\Vec}{\text{Vec}}

\newcommand{\cC}{{\mathcal C}}
\newcommand{\cZ}{{\mathcal Z}}

\newcommand{\ot}{{\otimes}}

\newcommand{\Maps}{\operatorname{Maps}}

\newcommand{\ku}{{\Bbbk}}
\newcommand{\Z}{{\mathbb Z}}

\newcommand{\Q}{{\mathbb Q}}

\newcommand{\bZ}{\mathbb{Z}}

\newcommand{\id}{\operatorname{id}}

\newcommand{\cB}{\mathcal{B}}
\newcommand{\cP}{\mathcal{P}}
\newcommand{\cD}{\mathcal{D}}

\newcommand{\Rep}{\operatorname{Rep}}
\newcommand{\Irr}{\operatorname{Irr}}

\newcommand{\Hom}{\operatorname{Hom}}

\newcommand{\End}{\operatorname{End}}

\newcommand{\Id}{\operatorname{Id}}

\newcommand{\Aut}{\operatorname{Aut}}

\newcommand{\otlam}{\stackrel{\lambda}{\otimes}}

\theoremstyle{plain}

\numberwithin{equation}{section}

\newtheorem{theorem}{Theorem}[section]
\newtheorem{lemma}[theorem]{Lemma}

\newtheorem{corollary}[theorem]{Corollary}
\newtheorem{proposition}[theorem]{Proposition}

\theoremstyle{definition}
\newtheorem{definition}[theorem]{Definition}
\newtheorem{example}[theorem]{Example}

\theoremstyle{remark}

\newtheorem{remark}[theorem]{Remark}

\theoremstyle{remark}

\DeclareRobustCommand{\hlGreenYellow}[1]{{\sethlcolor{GreenYellow}\hl{#1}}}

\newcounter{commentcounter}

\newcounter{todocounter}
\newcommand{\todo}[1]
{\stepcounter{todocounter}
  \textbf{Todo\arabic{todocounter}}: 
{\hlGreenYellow{#1}} }

\newcommand{\parcen}[1]{\phantom{(}#1\phantom{)}}
\newcommand{\br}{to [out=90,in=-90]}

\newcommand{\onebox}[3]{
\begin{scope}[line width=1]
\draw (0,0) node[below] {$#1$} --(0,2) node[above] {$#2$};
\draw[fill=white] (-.33,.66) rectangle node {$#3$} (.33,1.33);
\end{scope}
}

\newcommand{\onecirc}[3]{
\begin{scope}[line width=1]
\draw (0,0) node[below] {$#1$} --(0,2) node[above] {$#2$};
\draw[fill=white] (0,1) circle (.45) node {$#3$};
\end{scope}
}

\newcommand{\onetri}[3]{
\begin{scope}[line width=1]
\draw (0,0) node[below] {$#1$} --(0,2) node[above] {$#2$};
\draw[fill=white] (-.5,.5) --(.5,.5)--(0,1.5)--(-.5,.5);
\draw (0,.75) node {\small $#3$};
\end{scope}
}

\newcommand{\kink}[2]{
\begin{scope}[line width=1]
\draw (0,0) node[below] {$#1$}-- (0,1) to [out=90,in=180] (.25,1.25) to [out=0,in=90] (.5,1) to [out=-90,in=0] (.25,.75);
\draw[white,line width=10] (.25,.75) to [out=180,in=-90] (0,1)--(0,2);
\draw (.25,.75) to [out=180,in=-90] (0,1)--(0,2) node[above] {$#2$};
\end{scope}
}

\newcommand{\twobox}[5]{
\begin{scope}[line width=1]
\draw (0,0) node[below] {$#1$} --(0,2) node[above] {$#2$};
\draw (1,0) node[below] {$#3$} --(1,2) node[above] {$#4$};
\draw[fill=white] (-.33,.66) rectangle node {$#5$} (1.33,1.33);
\end{scope}
}

\newcommand{\twoboxnotop}[3]{
\begin{scope}[line width=1]
\draw (0,0) node[below] {$#1$} --(0,1);
\draw (1,0) node[below] {$#2$} --(1,1);
\draw[fill=white] (-.33,.66) rectangle node {$#3$} (1.33,1.33);
\end{scope}
}

\newcommand{\twoboxnobottom}[3]{
\begin{scope}[line width=1]
\draw (0,1) --(0,2) node[above] {$#1$};
\draw (1,1) --(1,2) node[above] {$#2$};
\draw[fill=white] (-.33,.66) rectangle node {$#3$} (1.33,1.33);
\end{scope}
}
\newcommand{\twoboxnostrands}[1]{
\begin{scope}[line width=1]
\draw[fill=white] (-.33,.66) rectangle node {$#1$} (1.33,1.33);
\end{scope}
}

\newcommand{\doublebraid}[2]{
   \draw (1,0) node[below] {$\phantom{#1}#2\phantom{#1}$} to [out=90,in=-90] (0,1);
    \draw[white, line width=10] (0,0) to [out=90,in=-90] (1,1);
    \draw (0,0) node[below] {$\phantom{#2}#1\phantom{#2}$} to [out=90,in=-90] (1,1) to [out=90,in=-90] (0,2) node[above] {$\phantom{#2}#1\phantom{#2}$};
    \draw[white, line width=10] (0,1) to [out=90,in=-90] (1,2);
    \draw (0,1) to [out=90,in=-90] (1,2) node[above] {$\phantom{#1}#2\phantom{#1}$};}

\newcommand{\assoc}[4]{
\begin{tikzpicture}[baseline=75*#4,line width=1, scale=#4]

 \foreach \x in {0,2,4}{
  \draw(\x,0)--(\x,2);
  \draw(\x,5)--(\x,5.5);
  }
  \foreach \y in {-2,-4}
 {
 \draw(\y,0)--(\y,5.5);
 }
  
   \begin{scope}[scale=2,xshift=-1cm,yshift=2.5cm]
  \braid[number of strands=3] a_1^{-1};
  \end{scope}
  
   \draw[fill=white] (1.25,1.25) rectangle node {$#1,#2,#3$} (4.5,2.5);
   
     \draw (-4,0) node[below] {$\parcen{#1}$};
     \draw (-2,0) node[below] {$\parcen{#2}$};
   \draw (0,0) node[below] {$\parcen{#3}$};
 \draw (2,0) node[below] {$(#2,#3)$};
\draw (4,0) node[below] {$(#1,#2#3)$};

   \draw (-4,5.5) node[above] {$\parcen{#1}$};
     \draw (-2,5.5) node[above] {$\parcen{#2}$};
   \draw (0,5.5) node[above] {$(#1,#2)$};
 \draw (2,5.5) node[above] {$\parcen{#3}$};
\draw (4,5.5) node[above] {$(#1#2,#3)$};

  \end{tikzpicture}
  }

\newcommand{\assocfull}[4]{
\begin{tikzpicture}[baseline=75*#4,line width=1, scale=#4]

 \foreach \x in {0,2,4}{
  \draw(\x,0)--(\x,2);
  \draw(\x,5)--(\x,5.5);
  }
  \foreach \y in {-2,-4}
 {
 \draw(\y,0)--(\y,5.5);
 }
  
   \begin{scope}[scale=2,xshift=-1cm,yshift=2.5cm]
  \braid[number of strands=3] a_1^{-1};
  \end{scope}
  
   \draw[fill=white] (1.25,1.25) rectangle node {$\lambda_{#1,#2,#3}$} (4.5,2.5);
   
     \draw (-4,0) node[below] {$\parcen{V_{#1}}$};
     \draw (-2,0) node[below] {$\parcen{W_{#2}}$};
   \draw (0,0) node[below] {$\parcen{Z_{#3}}$};
 \draw (2,0) node[below] {$\lambda(#2,#3)$};
\draw (4,0) node[below] {$\lambda(#1,#2#3)$};

   \draw (-4,5.5) node[above] {$\parcen{V_{#1}}$};
     \draw (-2,5.5) node[above] {$\parcen{W_{#2}}$};
   \draw (0,5.5) node[above] {$\lambda(#1,#2)$};
 \draw (2,5.5) node[above] {$\parcen{Z_{#3}}$};
\draw (4,5.5) node[above] {$\lambda(#1#2,#3)$};

  \end{tikzpicture}
  }

\author[C. Delaney]{Colleen Delaney}
\address{Department of Mathematics, Indiana University}
\email{crdelane@iu.edu}
\author[C. Galindo]{C\'esar Galindo}
\address{ Departamento de Matem\'aticas, Universidad de los Andes, Bogot\'a, Colombia}
\email{cn.galindo1116@uniandes.edu.co}

\author[J. Plavnik]{Julia Plavnik}
\address{Department of Mathematics, Indiana University}
\email{jplavnik@iu.edu}
\author[E. Rowell]{Eric C. Rowell}
\address{Department of Mathematics, Texas A\&M University}
\email{rowell@math.tamu.edu}
\author[Q. Zhang]{Qing Zhang}
\address{Department of Mathematics, Texas A\&M University}
\email{zhangqing@math.tamu.edu}
\begin{document}

\title[Braided zesting]{Braided zesting and its applications}

\thanks{The authors gratefully acknowledge the support of the American Institute of Mathematics, where this collaboration was initiated. C.G would like to thank the hospitality and excellent working conditions of the Department of Mathematics at the University of Hamburg, where he carried out this research as a Fellow of the Humboldt Foundation. J.P was partially supported by US NSF Grants DMS-1802503 and DMS-1917319. E.C.R. was partially supported by US NSF Grant DMS-1664359, a Simons Foundation Fellowship, and a Texas A\&M Presidential Impact Fellowship.  Part of this research was carried out while CD, JP, ER  and QZ participated in a semester-long program at MSRI, which is partially supported by NSF grant DMS-1440140. We thank Richard Ng for useful discussions.}

\begin{abstract}
We give a rigorous development of the construction of new braided fusion categories from a given category known as zesting.  This method has been used in the past to provide categorifications of new fusion rule algebras, modular data, and minimal modular extensions of super-modular categories.  Here we provide a complete obstruction theory and parameterization approach to the construction and illustrate its utility with several examples. 
\end{abstract}

\subjclass[2000]{16W30, 18D10, 19D23}

%\date{\today}
\maketitle

\section{Introduction}

Despite recent progress on the classification of braided fusion categories, the general landscape remains largely unexplored.  This is partly due to our lack of well-studied examples.  Most come from a few basic classes of fusion categories: subquotients of representation categories of quantum groups at roots of unity, representations of quasi-Hopf algebras, bimodule categories over finite index finite depth subfactors, planar algebras, and near-group categories.  There are a few inter-related tools for obtaining new fusion categories from old such as the Drinfeld center construction, graded extensions by finite groups $G$, $G$-equivariantization/$G$-de-equivariantization, symmetry gauging, and Deligne products.  In this paper, we develop another recent construction known as \emph{zesting}.

Zesting of braided fusion categories first explicitly appeared as a construction technique in \cite{AIM2012} for the purpose of categorifying a mysterious rank $10$, dimension $36$ fusion rule algebra, with fusion rules  reminiscent of, but distinct from, those of $SU(3)_3$. It was recognized that the fusion rules could be obtained from those of $SU(3)_3$ by rearranging them slightly, via a twisting of the tensor product.  This new (zested) fusion category was expected to admit a modular structure, but no proof was available until now.  

The basic goal of zesting is to define new (possibly braided, ribbon) fusion categories from a given $A$-graded  (braided, ribbon) fusion category $\cC=\bigoplus_{a \in A} \cC_a$ by defining a new tensor product    $X_a\stackrel{\lambda}{\ot} Y_b:=(X_a\,\otimes\, Y_b)\otimes\,\lambda(a,b)$, where $X_a\in\cC_a$ and  $Y_b\in\cC_b$ are simple objects in their corresponding graded components and $ \lambda(a,b)\in\cC_e$ is an invertible object in the trivial component.  

Zesting fits into the more general context of graded extensions found in \cite{ENO3}--as a fusion category a zesting of an $A$-graded braided fusion category $\cC$ is an $A$-graded extension of the trivial component $\cC_e$.
As one expects from the results of \emph{loc. cit.}, there are obstructions to $(\cC,\stackrel{\lambda}{\ot},\unit)$ admitting the structure of a monoidal category for a given choice of $\lambda:A\times A\rightarrow \cC_{pt}\cap \cC_{e}$.  That is, it is not immediate that associativity morphisms satisfying pentagons exist, and when these obstructions vanish there are inequivalent choices of associativity morphisms.  Fixing a particular such \emph{associative zesting} one can further investigate whether the category admits a braiding, which leads to more obstructions and choices.  Such a braided structure is called a \emph{braided zesting}.  Finally, for a fixed braided zesting of a ribbon fusion category we may look for a balancing structure, which we call a \emph{twist zesting} in general and a \emph{ribbon zesting} in the case the twist has the ribbon property.

We hasten to point out that our associativity, braiding and twist choices for a zesting are assumed to only depend on the grading group $A$, the pointed subcategory of the trivial component, and the structures already present in $\cC$.  It follows that the trivial component of $\cC$ and that of any zesting of $\cC$ are equivalent as fusion categories so that a zesting is always an extension (in the sense of \cite{ENO2}) of the trivial component by the group $A$.  Moreover, at each step, some of these extensions may fail to be realized by our construction. For example, the Ising categories and the pointed modular categories with fusion rules like $\Z/4$ are braided $\Z/2$-extensions of the category of super-vector spaces; however it is not possible to use zesting to construct one from the other, since any zesting of a pointed fusion category remains pointed. 
%For example, we cannot in general expect to obtain as a zesting of a braided fusion category $\cC$ the category with reversed braiding $\cC^{rev}$ since these typically have different braidings on the trivial component 
 On the other hand, we can explicitly obtain both new fusion categories that do not admit braidings and new ribbon braided fusion categories with explicit formulas of their modular data from our approach.

There are two related constructions found in the literature that should be mentioned.  The first is gauging \cite{SCJZ}: one begins with a modular category $\cB$ that admits an action of a finite group $G$ by braided tensor  autoquivalences, and constructs, assuming certain cohomological obstructions vanish, new modular categories $\cC$. The first step is to construct the $G$-crossed braided categories $\cD$ with $\cB$ as the trivial component (using \cite{ENO2}), and the second takes the $G$-equivariantization, which will be modular.  In some cases zestings can be placed in this framework.  If $\cC$ is modular and $\cC_{pt}\cap\cC_{e}\cong\Rep(G) $ is Tannakian with $G$ abelian, we may set $\cB=(\cC_e)_G$ the $G$-de-equivariantization of the trivial component.  Then any modular zesting $\tilde{\cC}$ will be a $G$-gauging of $\cB$.  The second construction is related: if $ \cC_{pt}\cap\cC_{e}\cong\Rep(G)$ as above, we may construct new categories as tensor products over $G$ by condensing the diagonal algebra in $\cC\boxtimes\Rep(D^\omega G)$.  

While both of these constructions have the advantage of providing various structures from general arguments, zesting has several key features that these do not: 1) one has the fusion rules at the outset, 2) we produce new categories that are not necessarily modular or even braided,  3) it depends only on essentially cohomological choices, and 4) in the case that the resulting category is modular we have explicit formulas for the modular data.

In a bit more detail, the construction goes as follows. For a fixed $A$-graded ribbon category $\cC$ where $A=U(\cC)$ is the universal grading group:
\begin{enumerate}
    \item Fix a normalized $2$-cocycle $\lambda\in Z^2(A,\Inv(\cC_e))$.
    \item Find a $3$-cochain $\lambda\in C^3(A,\ku^\times)$ satisfying the second associative zesting constraint, if possible (see Figure \ref{axiom:asso}).  The set of all associative zestings form a torsor over $H^3(A,\ku^\times)$ for the chosen $2$-cocycle.
    \item For a fixed associative zesting, a braided zesting is determined by a cochain $c\in C^2(A,\ku^\times)$ such that $t(a,b):=c(a,b)\id_{\lambda(a,b)}$ satisfies the two braided zesting constraints (see Figures \ref{fig:unnormalized-1} and \ref{fig:normalized-2}).  For a fixed such $c(a,b)$ the set of all braided zestings forms a torsor over the group of bicharacters of $A$.
    \item For a fixed braided zesting of a braided fusion category $\cC$ with a twist $\theta$, we may determine all braided twist (ribbon) zestings in terms of a function $f:A\rightarrow\ku^\times$, as in Corollary \ref{lem:twist}, and all twist (ribbon) zestings form a torsor over the characters of $A/2A$. 
\end{enumerate}

More general situations can be considered as well, for example, we may choose a different grading group $B$ for $\cC$ (i.e., a quotient of $U(\cC)$) and the above is still true provided the image of $\lambda\in Z^2(A,\Inv(\cC_e))$ centralizes the trivial component $\cC_e$ with respect to the grading $B$.  Failing this, we may still develop a theory of zesting, but there are several subtleties that must be addressed in the form of additional choices and more elaborate constraints.  Moreover, the general definition of associative zesting does not require a braiding -- one may apply it to any fusion category by passing to the relative centralizer.

Zesting supersedes several known constructions as special cases.  If one chooses the trivial $2$-cocycle in step (1), the second step is the well-known associativity twist (see, e.g. \cite{KazWen}).  If one makes the trivial choice of $2$-cocycle and $3$-cochain  in steps (1) and (2), then the braided zestings are simply modifying the braiding by a bicharacter, which is also well-known.  Finally, if one takes the trivial choice in each of steps (1), (2), and (3) for a braided fusion category with a twist then the last step is a change of pivotal structure on the underlying braided fusion category.  

Here is a more detailed explanation of the contents of this article.  In Section \ref{prelim}, we lay out the basic definitions and general results from the literature that we use in the sequel.  Section \ref{sec:associative zesting} contains the general definition of associative zesting and the rigidity structure, the notational conventions for diagrams and the obstruction theory. Section \ref{section: braided zesting} details braided zesting, in which we study the braided structures on associative zestings and the attendant obstruction theory.  Section \ref{section:twist zesting} studies the twist and ribbon twist structures on a braided zesting and the corresponding categorical trace and modular data.
We illustrate our techniques with several examples coming from quantum groups of type $A$ in  Section \ref{section: applications}.  After submitting our paper a related manuscript \cite{DN} was posted which has some overlap with our results, which we address in Section \ref{section: davydov nikshych}.

\section{Preliminaries}\label{prelim}

\subsection{Group cohomology}
To fix notation, we will recall the basic definitions of the standard cocycle description of group cohomology, for more details, see \cite{weibel_1994}.

Let $G$ be a group and $M$ a $G$-module. We will denote by  $\big (C^n(G,M), \delta \big )$ the cochain complex 
\[0\to C^0(G,M)\to C^1(G,M)\to C^2(G,M)\to \cdots \to C^n(G,M)\to \cdots,\]
where $C^0(G,M)=M$, $C^n(G,M)$ is the abelian group of all maps from  $G^{\times n}$ to $M$  and $\delta:C^n(G,M)\to C^{n+1}(G,M)$ is given by
\begin{align}
\delta^n(f)(g_1,\ldots, g_{n+1})= &\,g_1f(g_2,\ldots,g_{n+1})+ \sum_{i=1}^{n}(-1)^if(g_1,\ldots,g_ig_{i+1},\ldots,g_{n+1})\label{cochains G}\\
&+(-1)^{n+1}f(g_1,\ldots,g_{n}). \notag 
\end{align} 
The group cohomology of $G$ with coefficients in  $M$ is defined as the cohomology of the cochain complex \eqref{cochains G}, that is, $H^n(G,M)=\ker(\delta^n)/\operatorname{Im}(\delta^{n-1})$. As usual, we will denote by $Z^n(G,M)=\ker(\delta^n)$ the group of  $n$-cocycles and by $B^n(G,M)=\operatorname{Im}(\delta^{n-1})$ the $n$-coboundaries.

\subsection{Basic definitions on fusion and modular categories}

In this section, we recall some basic definitions and standard notions from \cite{EGNO}, mainly in order to fix notation.

By a \emph{monoidal category} we mean a tuple $(\cC,\otimes, \alpha, \unit,\lambda,\rho)$, where $\cC$ is a category, $\otimes: \cC\times \cC\to \cC$ is a   bifunctor,  a natural isomorphism $$\alpha_{X,Y,Z}:(X\otimes Y)\otimes Z\to X\otimes (Y\otimes Z),$$ called \emph{the associator}, natural isomorphisms 
\begin{align*}
\bold{\lambda}_X:\unit\otimes X\to X,&& \bold{\rho}_X:X\otimes \unit\to X
\end{align*}
called the right and left \emph{unitors}, respectively. This data must satisfy the well known pentagon and triangle axioms \cite{EGNO}. Hereafter we suppress the associators and unitors and denote a monoidal category by the tuple $(\cC,\otimes, \bold{1})$. Throughout this paper, we will always assume that the monoidal unit is simple (as it is the case for fusion categories).

A monoidal category has \emph{duals} if for every $X\in \cC$ there is an object $X^*\in \cC$ and morphism $\epsilon_X:X^*\otimes X\to \unit,$ $\delta_X:\unit \to X\otimes X^*$ satisfying  the zig-zag axioms:
\begin{align*}
    (\id_{X} \otimes \epsilon_X)\circ \alpha_{X, X^*, X}\circ (\delta_X\otimes \id_X)&=\id_{X}\\
    (\epsilon_{X} \otimes \id_{X^*})\circ \alpha^{-1}_{X^*, X, X^*}\circ (\id_{X^*} \otimes \delta_{X})&=\id_{X^*}.
\end{align*}
A monoidal category with duals is called \emph{rigid} if for every $X\in\cC$ there is $^*X\in \cC$ such that $(^*X)^*\cong X$.

We will denote by $\ku$  an algebraically closed field of characteristic zero.  By a \emph{fusion category}, we mean a semisimple $\ku$-linear abelian rigid monoidal category $(\cC,\otimes,\unit)$  such that the unit object $\unit$ of $\cC$ is simple and there are finite many isomorphism classes of simple objects. The set of isomorphism classes of simple objects of $\cC$ is denoted by $\Irr(\cC)$. By a \emph{fusion subcategory} of a fusion category, we mean a full monoidal abelian subcategory.

For a fusion category $\cC$, we will denote by $\cC_{\operatorname{pt}}$ the full fusion subcategory generated by $\otimes$-invertible objects. We will denote by $\Inv(\cC)$ the group of isomorphism classes of $\otimes$-invertible objects of $\cC$ under the tensor product.  A fusion category is called \emph{pointed} if every simple object is $\ot$-invertible.

\begin{example}[Pointed fusion categories]\label{Ex: pointed fusion cat}
Let $G$ be a finite group. A (normalized)  3-cocycle $\omega \in Z^3(G, \ku^\times)$ is a map $\omega:G\times G\times G\to \ku^{\times}$ such that 
\begin{align*}
\omega(ab,c,d)\omega(a,b,cd)&=
\omega(a,b,c)\omega(a,bc,d)\omega(b,c,d), & \omega(a,1,b)=1,
\end{align*}
for all   $a,b,c,d\in G. $ 

Let us recall the description of the pointed fusion category $\Vec_G^\omega$. The objects of $\Vec_G^\omega$ are $G$-graded finite dimensional vector spaces $V=\bigoplus_{g\in G} V_{g}$. Morphisms are $G$-linear $G$-homogeneous maps. 
The tensor product of $V=\oplus_{g\in G}V_g$ and $W=\oplus_{g\in G}W_g$  is
$V\otimes W$ as vector space, with $G$-grading
\[(V\otimes W)_g=\bigoplus_{h\in G}V_h\otimes W_{h^{-1}g}.\]
For objects $V, W, Z \in \Vec_G^\omega$ the associativity constraint is defined by \begin{align*}
a_{V,W,Z}: (V\otimes W)\otimes Z&\to V\otimes (W\otimes Z)\\ 
(v_g\otimes w_h)\otimes z_k &\mapsto \omega(g,h,k) v_g\ot( w_h\otimes z_k)
\end{align*}
for all $g,h,k \in G, v_g\in V_g, w_h\in W_h, z_k\in Z_k$.  The unit objects is $\ku_e$, the vector space $\ku$ graded only by the identity element $e\in G$.

For  $V \in \Vec_G^\omega$, the dual object is $V^*=\Hom_\ku(V,\ku)$, with $G$-grading  $V^*_g=\Hom_\ku(V_{g^{-1}},\ku)$ and 
\begin{align*}
\epsilon_V: V^*\otimes V &\to \ku_e,  &
\delta_V: \ku_e &\to V\otimes V^*\\
 \alpha_{h}\otimes v_g&\mapsto  \omega(g,g^{-1},g)^{-1}\alpha_{h}(v_g),& 
1 &\mapsto \sum v_i\otimes v^i
\end{align*}
where $g,h \in G$, $v_g\in V^*_g$ and $\alpha_{h}\in V_h^*$,
and $\delta_V$ is the usual coevaluation map of finite dimensional vector spaces.

\end{example}

\subsection{Pivotal and spherical fusion categories}

If $\cC$ is a monoidal category with duals, we can define a monoidal functor $(-)^*:\cC\to \cC^{\operatorname{op}}$, where $\cC^{\operatorname{op}}$ is the opposite category with tensor product $X\ot^{\operatorname{op}} Y:=Y\otimes X$. Here, for a morphism $f:X\to Y$, we have that $f^*:Y^*\to X^*$ is given by $(\epsilon_Y\otimes \id_{X^*})\circ (\id_{Y^*}\otimes f\otimes \id_{X^*})\circ (\id_{Y^*}\otimes \delta_{X})$.

A \emph{pivotal structure} on a rigid monoidal category is monoidal natural isomorphism $\psi:\operatorname{Id}_{\cC}\to (-)^{**}$. The left and right pivotal traces of an endomorphism $f:X\to X$ are given by 
\begin{align*}
\operatorname{Tr}_L(f)&=\epsilon_X\circ (\id_{X^*}\otimes f)\circ (\id_{X^*}\otimes \psi_X^{-1})\circ \delta_{X^*}\\
\operatorname{Tr}_R(f)&=\epsilon_{X^*}\circ (\psi_X\otimes \id_{X^*} )\circ (f\otimes \id_{X^*})\circ \delta_{X}.
\end{align*}  

A rigid monoidal category with a pivotal structure is called \emph{pivotal} monoidal category. 

A \emph{spherical} fusion category is a pivotal fusion category such that the left and right traces of every endomorphism coincide. For spherical fusion categories the left and right trace of an endomorphism $f$ will be denoted simply by $\operatorname{Tr}(f)$. The \emph{quantum dimension} or just the \emph{dimension} of an object $X\in \cC$ is $\dim(X)=\operatorname{Tr}(\id_X)$.

\begin{example}
The pointed fusion category $\Vec_G^\omega$ has a canonical pivotal structure given by 
\[\psi_V=\bigoplus_{g\in G}\omega(g^{-1},g,g^{-1})\id_{V_g} .\]
Any other pivotal structure differs from the canonical one by a linear character $\chi:G\to \ku^\times$, given by
\begin{equation}\label{spherical pointed}
\psi_V^\chi=\bigoplus_{g\in G}\chi(g)\omega(g^{-1},g,g^{-1})\id_{V_g} .
\end{equation}
The pivotal structure $\psi_V^\chi$ is spherical if and only if $\chi(g)\in \{1,-1\}$.
\end{example}

\subsection{Premodular and modular tensor categories}
A \emph{braiding} for a monoidal category $\cB$ is a natural isomorphism 
\begin{align*}
c_{X,Y}:X\otimes Y \to Y\otimes X, && X,Y\in \cB
\end{align*}satisfying the two well known hexagon axioms. A \emph{braided fusion category} is a fusion category with a braiding. A braided category is called \emph{symmetric} if $c_{X,Y}^{-1}=c_{Y,X}$ for all $X,Y\in \cB$. The centralizer of a set of objects $\mathcal{S}$ is the subcategory with objects \[C_{\mathcal{S}}(\cB):=\{X\in \cB: c_{Y,X}\circ c_{X,Y}=\id_{X\otimes Y}, X\in\mathcal{S}\}.\]  An object $Y$ is called \emph{transparent} if $C_{\{Y\}}(\cB)=\cB$, so that every object is transparent in a symmetric category.

%The \emph{Müger'}s center of $\cB$ is the full subcategoryv\[\mathcal{Z}_2(\cB)=\{X\in \cB: c_{Y,X}\circ c_{X,Y}=\id_{X\otimes Y}\}.\]

In \cite{delignetannakian}, Deligne establishes that every symmetric fusion category is braided equivalent to one of the following:

\begin{itemize}
    \item \emph{Tannakian categories.} These take the form $\operatorname{Rep}(G)$ of finite dimensional representations of a finite group $G$, with the standard braiding $c_{X,Y}(x \otimes y) := y \otimes x$.
    \item \emph{Super-Tannakian categories.} These are categories of finite-dimensional representations of finite super-groups, denoted by $\operatorname{Rep}(G,z)$. A finite super-group is a pair $(G, z)$,
where $G$ is a finite group and $z$ is a central element of order two. As fusion categories they can be understood as $\operatorname{Rep}(G)$ but with a non-standard braiding $c^z$:\\
An irreducible representation of $G$ is called odd if $z$ acts as the
scalar $-1$, and even if $z$ acts as the identity. If the degree of a
simple object $X$ is denoted by $|X|\in \{0, 1\}$, then a braiding on $\operatorname{Rep}(G)$ is given by
$c^z_{X,Y}(x \otimes y) := (-1)^{|x||y|}y \otimes x$ for $x \in X$ and $y\in Y$, where $X$ and $Y$ are simple representations. 
\end{itemize}
 
A \emph{twist} for a braided category is a natural isomorphism of the identity $$\theta_X:X\to X$$ such that \[\theta_{X\otimes Y}= c_{Y,X}\circ c_{X,Y}\circ (\theta_{X}\otimes \theta_{Y}),\] for all $X, Y\in \cC$. It is well-known (see e.g. \cite{BNRW}) that pivotal structures on a braided fusion category are in bijective correspondence with twists. A twist on a braided fusion category is called a \emph{ribbon twist} if $\theta_{X^*}=\theta_X^*$ for all $X\in \cC$, and ribbon twists correspond to spherical pivotal structures under the pivotal/twist correspondence. We will recall this correspondence briefly for more details, see \cite[Appendix A.2]{henriques2016a}.

Let $\cC$ be a rigid braided monoidal category. The \emph{Drinfeld isomorphism} is a natural isomorphism $u:\Id\to (-)^{**}$ given by 
\begin{align*}
u_X:=(\epsilon_X\otimes \id_{X^{**}})\circ (\id_{X^*}\otimes c^{-1}_{X^{**},X})\circ (\delta_{X^*}\otimes \id_{X}):X\to X^{**}.
\end{align*}  
\begin{center}
\begin{tikzpicture}[line width=1]
\draw (2,-1) node[below] {$X^{**}$}--(2,0) \br (1,2);
\draw[line width=10, white] (1,0) \br (2,2);
\draw (1,0) \br (2,2) -- (2,3) node[above] {X};
\draw [looseness=2] (1,0) to [out=-90, in=-90] (0,0)--(0,2) to [out=90, in=90] (1,2);
\end{tikzpicture}
\end{center}

Given a twist $\theta$, the natural isomorphism
\begin{align*}
\psi_X=u_X\circ \theta_X, && X\in \cC
\end{align*}
is a pivotal structure. Conversely, if $\psi_X:X\to X^{**}$ is a pivotal structure then \begin{align*}
\theta_X=u_X^{-1}\circ \psi_X, && X\in \cC,
\end{align*}is twist.

A braided fusion category with a spherical pivotal structure (or equivalently a ribbon twist) is called a premodular tensor category.

Following \cite{BK}, we define the \emph{modular data} of a premodular category as the following pair of matrices with respect to the a basis given by a fixed ordering of $\Irr(\cB)$:

\begin{itemize}
    \item[(i)] \emph{$S$-matrix}. $S_{X,Y}=\operatorname{Tr}(c_{Y^*,X}\circ c_{X,Y^*})$,
    \item[(ii)]  \emph{$T$-matrix.} $T_{X,Y}=\theta_X\delta_{X,Y}$.
\end{itemize}

Notice that the categorical dimension of a simple object $X$ is $\dim(X)=S_{X,\unit}$.  
A premodular tensor category is called \emph{modular} if $S$ is invertible. Any modular tensor category defines a projective representation of the modular group $\operatorname{SL}(2,\Z)$ as follows: the matrices
\begin{align*}
 \mathfrak{s}:=\begin{pmatrix}
0 & -1 \\ 1& 0
\end{pmatrix}, &&  \mathfrak{t}:=\begin{pmatrix}
1 & 1 \\ 0& 1
\end{pmatrix}
\end{align*}
generates $\operatorname{SL}(2,\Z)$ and by \cite{BK} the assignment  
\begin{align*}
\mathfrak{s}\mapsto \frac{1}{\sqrt{\dim(\cB)}}S, && \mathfrak{t}\mapsto T
\end{align*}defines a projective representation, where $\dim(\cB)=\sum_{X\in \Irr(\cB)}\dim(X)^2$.
\begin{remark}
In \cite{EGNO} the $S$-matrix and $T$-matrix are defined by 
\begin{align*}
S_{X,Y}'=\operatorname{Tr}(c_{Y,X}\circ c_{X,Y}), && T'_{X,Y}=\theta_X^{-1}\delta_{X,Y}.
\end{align*}
the $(S',T')$ and $(S,T)$ are directly related by $T'=T^{-1}$ and $S'=S^{-1}$, see \cite[Proposition 8.14.2]{EGNO}.
\end{remark} 

\begin{example}[Pointed braided fusion categories]\label{Ex pointed braid}
Let $\Vec_G^\omega$ be a pointed fusion category, with $G$ abelian. A braiding on $\Vec_G^\omega$ is  defined by a function $c:G\times G\to \Bbbk^{\times}$ as
\begin{align*}
c_{V,W}:V\otimes W &\to W\otimes V\\
v_g\otimes w_h &\mapsto c(g,h)w_h\otimes v_g.
\end{align*}
The function $c$ must satisfy the following equations:
\begin{align}\label{eq:abelian-cocycle}
&
\begin{aligned}
\frac{c(g,hk)}{c(g,h)c(g,k)}&=\frac{\omega(g,h,k)\omega(h,k,g)}{\omega(h,g,k)}\\
\frac{c(gh,k)}{c(g,k)c(h,k)}&=\frac{\omega(g,k,h)}{\omega(g,h,k)\omega(k,g,h)},
\end{aligned}
& 
\text{for all } & g,h,k \in G,
\end{align}
These equations correspond to the hexagon axioms.
A pair $(\omega,c)$ satisfying \eqref{eq:abelian-cocycle} is called an \emph{abelian 3-cocycle}.
Following \cite{EM1,EM2} we denote by $Z^3_{ab}(G, \Bbbk^\times)$ the abelian group of all abelian 3-cocycles $(\omega,c)$.

An abelian 3-cocycle $(\omega,c)\in Z_{ab}^3(G,\Bbbk^\times)$ is called an \emph{abelian 3-coboundary} if there is $\alpha:G^{\times 2}\to \Bbbk^\times$, such that
\begin{align}\label{eq:abelian-coboundary}
&
\begin{aligned}
\omega(g,h,k)&=\frac{\alpha(g,h)\alpha(gh,k)}{\alpha(g,hk)\alpha(h,k)}\\
c(g,h)&=\frac{\alpha(g,h)}{\alpha(h,g)},
\end{aligned}
& 
\text{for all } & g,h,k \in G.
\end{align}
$B^3_{ab}(G,\Bbbk^\times)$ denotes the subgroup of $Z_{ab}^3(G,\Bbbk^\times)$ of abelian 3-coboundaries. The quotient group $H^3_{ab}(G,\Bbbk^\times):=Z_{ab}^3(G,\Bbbk^\times)/B^3_{ab}(G,\Bbbk^\times)$ is called the \emph{third group of abelian cohomology} of $G$.

Under the correspondence between pivotal structures and twist, we have that pivotal structure $\psi^\chi$ in \eqref{spherical pointed} corresponds to the twist 
\[
\theta_{V}^\chi=\bigoplus_{g\in G}\chi(g)c(g,g)\id_{V_g}.
\]
\end{example}

\subsection{Graded fusion categories and tensor natural isomorphisms of the identity functor}

Let $G$ be a finite group. A fusion category $\cC$ is $G$-graded if there is a decomposition \[\cC=\bigoplus_{g\in G}\cC_g\]of $\cC$ into a direct sum of full abelian subcategories such that the tensor product of $\cC$ maps $\cC_g\times \cC_h$ to $\cC_{gh}$ for all $g, h\in G$. We will say that the $G$-grading is faithful if $\cC _g\neq 0$ for all $g\in G$.

\begin{example}\label{universal abelian grading}
Let $\cC$ be a fusion category and $A=\widehat{\Aut_\ot(\Id_\cC)}$  the group of linear characters of $\Aut_\ot(\Id_\cC)$,  the abelian group of tensor automorphism of the identity functor.  Then $\cC=\bigoplus_{\gamma \in A} \cC_{\gamma}$ is faithfully $A$-graded, where

\[\cC_\gamma=\{X\in \cC: \rho_X=\gamma(\rho)\id_X, \quad \forall \rho \in \Aut_\ot(\Id_\cC) \}.\]
\end{example}

It was proved in \cite[Theorem 3.5]{GELAKI20081053} that any fusion category
$\cC$ is naturally graded by a group $U(\cC)$, called the \emph{universal grading group} of $\cC$, and the \emph{adjoint fusion subcategory} $\cC_{ad}$ (tensor generated by all subobjects of $ X \otimes X^*$ for all simple objects $X$) is the trivial component of this grading. Additionally, any other faithful grading arises from a quotient of $U(\cC)$ .

\begin{definition}
Let $\cC$ be a faithfully $G$-graded fusion category. We will denote by $\Aut_{\ot}^G(\Id_\cC)$ the abelian group of all tensor natural isomorphisms of the identity $\Phi \in \Aut_{\ot}(\Id_\cC)$ such that $\Phi_X=\id_X$ for all $X\in \cC_e$. 
\end{definition}

Let $\cC$ be a faithfully $G$-graded fusion category.   It follows as in \cite[Proposition 3.9]{GELAKI20081053} that given $\gamma \in \widehat{G}$ the assignment  $\Phi_\gamma\in \Aut_\ot(\Id_\cC)$ given by \[\Phi_\gamma(X_s)=\gamma(s)\id_{X_s}, \quad \quad X_s \in \cC_s,\]defines a group homorphism from $\widehat{G}$ to $\Aut_\ot^G(\Id_{\cC})$.

The following result is a direct consequence of \cite[Proposition 3.9]{GELAKI20081053}.
\begin{proposition}\label{prop: iso id-g}
Let $\cC$ be a faithfully $G$-graded fusion category. The group homomorhism $\Phi:\widehat{G}\to \Aut_{\ot}^G(\Id_\cC)$ is an isomorphism. In particular $\Phi:\widehat{U(\cC})\cong \Aut_{\ot}(\Id_\cC)$ so that any $\psi\in\Aut_\ot(\Id_\cC)$ is constant on $U(\cC)$-graded components.
\end{proposition}

Let $\cB$ be a braided monoidal category and $a\in\cB$ an \emph{invertible} object. Then we can define a natural isomorphism of the identity functor $\chi_a \in \Aut(\Id_{\cB})$ by the equality  $\chi_a(X)\otimes \id_a= c_{a,X}\circ c_{X,a}$ for  $X\in \operatorname{Obj}(\cB)$, see Figure \ref{fig:chi}. 

\begin{figure}[h]
    \begin{center}
    \begin{tikzpicture}[scale=1.5,xscale=-1,line width=1,baseline=40]
    \draw (1,0) node[below] {$X$} to [out=90,in=-90] (0,1);
    \draw[white, line width=10] (0,0) to [out=90,in=-90] (1,1);
    \draw (0,0) node[below] {$\phantom{X}a\phantom{X}$} to [out=90,in=-90] (1,1) to [out=90,in=-90] (0,2) node[above] {$\phantom{X}a\phantom{X}$};
    \draw[white, line width=10] (0,1) to [out=90,in=-90] (1,2);
    \draw (0,1) to [out=90,in=-90] (1,2) node[above] {$X$};
    \end{tikzpicture}
    \hspace{5pt}:=\hspace{15pt}
    \begin{tikzpicture}[scale=1.5,line width=1,baseline=40]
    \onebox{X}{X}{\chi_a}
    \draw (1,0) node[below]{$\phantom{X}a\phantom{X}$}--(1,2) node[above]{$\phantom{X}a\phantom{X}$};
    \end{tikzpicture}
    \end{center}
    \caption{Definition of $\chi$}
    \label{fig:chi}
\end{figure}
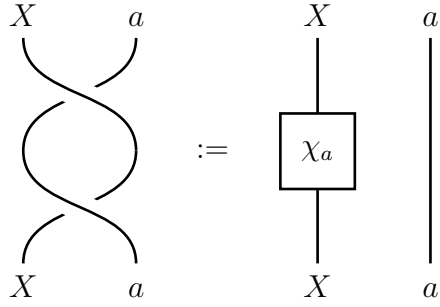

The following result is essentially the same as \cite[Lemma 8.22.9]{EGNO}.  Note that the result does not use semisimplicity or finiteness. 

\begin{proposition}\label{prop:chi-properties}
Let $\cB$ be a braided tensor category and $a\in \cB$ an invertible object. Then
\begin{itemize}
  %  \item[(0)] $\chi_a$ only depends on the isomorphism class of $a$.
    \item[(i)] $\chi_a$ is a monoidal natural isomorphism.
    \item[(ii)] The map $\chi: \Inv(\cB)\to \Aut_\otimes(\Id_{\cB})$ is a group morphism.
    \item[(iii)] The kernel of $\chi$ is the group of transparent invertible objects.
   
\end{itemize}
\end{proposition}

\begin{corollary}\label{corol grading}
Let $\cB$ be a braided fusion category. Then $\cB$ is faithfully graded over the group $\widehat{\Inv(\cB)/\ker(\chi)}$ (linear characters over $\Inv(\cB)/\ker(\chi)$) as follows:

\begin{align*}
    \cB_{\gamma}=\{X\in \cB: \chi_a(X)=\gamma(a)\id_X, \quad \forall a\in \Inv(\cB)\}, && \gamma \in \widehat{\Inv(\cB)/\ker(\chi)}.
\end{align*}
\end{corollary}
\begin{proof}
Since the induced map $\chi:\Inv(\cB)/\ker(\chi)\to \Aut_\ot(\Id_\cB)$ is injective, then the induced surjective map $\chi^*:\widehat{\Aut_\ot(\Id_\cB)}\to \widehat{\Inv(\cB)/\ker(\chi)}$ defines a faithful $\widehat{\Inv(\cB)/\ker(\chi)}$ grading by Example \ref{universal abelian grading}.
\end{proof}

\subsection{Conventions and Graphical calculus}
 In the following sections $\cC$ is a fusion category, which we may assume is strict without loss of generality by MacLane's Strictness and Coherence Theorems. In particular we can ignore associators and draw diagrams modulo isotopy that preserves the order of objects, for example
$$\begin{tikzpicture}[line width=1,scale=.66,baseline=35]
\draw (0,0)--(0,4);
\draw (2,0)--(2,1) \br (1,3) --(1,4);
\draw (3,0) --(3,4);
\end{tikzpicture} 
\hspace{10pt}
=
\hspace{10pt}
\begin{tikzpicture}[line width=1,scale=.66,baseline=35]
\foreach \x in {0,1,2}
\draw (\x,0)--(\x,4);
\end{tikzpicture}
\hspace{10pt}
=
\hspace{10pt}
\begin{tikzpicture}[line width=1,scale=.66,baseline=35]
\begin{scope}[xscale=-1]
\draw (0,0)--(0,4);
\draw (2,0)--(2,1) \br (1,3) --(1,4);
\draw (3,0) --(3,4);
\end{scope}
\end{tikzpicture}
\hspace{5pt}.$$
Our diagrams are oriented top to bottom and left to right. Our convention for braiding diagrams is that for positive braids the $i+1$st strand passes over the $i$th strand.

\section{Associative zestings}

Associative zesting may by regarded as a special case of $G$-graded extension \cite{ENO3}: given a $G$-graded fusion category $\cC$ we construct new $G$-graded fusion categories by twisting the fusion rules on the graded components of $\cC$ by an invertible object in the relative centralizer of $\cC$.  While explicitly constructing all $G$-graded extensions of a given category can be a formidable task, associative zesting takes a particular extension as input. This allows for a precise description of the new fusion categories and simplifies the subsequent analysis of braiding and pivotal structures in terms similar to the obstruction/parameterization approach of \emph{loc. cit.}.

\subsection{Relative centralizer of monoidal subcategories} 
\label{sec:relcen}
Let $\cC$ be a fusion category and $\cD\subset \cC$ a fusion subcategory. The relative centralizer $R_\cD(\cC)$ is the fusion subcategory of the Drinfeld center $\cZ(\cC)$ whose objects are pairs $(X,\sigma_{X,-})$ where $X\in \cD$ and $$\sigma_{X,-}=\{\sigma_{X,V}:X \otimes V \to V\otimes X\}_{V\in \cC}$$ is a family of isomorphisms natural in $V\in \cC$ such that 
\begin{align}
\label{relcen}
\sigma_{X,V\otimes W}= (\id_V \otimes \sigma_{X,W} )\circ(\sigma_{X,V}\otimes \id_W).
\end{align}
for all $V, W\in \cC$.

While $\cC$ is not a priori braided, since $R_{\cD}(\cC)$ is a subcategory of the braided Drinfeld center $\cZ(\cC)$ we can use crossings in our graphical calculus when objects of the relative center are involved. For example, Equation \ref{relcen} becomes
$$
\begin{tikzpicture}[line width=1,scale=.66,baseline=40]

\draw (2,0) node[below] {$X$} \br (0,4)node[above] {$X$};
\draw[line width=10, white] (0,0)  \br (1,4);
\draw[line width=10, white] (1,0)  \br (2,4);
\draw (0,0) node[below] {$V$} \br (1,4)node[above] {$V$};
\draw (1,0) node[below] {$W$}\br (2,4)node[above] {$W$};
\end{tikzpicture}
\hspace{10pt}
=
\hspace{10pt}
\begin{tikzpicture}[line width=1,scale=.66,baseline=40]
\draw (2,0) node[below] {$X$} \br (1,2) \br (0,4)node[above] {$X$};
\draw[line width=10, white] (0,0) -- (0,2) \br (1,4);
\draw[line width=10, white] (1,0)\br (2,2) -- (2,4);
\draw (0,0) node[below] {$V$} --(0,2)  \br (1,4) node[above] {$V$};
\draw (1,0) node[below] {$W$}  \br (2,2)-- (2,4) node[above] {$W$};
\end{tikzpicture}. $$

A morphism $f:(X,\sigma_{X,-})\to (Y,\sigma_{Y,-})$ in $R_\cD(\cC)$ is a morphism $f:X\to Y$ in $\cD$ such that $(\id_V\otimes f)\circ \sigma_{X,V}=\sigma_{Y,V}\circ (f\otimes \id_V)$ for all $V\in \cC$. In pictures, 

$$
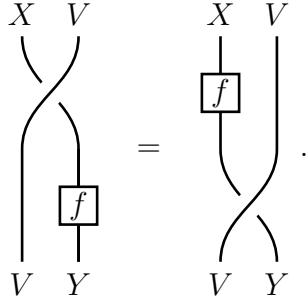
\begin{tikzpicture}[line width=1,scale=.75,baseline=40]
\draw (1,0)--(1,2) \br (0,4) node[above] {$X$};
\draw[line width=10, white] (0,0)--(0,2) \br (1,4);
\draw (0,0) node[below] {$V$} -- (0,2) \br (1,4) node[above] {$V$};
\begin{scope}[xshift=1cm]
\onebox{Y}{}{f}
\end{scope}
\end{tikzpicture} 
\hspace{10pt}
=
\hspace{10pt}
\begin{tikzpicture}[line width=1,scale=.75,baseline=40]
\draw (1,0) node[below] {$Y$} \br (0,2); 
\draw [white, line width=10](0,0) \br (1,2) --(1,4);
\draw (0,0) node[below] {$V$}\br(1,2)  -- (1,4) node[above] {$V$};
\begin{scope}[yshift=2cm]
\onebox{}{X}{f}
\end{scope}
\end{tikzpicture}.
$$

The isomorphism $\sigma$ will be called the relative half braiding. The category  $R_\cD(\cC)$ is monoidal with tensor product given by \[(X,\sigma_{X,-})\otimes (Y,\sigma_{Y,-})=(X\otimes Y, (\sigma_{X,-}\otimes \id)\circ(\id \otimes \sigma_{Y,-})),\] and unit object $(\unit,\id)$.

\begin{remark}
The notion of relative center was defined in \cite{GNN}. This concept is closely related to the one of relative centralizer introduced above. If $\cC$ is a fusion category and $\cD\subset \cC$ a fusion subcategory, then the relative centralizer $R_\cD(\cC)$ is a full fusion subcategory of the relative center $\cZ_{\cD}(\cC)$ (see \cite[Definition 2.1]{GNN} for the precise definition). The relative centralizer $R_\cD(\cC)$ is also a full fusion subcategory of the Drinfeld centers $\cZ(\cD)$ and $\cZ(\cC)$. In particular, $R_{\cD}(\cC)$ is braided. In the case that $\cC=\cD$, the fusion category $R_\cC(\cC)$ coincides with the Drinfeld center $\cZ(\cC)$. 
\end{remark}

%Let $G$ be a finite group and $\cC=\oplus_{g\in G}\cC_g$ be a faithfully $G$-graded fusion category. Since $R_{\cC_e}(\cC) \subset \cZ(\cC)$, we can use crossings in our graphical calculus when objects of $R_{\cC_e}(\cC)$ are involved. 

\subsection{Associative zesting}\label{sec:associative zesting}

\begin{definition}
\label{def:assoczesting}
Let $G$ be a group and $\cC=\oplus_{g\in G}\cC_g$ be a faithfully $G$-graded fusion category.

An \emph{associative $G$-zesting} $\lambda$ for $\cC$ consists of the following data:
\end{definition}

\begin{enumerate}
    \item A map \[\lambda: G\times G \to  \left(R_{\cC_e}(\cC)\right)_{pt}, \quad (g,h)\mapsto \lambda(g,h)\] where $\lambda(g,h)$ is simple.
    \item For each $(g_1,g_2,g_2) \in G^{\times 3}$ an isomorphism
    \[\lambda_{g_1,g_2,g_3}:\lambda(g_1,g_2)\otimes \lambda(g_1g_2,g_3)\to \lambda(g_2,g_3)\otimes \lambda(g_1,g_2g_3)\]
  which we represent graphically by
    
    
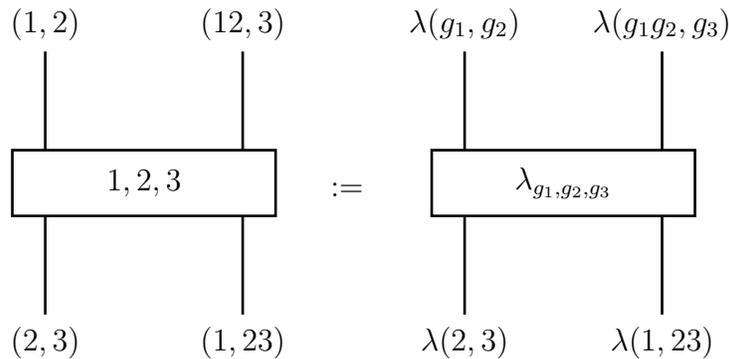
\begin{align*}
  \begin{tikzpicture}[line width=1, scale=1.75,baseline=45]
  \foreach \x in {0,1.5}{
  \draw(\x,0)--(\x,2);
  }
  \draw[fill=white] (-.25,.75) rectangle (1.75,1.25); 
  \draw (0,0) node[below] {$(2,3)$};
  \draw (1.5,0) node[below] {$(1,23)$};
  \draw (0,2) node[above] {$(1,2)$};
  \draw (1.5,2) node[above] {$(12,3)$};
  \draw (.75,1) node {$1,2,3$};
  \end{tikzpicture}
  &\hspace{10pt}:=\hspace{10pt}
  \begin{tikzpicture}[line width=1, scale=1.75,baseline=45]
  \foreach \x in {0,1.5}{
  \draw(\x,0)--(\x,2);
  }
  \draw[fill=white] (-.25,.75) rectangle (1.75,1.25); 
  \draw (0,0) node[below] {$\lambda(g_2,g_3)$};
  \draw (1.5,0) node[below] {$\lambda(g_1,g_2g_3)$};
  \draw (0,2) node[above] {$\lambda(g_1,g_2)$};
  \draw (1.5,2) node[above] {$\lambda(g_1g_2,g_3)$};
  \draw (.75,1) node {$\lambda_{g_1,g_2,g_3}$};
  \end{tikzpicture}
  \end{align*}
   such that for any $(g_1,g_2,g_3,g_4)\in G^{\times 4}$ the equation in Figure \ref{axiom:asso} holds (see Remark \ref{rmk: notation}(1) for notation conventions). 
    
 \end{enumerate}

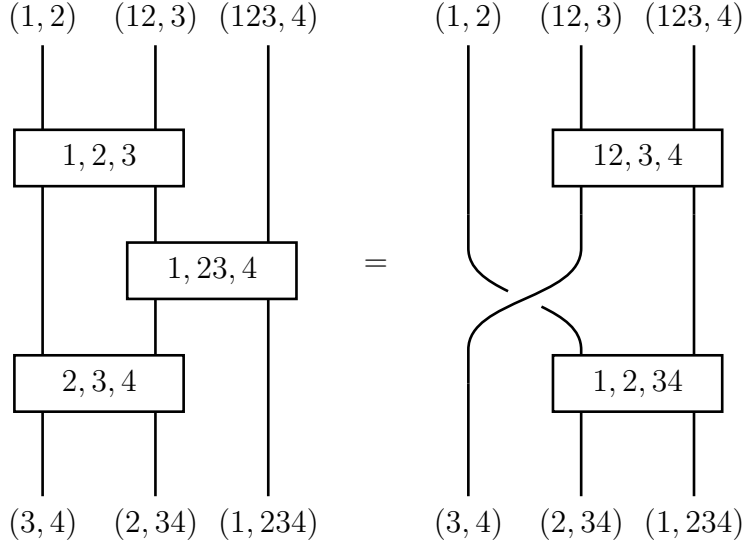
\begin{figure}[h]
    $$\begin{tikzpicture}[line width=1,scale=1.5,baseline=3cm]
  \foreach \x in {0,1,2}{
  \draw(\x,0)--(\x,4);
  }
  \draw[fill=white] (-.25,.75) rectangle node {$2,3,4$} (1.25,1.25) ;
  \draw[fill=white] (.75,1.75) rectangle node {$1,23,4$} (2.25,2.25);
  \draw[fill=white] (-.25,2.75) rectangle node {$1,2,3$} (1.25,3.25) ;
  \draw (0,0) node[below] {$(3,4)$};
   \draw (1,0) node[below] {$(2,34)$};
    \draw (2,0) node[below] {$(1,234)$};
    \draw (0,4) node[above] {$(1,2)$};
      \draw (1,4) node[above] {$(12,3)$};    
      \draw (2,4) node[above] {$(123,4)$};
  \end{tikzpicture} \hspace{10pt}
  =
  \hspace{10pt}
  \begin{tikzpicture}[line width=1,scale=1.5, baseline=3cm]
   \foreach \x in {0,1,2}{
  \draw(\x,0)--(\x,1);
  \draw(\x,2.5)--(\x,4);
  }
   \begin{scope}[xshift=-1cm,yshift=2.5cm]
  \braid[number of strands=3] a_1^{-1};
  \end{scope}
   \draw[fill=white] (.75,.75) rectangle node {$1,2,34$} (2.25,1.25);
   \draw[fill=white] (.75,2.75) rectangle node {$12,3,4$} (2.25,3.25);
  \draw (0,0) node[below] {$(3,4)$};
   \draw (1,0) node[below] {$(2,34)$};
    \draw (2,0) node[below] {$(1,234)$};
    \draw (0,4) node[above] {$(1,2)$};
      \draw (1,4) node[above] {$(12,3)$};    
      \draw (2,4) node[above] {$(123,4)$};
  \end{tikzpicture}
  $$
   \caption{Associative zesting constraint}
   \label{axiom:asso}
\end{figure}

Moreover, we impose the following normalization conditions:

\begin{align}
\lambda(e,g_1)&=\lambda(g_1,e)=\unit, \label{unit1}\\  
\lambda_{g_1,e,g_2}&=\id_{\lambda(g_1,g_2)}.\label{unit2}
%\\ \lambda(e,g_1,g_2)&=c_{\unit,\lambda(g_1,g_2)},\label{unit3}\\
%\lambda(g_1,e,g_2)&=c_{\lambda(g_1,g_2),\unit} \label{unit3}.
\end{align}

 \begin{remark}\label{rmk: notation}
 \begin{enumerate}
     \item The label $1,2,3$ of the box on the right-hand side of Definition \ref{def:assoczesting}(2) has enough information to recover the target and source of the isomorphism $\lambda_{g_1,g_2,g_3}$, so we suppress the labels on the strands.  For further notational convenience we identify $g_i$ with the index $i$ so that, for example $g_1g_2$ becomes $12$ in Figure \ref{axiom:asso} and in subsequent figures. 
  
     \item Condition (2) of Definition \ref{def:assoczesting} implies that $\lambda$ in (1) is a 2-cocycle in the sense that $(g_1,g_2)\mapsto[\lambda(g_1,g_2)]$ satisfies the $2$-cocycle condition, where we interpret $[\lambda(g_1,g_2)]$ as an element of the group of isomorphism classes of invertible objects.
     \item It follows from the assumptions (\ref{unit1}),(\ref{unit2}), and the associative zesting condition that we also have $\lambda_{e,g_1,g_2}=c_{\unit,\lambda(g_1,g_2)}=\id_{\lambda(g_1,g_2)}$, and $\lambda_{g_1,e,g_2}=c_{\lambda(g_1,g_2),\unit}=\id_{\lambda(g_1,g_2)}$.
     
 \end{enumerate}
 
 \end{remark}

\begin{proposition}\label{prop:zesting-cat}
Let $G$ be a finite group and $\cC$ a faithfully $G$-graded fusion category. Given an associative $G$-zesting  $\lambda$,  we can define a new faithfully $G$-graded fusion category $\cC^\lambda:=(\cC,\otlam, \boldsymbol{a}^\lambda)$, where the tensor product $\otlam$ is defined as
\[V_{g_1}\,\otlam\, W_{g_2} :=  V_{g_1}\,\otimes W_{g_2}\, \otimes\, \lambda(g_1,g_2), \]   the associativity constraint $\boldsymbol{a}^\lambda_{V_{g_1}, W_{g_2}, Z_{g_3}}$ by 
   
   \begin{align}
     \hspace{-30pt}\assoc{1}{2}{3}{.85} &:=&
 \assocfull{g_1}{g_2}{g_3}{.85}
   \end{align}

and the same unit object and unit constraint as $\cC$.

\end{proposition}
\begin{proof}
An associative zesting  is a particular case of the construction of a faithfully graded fusion category given in \cite[Section 8]{ENO3}.  For the convenience of the reader, we will check the pentagon axiom. The pentagon axiom is equivalent to the equality in Figure \ref{fig:proof3}. 
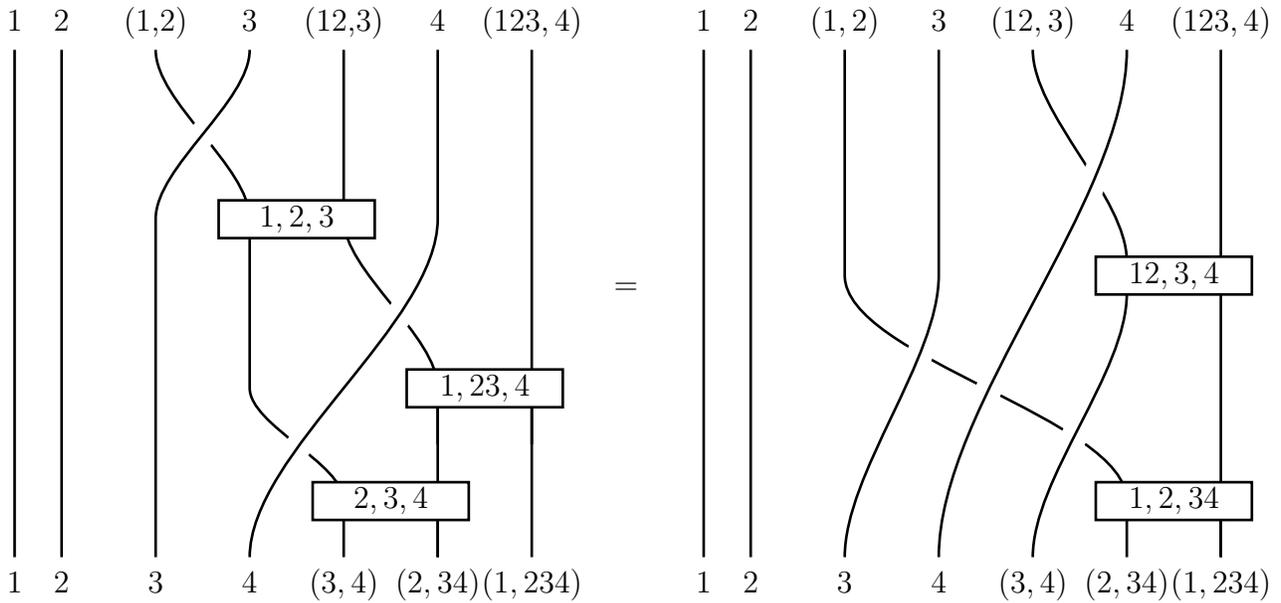
\begin{figure}[h]
  $$\hspace{-35pt}\begin{tikzpicture}[line width=1, looseness=.75, xscale=1.25,yscale=.75,baseline=3.5cm]
    \draw (.5,0) node[below] {$\parcen{1}$}--(.5,9)node[above] {$\parcen{1}$};
    \draw (1,0)node[below] {$\parcen{2}$}--(1,9) node[above] {$\parcen{2}$};
    \draw (6,0) node[below] {$(1,234)$}--(6,9) node[above] {$(123,4)$};
    \draw (5,0) --(5,3);
    \draw (5,3) \br (4,6)--(4,9) node[above] {(12,3)};
    \draw (4,1) \br (3,3) --(3,6) \br (2,9) node[above] {(1,2)};
    \draw[white, line width=10] (3,0)  \br (5,6) -- (5,9);
    \draw (3,0)node[below] {$\parcen{4}$}  \br (5,6) -- (5,9) node[above] {$\parcen{4}$};
     \begin{scope}[xshift=4cm]
     \twoboxnotop{(3,4)}{(2,34)}{2,3,4}
 \end{scope}
    \begin{scope}[xshift=5cm, yshift=2cm]
     \twoboxnotop{}{}{1,23,4}
 \end{scope}
 \begin{scope}[xshift=3cm,yshift=5cm]
     \twoboxnostrands{1,2,3}
 \end{scope}
    \draw[white, line width=10] (2,6) \br (3,9);
    \draw (2,0)node[below] {$\parcen{3}$} --(2,6) \br (3,9) node[above] {$\parcen{3}$};
    \end{tikzpicture}\hspace{5pt}=\hspace{10pt}
    \begin{tikzpicture}[line width=1, looseness=.75, xscale=1.25,yscale=.75,baseline=3.5cm]
   \draw (.5,0) node[below] {$\parcen{1}$}--(.5,9) node[above] {$\parcen{1}$};
    \draw (1,0)node[below] {$\parcen{2}$}--(1,9) node[above] {$\parcen{2}$};
    \draw (6,0) node[below] {$(1,234)$}--(6,9) node[above] {$(123,4)$};
    \draw (5,1) \br (2,5)--(2,9) node[above] {$(1,2)$};
    \draw (5,5.25) \br (4,9) node[above] {$(12,3)$};
  \draw[white, line width=10] (2,0) \br (3,5);
  \draw (2,0) node[below] {$\parcen{3}$}\br (3,5)--(3,9) node[above] {$\parcen{3}$};
  \draw[white, line width=10] (3,0)   \br (5,9);
  \draw (3,0) node[below] {$\parcen{4}$} \br (5,9) node[above] {$\parcen{4}$};
  \draw[white, line width=10] (4,0) \br (5,4.675);
  \draw (4,0) node[below] {$(3,4)$} \br (5,4.675);
   \begin{scope}[xshift=5cm]
     \twoboxnotop{(2,34)}{}{1,2,34}
    \begin{scope}[yshift=4cm]
     \twoboxnostrands{12,3,4}
 \end{scope}
  \end{scope}
    \end{tikzpicture}
    $$
   \caption{Pentagon axiom that must be satisfied by the zested associators.}
    \label{fig:proof3}
\end{figure}

Using the graphical calculus is easy to check that the equality depicted in Figure \ref{fig:proof3} is equivalent to the one in Figure \ref{fig:proof1}.  Now, the associative zesting condition of Figure \ref{axiom:asso} implies the equality in Figure \ref{fig:proof1} and therefore the pentagon axiom in the zested category. 
 
Finally, the fact that $\unit$ is the unit object with the same unit constraints follows directly from the definition of the tensor product of $\cC^\lambda$ and the conditions \eqref{unit1} and \eqref{unit2}.

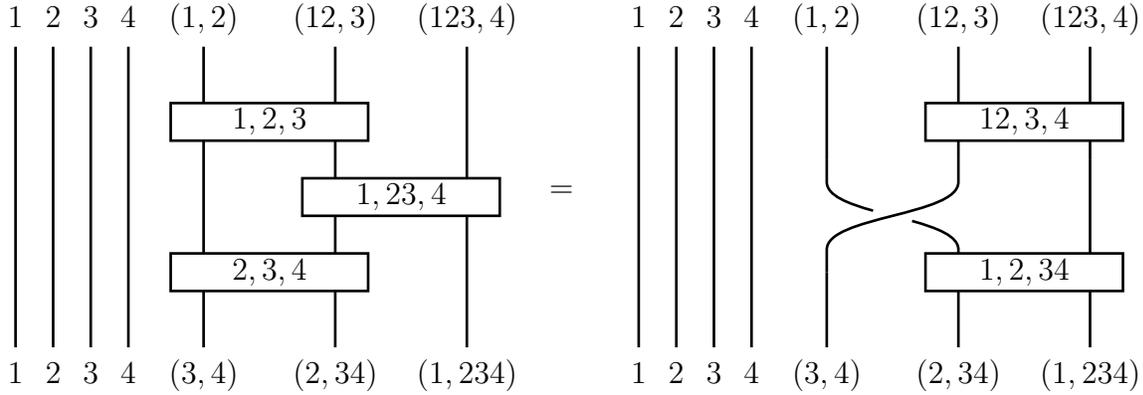
\begin{figure}[h]
 $$\hspace{-25pt}\begin{tikzpicture}[line width=1,scale=1.0,baseline=2cm]
    \begin{scope}[xscale=1.75]
  \foreach \x in {0,1,2}{
  \draw(\x,0)--(\x,4);
  }
  \draw[fill=white] (-.25,.75) rectangle node {$2,3,4$} (1.25,1.25) ;
  \draw[fill=white] (.75,1.75) rectangle node {$1,23,4$} (2.25,2.25);
  \draw[fill=white] (-.25,2.75) rectangle node {$1,2,3$} (1.25,3.25) ;
  % Labels
  \draw (0,4) node [above] {$(1,2)$};
  \draw (1,4) node [above] {$(12,3)$};
  \draw (2,4) node [above] {$(123,4)$};
   \draw (0,0) node [below] {$(3,4)$};
  \draw (1,0) node [below] {$(2,34)$};
  \draw (2,0) node [below] {$(1,234)$};
  \end{scope}
  % Strands, Labels
  \foreach \z in {1,2,3,4}{
  \draw (1/2*\z-3,0) node[below] {$\phantom{(}\z\phantom{)}$}--(1/2*\z-3,4) node[above] {$\phantom{(}\z\phantom{)}$};
  }
  \end{tikzpicture} \hspace{5pt}
  =
  \hspace{10pt}
  \begin{tikzpicture}[line width=1,baseline=2cm]
  \begin{scope}[xscale=1.75]
   \foreach \x in {0,1,2}{
  \draw(\x,0)--(\x,1);
  \draw(\x,2.5)--(\x,4);
  }
   \begin{scope}[xshift=-1cm,yshift=2.5cm]
  \braid[number of strands=3] a_1^{-1};
  \end{scope}
   \draw[fill=white] (.75,.75) rectangle node {$1,2,34$} (2.25,1.25);
   \draw[fill=white] (.75,2.75) rectangle node {$12,3,4$} (2.25,3.25);

  \draw (0,4) node [above] {$(1,2)$};
  \draw (1,4) node [above] {$(12,3)$};
  \draw (2,4) node [above] {$(123,4)$};
   \draw (0,0) node [below] {$(3,4)$};
  \draw (1,0) node [below] {$(2,34)$};
  \draw (2,0) node [below] {$(1,234)$};
\end{scope}
  % Strand Labels
  \foreach \z in {1,2,3,4}{
  \draw (1/2*\z-3,0) node[below] {$\phantom{(}\z\phantom{)}$}--(1/2*\z-3,4) node[above] {$\phantom{(}\z\phantom{)}$};
  }
  \end{tikzpicture}$$

    \caption{Equivalent formulation of the pentagon axiom from Figure 3.}
 \label{fig:proof1}
\end{figure}

    \begin{figure}[h]
    $$\hspace{20pt}\begin{tikzpicture}[line width=1, looseness=.75, xscale=1.75]
    \draw (0,0) node[below] {1} -- (0,3) node[above] {$\phantom{(}1\phantom{)}$};
    \draw (1,0) node[below] {2} -- (1,3) node[above] {$\phantom{(}2\phantom{)}$};
    \draw (6,0) node[below] {$(123,4)$} -- (6,3) node[above] {$(123,4)$};
    % Understands
    \draw (4,0)node[below] {$(1,2)$} to [out=90,in=-90] (2,3) node[above] {$(1,2)$};
    \draw (5,0) node[below] {$(12,3)$}to [out=90,in=-90] (4,3) node[above] {$(12,3)$};
    \draw[white,line width=10] (2,0) to [out=90,in=-90] (3,3);
    \draw (2,0) node[below] {3} to [out=90,in=-90] (3,3) node[above] {$\phantom{(}3\phantom{)}$};
    \draw[white,line width=10] (3,0) to [out=90,in=-90] (5,3);
    \draw (3,0) node[below] {4} to [out=90,in=-90] (5,3) node[above] {$\phantom{(}4\phantom{)}$};
    \end{tikzpicture}$$
    \label{fig:proof2}
\end{figure}

\end{proof}

\begin{remark}\label{rmk:solutions pentagon} 

    \item Given a map $\lambda: G\times G \to \left(R_{\cC_e}(\cC)\right)_{pt}$ there are at least three associated bifunctors, namely
    \begin{align*}
        V_g \,\ot^1\, W_h &:=  V_g\,\otimes\, W_h \,\otimes \,\lambda(g,h),\\  
        V_g \,\ot^2 \,W_h &:=  V_g\,\otimes \,\lambda(g,h)\,\otimes W_h \\
        V_g \,\ot^3 \,W_h &:=  \lambda(g,h)\,\otimes\, V_g\,\otimes \,W_h.
    \end{align*}

They are easily seen to be naturally isomorphic (proof supplied upon request), so our choice of $\ot^1$ is no loss of generality.

\end{remark}

\begin{example}
As a special case of our construction we can recover some examples found in \cite{KazWen} from the modular $\Z/N$-graded category $SU(N)_k$. In \emph{loc.~ cit.} they classify fusion categories with the same fusion rules as $SU(N)_k$, showing that any such category is obtained from $SU(N)_k$ by either changing the quantum parameter $q$ or twisting the associativity morphisms by a $3$-cocycle, or both. If we choose the \emph{trivial} $2$-cocycle $\lambda:\Z/N\times\Z/N\rightarrow \cC_{pt}\cap\cC_{0}$, i.e., $\lambda(a,b)=\unit$ then the second associative zesting constraint (Figure \ref{axiom:asso}) is simply the condition that $\lambda_{a,b,c}$ is a (normalized) $3$-cocycle on $\Z/N$.  Thus the associative zestings of $SU(N)_k$ with trivial $2$-cocycle are precisely the ones obtained in \cite{KazWen} by twisting the associativity morphisms.  We will study some cases with non-trivial $2$-cocycle below.
\end{example}

\subsection{Rigidity of associative zesting}\label{section: rigidity}

\begin{lemma}\label{Lemma: ridid pair}
Let $\cC$ be a fusion category and $X,Y \in \cC$  simple objects such that $\Hom(\unit, X\otimes Y)\neq 0$. If 
\begin{align*}
\phi: Y\otimes X\to \unit,&& \rho: \unit \to X\otimes Y    
\end{align*}
is a pair of non-zero morphisms, then the scalar $z(\phi,\rho)\in \ku$ defined by
\[
\begin{tikzcd}
X \arrow[d, "{\rho\otimes\id_X}"']\arrow[rr, "{z(\phi,\rho)\id_X}"]&& X\\ \arrow[rr, "{\alpha_{X,Y,X}}"]
(X\otimes Y)\otimes X & & X\otimes (Y \otimes X) \arrow[u, "{\id_X\otimes \phi}"']
\end{tikzcd}
\]is non-zero and the triple $(Y,\phi, z(\phi,\rho)\rho)$ is a dual of $X$. Moreover, given a non-zero map $\phi:Y\otimes X\to \unit$ the map $z(\phi,\rho)\rho$ does not depend in the choice of $\rho$, that is, if $\rho, \rho': \unit \to X\otimes Y$ are non-zero maps, then $z(\phi,\rho)\rho=z(\phi,\rho')\rho'$.
\end{lemma}
\begin{proof}
In a fusion category we have that, for simple objects $X $ and $Y$,  $\Hom(\unit, X\otimes Y)\neq 0$ if and only if $Y\cong X^*$. Moreover, since in that case, $\Hom(\unit, X\otimes Y)$ is one dimensional, there are non-zero scalars $c_1, c_2 \in \ku^\times$ such that $\epsilon_X=c_1\phi$ and $\delta_X=c_2\rho$ where $(\delta_X,\epsilon_X)$ defines a dual for $X$. Clearly $z(\phi,\rho)=c_1c_2$, thus $z(\phi,\rho)$ is non-zero. Moreover, we have that  $(\phi, z(\phi,\rho)\rho)= (c_1^{-1}\epsilon_X, c_1\delta_X)$. Hence  $(\phi,z(\phi,\rho) \rho)$ also defines a dual for $X$.

For the uniqueness, note  that $z(\phi,\rho)\rho=c_1c_2\rho=c_1\delta_X$. Since $c_1$ only depends on $\phi$ then $z(\phi,\rho)\rho$ only depends on $\phi$.
\end{proof}

Let $\cC$ be a faithfully $G$-graded fusion category and $\lambda$ an associative zesting. Given a simple object $X_g\in \cC_g$ we will denote by $X^*_g\in  \cC_{g^{-1}}$ the dual object with respect to the tensor product $\otimes$. Then $$\overline{X}_g:=X^*_g\otimes \lambda(g,g^{-1})^*\in \cC_{g^{-1}}^{\lambda}$$ is also a simple object and $ X^*_g\otimes X_g \cong  \overline{X}_g \otlam X_g.$ Hence 
$\Hom(\unit,X^*_g\otlam X_g)\neq 0,$ and we will use Lemma \ref{Lemma: ridid pair} to find specific formulas for the evaluation and coevaluation maps  of $X_g$ in $\mathcal{C}^{\lambda}$.

We are assuming that $\lambda(e,g)=\lambda(g,e)=\unit$ for all $g\in G$. Hence we have isomorphisms \[\lambda^g:=\lambda_{g,g^{-1},g}:\lambda(g,g^{-1})\to \lambda(g^{-1},g), \quad \quad g\in G.\]

As we notice in Example \ref{Ex: pointed fusion cat} any pointed fusion category has a spherical structure. Then for any invertible objects $a\in \cC$ in addition to the maps
\begin{align*}
    \epsilon_a:a^*\otimes a\to \unit, && \delta_a:\unit \to a\otimes a^*,
\end{align*}we have maps 
\begin{align}
    \epsilon_a':a\otimes a^*\to \unit, && \delta_a':\unit \to a^*\otimes a,
\end{align}such that 
\begin{equation}\label{defi d_g}
\dim(a)=\epsilon'_a\circ\delta_a= \epsilon_a\circ\delta_a'\in \ku^\times.
\end{equation}
Using these maps, we define $\phi_{X_g}: \overline{X}_g\otlam X_g \to \unit$ via the following  pictures,

%\[ \begin{tikzcd} \overline{X}_g\otlam X_g  \arrow[rr, "{\phi_{X_g}}"] \arrow[d, equal]&& \unit\\ X_g^*  \lambda(g)^*  X_g   \lambda(g^{-1}) \arrow[rd, "{\id \otimes d_{\lambda(g)^*,X_g}\otimes \id}"'] &&  \lambda(g)^*   \lambda(g) \arrow[u, "{\epsilon_{\lambda(g)}}"]\\ &X_g^*   X_g    \lambda(g)^*   \lambda(g^{-1}) \arrow[ru, "{\epsilon_{X_g}\otimes \id_{\lambda(g)}\otimes \lambda_g^{-1} }"']& \end{tikzcd} \]

%and  \[ \rho_{X_g} := \delta_ {X_g}\otimes \delta_{\lambda(g)}':\unit\to   X_g \otimes X_g^*\otimes \lambda(g)^{*}\otimes \lambda(g)=X_g\otlam \overline{X}_g \] (where $\lambda(g):=\lambda(g,g^{-1})$ and the tensor product among objects have been omitted). 

\begin{eqnarray} 
\phi_{X_g} & = \begin{tikzpicture}[scale=1.5, line width=1,baseline=-10]

\draw[looseness=1.5] (1,0) node[above] {$\lambda(g)^{*}$} to [out=-90, in=-90] (3,0)  node[above] {$\phantom{\lambda(g)^{*}}\lambda(g^{-1})\phantom{\lambda(g)^{*}}$};
\draw[white, line width=10, looseness=1.5] (0,0) to [out=-90, in=-90] (2,0);
\draw[looseness=1.5] (0,0) node[above] {$X_g^{*}$} to [out=-90, in=-90] (2,0) node[above] {$\phantom{\overline{X_g}}X_g\phantom{\overline{X_g}}$};

\draw[fill=white] (2.7,-.1) rectangle node {\scriptsize $\lambda^g$} (3.1, -.5);
\draw (3.1, -.2) node[right]  {$^{-1}$};
\end{tikzpicture} \\
\rho_{X_g} & = \begin{tikzpicture}[scale=1.5,line width=1,baseline=10]
\draw[looseness=2] (0,0) node[below] {$\phantom{X_g^{*}}X_g\phantom{X_g^{*}}$} to [out=90, in=90] (1,0) node[below] {$X_g^{*}$};
\draw[looseness=2] (2,0) node[below] {$\lambda(g)^{*}$} to [out=90, in=90] (3,0)  node[below] {$\phantom{\lambda(g)^{*}}\lambda(g^{-1})\phantom{\lambda(g)^{*}}$};
\end{tikzpicture}
\end{eqnarray}

We obtain the scalar  \[z(\phi_{X_g},\rho_{X_g})=\dim(\lambda(g,g^{-1}))^{-1}.\]
Hence for any $X_g\in \cC_g^\lambda$ (not necessarily simple) the data 
\begin{align}\label{dual for zesting}
(\overline{X}_g=X^*_g\otimes \lambda(g,g^{-1})^*,  \phi_{X_g},  \dim(\lambda(g,g^{-1}))^{-1}\rho_{X_g})
\end{align}define a dual in $\cC^\lambda$, where
%\commentp{I think it would be good to have the graphical interpretation for later when we present the Drinfeld isomorphism, etc. I included an idea of the graphical interpretation}
%\includegraphics[width=0.5\textwidth]{scans/zesting_ev_coev.png}

\begin{eqnarray}
\dim(\lambda(g,g^{-1})) & = \hspace{10pt} \begin{tikzpicture}[line width=1, baseline=0]  \draw (0,0) circle (.5);
\draw (.35,.35) node[right] {\scriptsize $\lambda(g,g^{-1})$};
\end{tikzpicture} =& \epsilon_{\lambda(g,g^{-1})}\circ \delta'_{\lambda(g,g)}.
\end{eqnarray}

\subsection{Obstruction to associative zestings}

Let $G$ be a finite group and $\cC$ a $G$-graded fusion category. Recall that $R_{\cC_e}(\cC)$ is a braided fusion category. We will denote by $B$ the abelian group $\operatorname{Inv} (R_{\cC_e}(\cC))$ of isomorphism classes of invertible objects in $R_{\cC_e}(\cC)$.  Recall (see Remark \ref{rmk: notation}(2)) that $\lambda(g,h)$ is a simple object in $R_{\cC_e}(\cC)$ and $(g,h)\mapsto [\lambda(g,h)]$ is a $2$-cocycle.

\begin{definition}\label{def: liftings and equivalent zestings}
\begin{enumerate}
    \item We will say that a 2-cocycle $\beta \in Z^2(G,B)$ has a \emph{lifting} if there is a $G$-zesting  $\lambda$ of $\cC$ such that $\beta(g,h)=[\lambda(g,h)]$ for all $g,h \in G$.
    \item Two liftings $\lambda$ and $\lambda'$ are called equivalent if there are isomorphisms $f_{g_1,g_2}:\lambda(g_1,g_2)\to \lambda'(g_1,g_2)$ such that \[(f_{g_2,g_3}\otimes f_{g_1,g_2g_3})\circ \lambda_{g_1,g_2,g_3}= \lambda'_{g_1,g_2,g_3}\circ (f_{g_1,g_2}\otimes f_{g_1g_2,g_3}),\] for all $g_1,g_2,g_3\in G$.
\end{enumerate}
\end{definition}
Note that inequivalent liftings may yield equivalent fusion categories.  Moreover, cohomologically distinct $2$-cocycles can even give equivalent fusion categories, as we will see in the examples in Section \ref{section: applications}.

%The following result will be useful to translate the existence of lifting to the case of pointed fusion categories.
%\begin{lemma}
%Let $\cB=\oplus_{a\in A} \cB_a$ be an $A$-graded braided tensor category. A 2-cocycle  $\beta \in Z^2(A, \Inv(\cB_e))$ has a lifting over $\cB$ if and only if  has a lifting over $\big (\cB_{pt} \big )_A$.
%\end{lemma}
%\qed

Since every invertible object $X$ is simple, we have that $\Aut_\cB(X)=\{c\id_X:c\in \ku^\times\}$. Hence, we can canonically identify $\Aut_\cB(X)$ with $\ku^\times$ for any invertible object.

Let $\beta \in Z^2(G,B)$. Take  $\lambda(g_1,g_2)\in R_{\cC_e}(\cC)$  such that the isomorphism class of $\lambda(g_1,g_2)$ is $\beta(g_1,g_2)$ and isomorphisms \[\lambda_{g_1,g_2,g_3}:\lambda(g_1,g_2)\otimes \lambda(g_1g_2,g_3)\to \lambda(g_2,g_3)\otimes \lambda(g_1,g_2g_3),\] for all $g_1,g_2,g_3 \in G$.

Define a map $\nu_\lambda:G^{\times 4}\to \ku^\times$, where $\nu_\lambda(g_1,g_2,g_3,g_4)\in \ku^\times$ is given by the automorphism of $\lambda(g_1,g_2)\otimes \lambda(g_1g_2,g_3)\otimes\lambda(g_1g_2g_3, g_4)$
defined in Figure \ref{fig:obstr}.
\begin{figure}[h]
    \begin{center}
    \begin{tikzpicture}[line width=1,xscale=1.75,yscale=.75]
    \draw (1,1.25) to [out=90, in=-90] (0,3.25)--(0,4);
    \draw[white, line width=10] (0,0) -- (0,1) to [out=90,in=-90] (1,2.75);
    \draw (0,0) node[below] {$(1,2)$} -- (0,1) to [out=90,in=-90] (1,2.75);
    \draw (0,6)--(0,8);
    \draw (2,1.25)--(2,10) node[above] {$(123,4)$};
    
    \begin{scope}[xshift=1cm]
    \twoboxnotop{(12,3)}{(123,4)}{12,3,4}
    \draw (1.5,1.25) node[] {$^{-1}$};
    \end{scope}
     \begin{scope}[xshift=1cm,yshift=2cm]
    \twoboxnobottom{}{}{1,2,34}
    \draw (1.5,1.25) node[] {$^{-1}$};
    \end{scope}
     \begin{scope}[yshift=4cm]
     \twobox{}{}{}{}{2,3,4}
     \end{scope}
       \begin{scope}[xshift=1cm,yshift=6cm]
     \twobox{}{}{}{}{1,23,4}
     \end{scope}
       \begin{scope}[yshift=8cm]
     \twobox{}{(1,2)}{}{(12,3)}{1,2,3}
     \end{scope}
    \end{tikzpicture}
        \end{center}

   \caption{4-cocycle obstruction}
   \label{fig:obstr}
\end{figure}
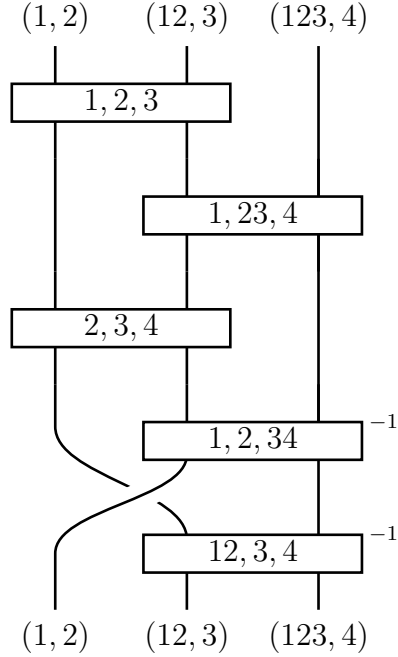

\begin{proposition}[\cite{ENO3}]\label{4-cocycle obstruction}
\begin{itemize}
    \item[(i)] $\nu_\lambda\in Z^4(G,\ku^\times)$.
    \item[(ii)] The cohomology class of $\nu_\lambda$ only depends on the cohomology class of $\beta$.
    \item[(iii)] The map $\nu$ induces a map $PW: H^2(G,\Inv(\cB))\to H^4(G,\ku^\times)$.
    \item[(iv)] The 2-cocycle $\beta$ admits a lifting if and only if $PW(\beta)=0$.
    \item[(v)] If $PW(\beta)=0$, the set of equivalence classes of liftings is a torsor over $H^3(G,\ku^\times)$.
\end{itemize}
\end{proposition}
\begin{proof}
The first three items of the proposition correspond to \cite[Proposition 8.10]{ENO3}. Item (iv) is consequence of \cite[Theorem 8.9]{ENO3} and item (v) is \cite[Proposition 8.15]{ENO3}
\end{proof}
\begin{remark}
If $G$ is cyclic then $H^4(G,\ku^\times)=0$ so any $\beta$ admits a lifting.
\end{remark}

%\begin{corollary}
%If a 2-cocycle takes values in a Tannakian subcategory of $\cB_e$ then it admits a canonical lifting.
%\end{corollary}

\section{Braided zesting}\label{section: braided zesting}

Recall that if $c$ is a braiding for a monoidal category $\cB$, then $c'_{X,Y}:=c_{Y,X}^{-1}$ is also a braiding for $\cB$. The category $\cB$ with the braiding $c'$ is denoted $\cB^{\operatorname{rev}}$. 

If $\cB$ is a braided monoidal category and $\cD\subset \cB$ is a monoidal subcategory, the functors 
\begin{align}
    F: \cD \to C_{\cD}(\cB), X&\mapsto (X, \{c_{X,V}\}_{V\in\cB}),\\
    G: \cD^{rev} \to C_{\cD}(\cB), X&\mapsto (X, \{c_{X,V}'\}_{V\in\cB})
\end{align}
are braided faithful functors.

\begin{definition}\label{def:braidedzesting}
Let $\cB$ be a braided faithfully $A$-graded fusion category, where $A$ is an abelian group.
 A \emph{braided zesting} consists of a triple $(\lambda, j, t)$, where
\begin{itemize}
    \item[(i)] $\lambda$ is an associative zesting  such that the relative half braiding of  $\lambda(a,b)$ is $\{c'_{\lambda(a,b),V} \}_{V\in \cB}$ for all $a,b \in A$.
    \item[(ii)] For any pair $g_1,g_2 \in A$ there is an isomorphism (see Figure \ref{fig:map-t}) $$t(g_1,g_2):\lambda(g_1,g_2)\to \lambda(g_2,g_1),$$ 
    \item[(iii)] A function $j: A\to \Aut_\otimes(\Id_{\cB})$, where $ \Aut_\otimes(\Id_{\cB})$ is the abelian group of all  tensor natural isomorphisms  of the identity.
\end{itemize}
The triple $(\lambda, j, t)$, must satisfy the following conditions:\\
\begin{itemize}
    \item[(BZ1)]\label{BZ1} For any $a, b \in A$,
    \[\omega(a,b):=\chi_{\lambda(a,b)}\circ j_{ab}\circ j_a^{-1}\circ j_b^{-1}\in \Aut_\ot^A(\Id_{\cB})\cong \widehat{A},\] where $\chi$ was defined in Figure \ref{fig:chi}. We will denote $\omega(a,b)(c):=\omega(a,b;c)\in\ku^\times$. 
    \item[(BZ2)]\label{BZ2}The equality in Figures \ref{fig:unnormalized-1} and \ref{fig:normalized-2} holds, for any $(g_1,g_2,g_3)\in A^3$ and objects $X_{g_1} \in \cB_{g_1}, Y_{g_2}\in \cB_{g_2}, Z_{g_3} \in \cB_{g_3}$.
\end{itemize}

Moreover, we impose the following normalization conditions:

\begin{itemize}
     \item[(a)] $t(e,g)=t(g,e)=\id_\unit$ (Recall that for a normalized associative zesting $\lambda(e,g)=\textbf{1}$.)
     \item[(b)] $j_e=\Id$ (The identity natural transformation.)
     \item[(c)] $j_g(\unit)=\id_\unit$ for all $g\in 
     A$.
 \end{itemize}
 \end{definition}
 
 \begin{figure}[h!]
    %\centering
  \begin{align*}\begin{tikzpicture}[line width=1,xscale=1.75,yscale=1.5, baseline=40]
    \draw(0,-.5) node[below] {$(2,1)$}--(0,0);
      \draw(0,2)--(0,2.5)node[above] {$(1,2)$};
    \onecirc{}{}{1,2}
    \end{tikzpicture}
    &\hspace{10pt}
    :=\hspace{10pt}
    \begin{tikzpicture}[line width=1,xscale=1.75,yscale=1.5, baseline=40]
    \draw(0,-.5) node[below] {$\lambda(g_2,g_1)$}--(0,0);
      \draw(0,2)--(0,2.5)node[above] {$\lambda(g_1,g_2)$};
    \onecirc{}{}{t(g_1,g_2)}
    \end{tikzpicture}
    \end{align*}
    \caption{Diagram for the isomorphisms $t(g_1,g_2)$.}
    \label{fig:map-t}
\end{figure}
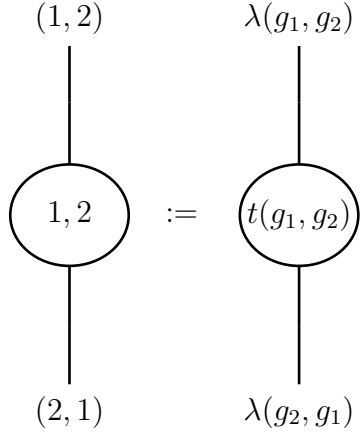

 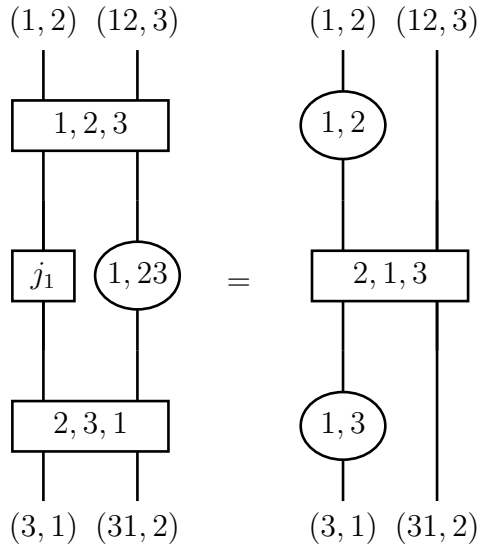
\begin{figure}[h!]
\begin{center}
\begin{tikzpicture}[line width=1,baseline=80,xscale=1.25]
\draw (0,2)--(0,4);
\twobox{(3,1)}{}{(31,2)}{}{2,3,1}
 \begin{scope}[yshift=2cm]
 \onebox{}{}{j_1}
 \begin{scope}[xshift=1cm]
 \onecirc{}{}{1,23}
 \end{scope}
 \end{scope}
  \begin{scope}[yshift=4cm]
  \twobox{}{(1,2)}{}{(12,3)}{1,2,3}
 \end{scope}
\end{tikzpicture}
\hspace{10pt}
=
\hspace{10pt}
\begin{tikzpicture}[line width=1,baseline=80,xscale=1.25]
\draw (1,0) node[below] {$(31,2)$}--(1,6) node[above] {$(12,3)$};
\onecirc{(3,1)}{}{1,3}
 \begin{scope}[yshift=2cm]
\twobox{}{}{}{}{2,1,3}
 \end{scope}
  \begin{scope}[yshift=4cm]
  \onecirc{}{(1,2)}{1,2}
 \end{scope}
\end{tikzpicture}
\end{center}
%\centering
  %\includegraphics[width=0.6\textwidth]{scans/normalizes-braid-zesting-condition-1.png}
   \caption{First  braided zesting condition.}
   \label{fig:unnormalized-1}
\end{figure}

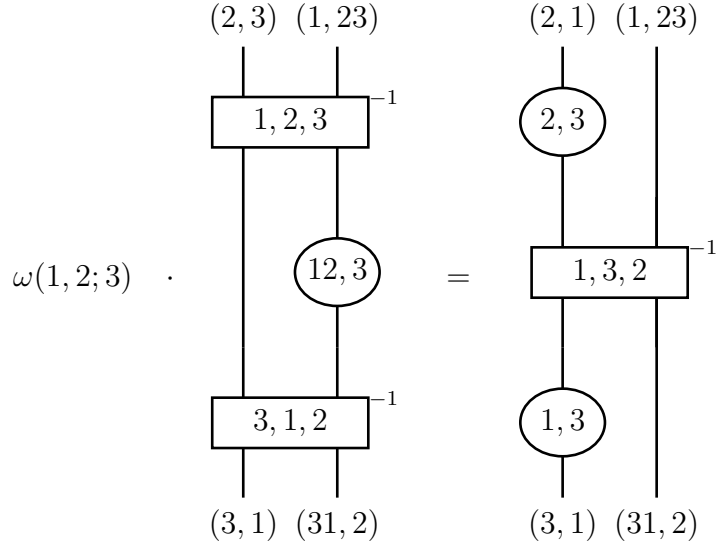
\begin{figure}[h!]
\begin{center}
\hspace{-2cm}
$\omega(1,2;3)$ \hspace{5pt} $\cdot$\hspace{5pt} 
\begin{tikzpicture}[line width=1,baseline=80,xscale=1.25]
\draw (0,2)--(0,4);
\twobox{(3,1)}{}{(31,2)}{}{3,1,2}
 \draw (1.5,1.25) node[] {$^{-1}$};
 \begin{scope}[xshift=1cm,yshift=2cm]
 \onecirc{}{}{12,3}
 \end{scope}
  \begin{scope}[yshift=4cm]
  \twobox{}{(2,3)}{}{(1,23)}{1,2,3}
 \draw (1.5,1.25) node[] {$^{-1}$};
 \end{scope}
\end{tikzpicture}
\hspace{10pt}
=
\hspace{10pt}
\begin{tikzpicture}[line width=1,baseline=80,xscale=1.25]
\draw (1,0) node[below] {$(31,2)$}--(1,6) node[above] {$(1,23)$};
\onecirc{(3,1)}{}{1,3}
 \begin{scope}[yshift=2cm]
\twobox{}{}{}{}{1,3,2}
 \draw (1.5,1.25) node[] {$^{-1}$};
 \end{scope}
  \begin{scope}[yshift=4cm]
  \onecirc{}{(2,1)}{2,3}
 \end{scope}
\end{tikzpicture}
\end{center}
\caption{Second braided zesting condition.}   \label{fig:normalized-2}
\end{figure}

\begin{remark}

\begin{enumerate}
    \item Condition (BZ1) is equivalent to the condition $$\chi_{\lambda(a,b)}|_{\cB_e}=( \frac{j_aj_b}{j_{ab}}\big )_{\cB_e}$$ for all $a, b \in \cB$. 
    \item The choice of the half braiding in braided zesting is compatible with the braidings used in the graphical calculus in the discussion of the relative centralizer from Section \ref{sec:relcen} and Definition \ref{def:assoczesting} of associative zesting.
    \item Since the two braided zesting conditions are isomorphisms of invertible simple objects, they can be expressed as scalar equations which take the form
\begin{eqnarray}
\label{eqn: braided zest scalar 1}
   j_1\left ( \lambda(g_2,g_3)\right )\,  \lambda_{g_1,g_2,g_3} \, t(g_1,g_2g_3) \, \lambda_{g_2,g_3,g_1} & = & t(g_1,g_2) \,\lambda_{g_2,g_1,g_3}\, t(g_1,g_3) \\  \label{eqn: braided zest scalar 2}
  \omega(1,2;3)\, \lambda_{g_1,g_2,g_3}^{-1}\, t(g_1g_2,g_3)\, \lambda_{g_3,g_1,g_2}^{-1} &=&  t(g_2,g_3)\, \lambda_{g_1,g_3,g_2}^{-1} \,t(g_1,g_3).
\end{eqnarray}
\end{enumerate}

\end{remark} 
\begin{lemma}
For a fixed $\lambda$ and $j$, both  $(\lambda,j,t)$ and $(\lambda,j,t^\prime)$ are braided $A$-zestings of $\cB$ if and only if $r(a,b):=\frac{t(a,b)}{t^\prime(a,b)}$ is a bicharacter on $A$.  In particular, such braided zestings form a torsor over the group of bicharacters of $A$.
\end{lemma}
\begin{proof}
From the form of equations (\ref{eqn: braided zest scalar 1}) and (\ref{eqn: braided zest scalar 2}) we see that $r(a,b+c)=r(a,b)r(a,c)$ and $r(a+b,c)=r(a,c)r(b,c)$.
\end{proof}

\begin{proposition}\label{prop: braided zesting construction}
Let $A$ be an abelian group and $\cB$ a faithfully $A$-graded braided fusion category. Given a braided zesting $(\lambda,j, t)$, the fusion category $\cB^\lambda$ defined in Proposition \ref{prop:zesting-cat} is braided with braiding $c^{(\lambda,j, t)}_{V_g,W_h}$ given by the natural isomorphisms defined 
in Figure \ref{fig:braiding}.
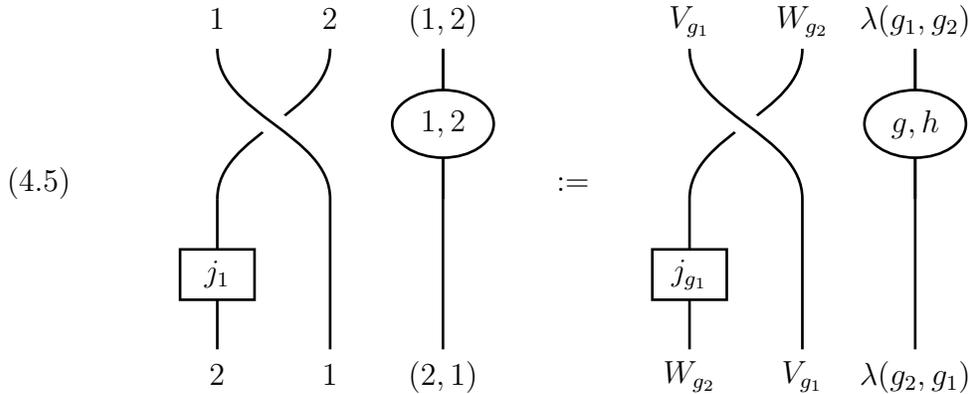
\begin{figure}[h]
\begin{align}
    \begin{tikzpicture}[line width=1,xscale=1.5,baseline=60]
        \onebox{\parcen{2}}{}{j_1}
    \draw (2,0) node[below] {$(2,1)$} --(2,4) node[above] {$(1,2)$};
    \draw (0,2) to [out=90, in=-90] (1,4)node[above] {$\parcen{2}$};
    \draw[white, line width=10] (1,0)  -- (1,2) to [out=90,in=-90] (0,4);
             \draw (1,0) node[below] {$\parcen{1}$} -- (1,2) to [out=90,in=-90] (0,4) node[above] {$\parcen{1}$};
    \begin{scope}[xshift=2cm,yshift=2cm]
    \onecirc{}{}{1,2}
    \end{scope}
    \end{tikzpicture}
&\hspace{20pt}:=\hspace{20pt}
  \begin{tikzpicture}[line width=1,xscale=1.5,baseline=60]
        \onebox{W_{g_2}}{}{j_{g_1}}
    \draw (2,0) node[below] {$\lambda(g_2,g_1)$} --(2,4) node[above] {$\lambda(g_1,g_2)$};
    \draw (0,2) to [out=90, in=-90] (1,4)node[above] {$W_{g_2}$};
    \draw[white, line width=10] (1,0)  -- (1,2) to [out=90,in=-90] (0,4);
             \draw (1,0) node[below] {$V_{g_1}$} -- (1,2) to [out=90,in=-90] (0,4) node[above] {$V_{g_1}$};
    \begin{scope}[xshift=2cm,yshift=2cm]
    \onecirc{}{}{g,h}
    \end{scope}
    \end{tikzpicture}
\end{align}
    \caption{Braiding for $\cB^\lambda$}
    \label{fig:braiding}
\end{figure}

\end{proposition}
\begin{proof}
The two hexagon equations that must be satisfied by the zested braiding and associators take the form of the equations depicted in Figures \ref{fig:hex-1-cat} and \ref{fig:hex-2-cat}.
That the equations in Figure \ref{fig:hex-1-cat} and \ref{fig:hex-2-cat} are equivalent to the hexagons can be readily checked in the graphical calculus using that the $j_g$'s are monoidal and by applying the definition of $\chi_a$ from Figure \ref{fig:chi}, respectively. These equations are then satisfied by the conditions found in Figures \ref{fig:unnormalized-1} and \ref{fig:normalized-2}.
\end{proof}

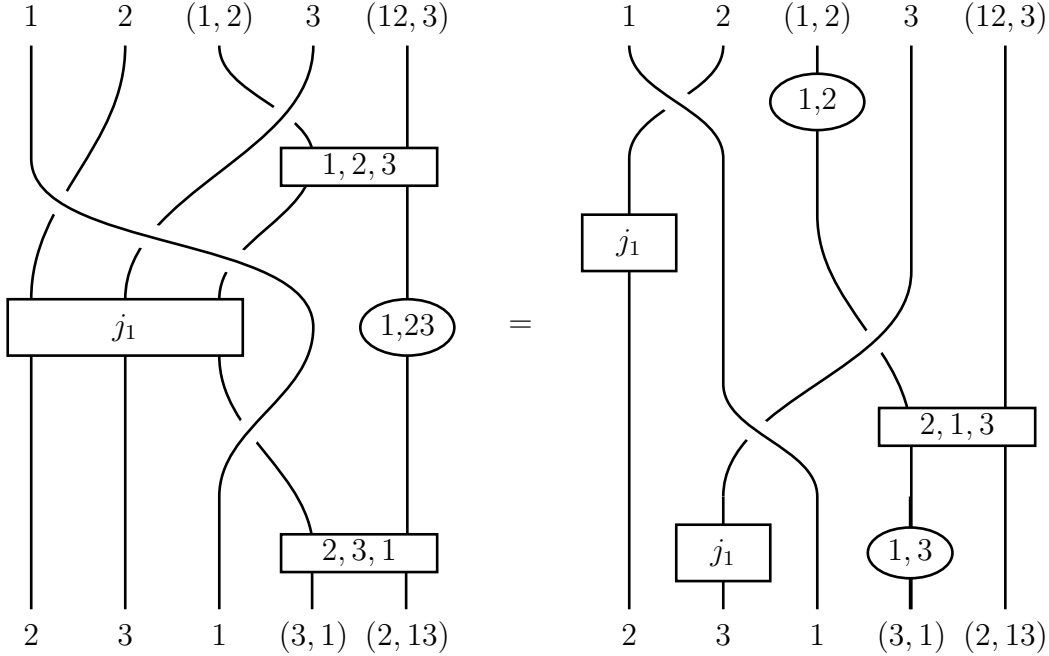
\begin{figure}[h]
$$\hspace{-20pt}\begin{tikzpicture}[line width=1,xscale=1.25,yscale=.75,baseline=105]
\draw (0,0) node[below] {$\parcen{2}$}--(0,4.5);
\draw (1,0) node[below] {$\parcen{3}$}--(1,4.5);
\draw (3,0) node[below] {$(3,1)$};
\draw (4,0) node[below] {$(2,13)$};
\draw (4,1) --(4,10) node[above] {$(12,3)$};
\draw (3,1) to [out=90,in=-90] (2,4.5);
\draw (0,5.5) to [out=90,in=-90] (1,10) node[above] {$\parcen{2}$};
\draw (2,5.5) to [out=90,in=-90] (3,8);
\draw (3,8) to [out=90,in=-90] (2,10) node[above] {$(1,2)$};
\draw[white,line width=10] (1,5.5) to [out=90,in=-90] (3,10);
\draw (1,5.5) to [out=90,in=-90] (3,10) node[above] {$\parcen{3}$};
\draw[white,line width=10] (2,0)--(2,2) to [out=90,in=-90] (3,5) to [out=90,in=-90] (0,8)--(0,10) ;
\draw (2,0) node[below] {$\parcen{1}$}--(2,2) to [out=90,in=-90] (3,5) to [out=90,in=-90] (0,8)--(0,10) node[above] {$\parcen{1}$};
\draw (-.25, 4.5) rectangle node {$j_1$} (2.25,5.5);
\draw[fill=white] (4,5) circle (.5) node {1,23} ;
\begin{scope}[xshift=85]
\twoboxnotop{}{}{2,3,1}
\begin{scope}[yshift=195]
\twoboxnostrands{1,2,3}
\end{scope}
\end{scope}
\end{tikzpicture} \hspace{15pt}=\hspace{15pt}\begin{tikzpicture}[line width=1,baseline=105,yscale=.75,xscale=1.25]
\draw (4,0)node[below] {$(2,13)$} --(4,10) node[above] {$(12,3)$};
\draw(0,0) node[below] {$\parcen{2}$}--(0,8) to [out=90,in=-90] (1,10)node[above] {$\parcen{2}$};
\draw(1,0)node[below] {$\parcen{3}$}--(1,2);
\draw(3,0)node[below] {$(3,1)$}--(3,3);
\draw[fill=white] (.5,.5) rectangle node {$j_1$} (1.5,1.5);
\draw (3,3) to [out=90,in=-90] (2,7) --(2,10) node[above] {$(1,2)$};
\draw[white, line width=10] (1,2) to [out=90,in=-90] (3,6)--(3,10);
\draw (1,2) to [out=90,in=-90] (3,6)--(3,10) node[above] {$\parcen{3}$};
\draw[fill=white] (-.5,6) rectangle node {$j_1$} (.5,7) ;
\begin{scope}[xshift=85]
\onecirc{}{}{1,3}
\begin{scope}[yshift=63.75]
\twoboxnostrands{2,1,3}
\end{scope}
\end{scope}
\draw[fill=white] (2,9) circle (.5) node[] {1,2};
\draw[white, line width=10] (2,0) --(2,2) to [out=90,in=-90] (1,4)--(1,8) to [out=90,in=-90] (0,10);
\draw (2,0) node[below] {$\parcen{1}$}--(2,2) to [out=90,in=-90] (1,4)--(1,8) to [out=90,in=-90] (0,10) node[above] {$\parcen{1}$};
\end{tikzpicture}
$$
    %\centering
    %\includegraphics[width=0.9\textwidth]{scans/hexagon-cat-1.png}
    \caption{First Hexagon condition}
    \label{fig:hex-1-cat}
\end{figure}
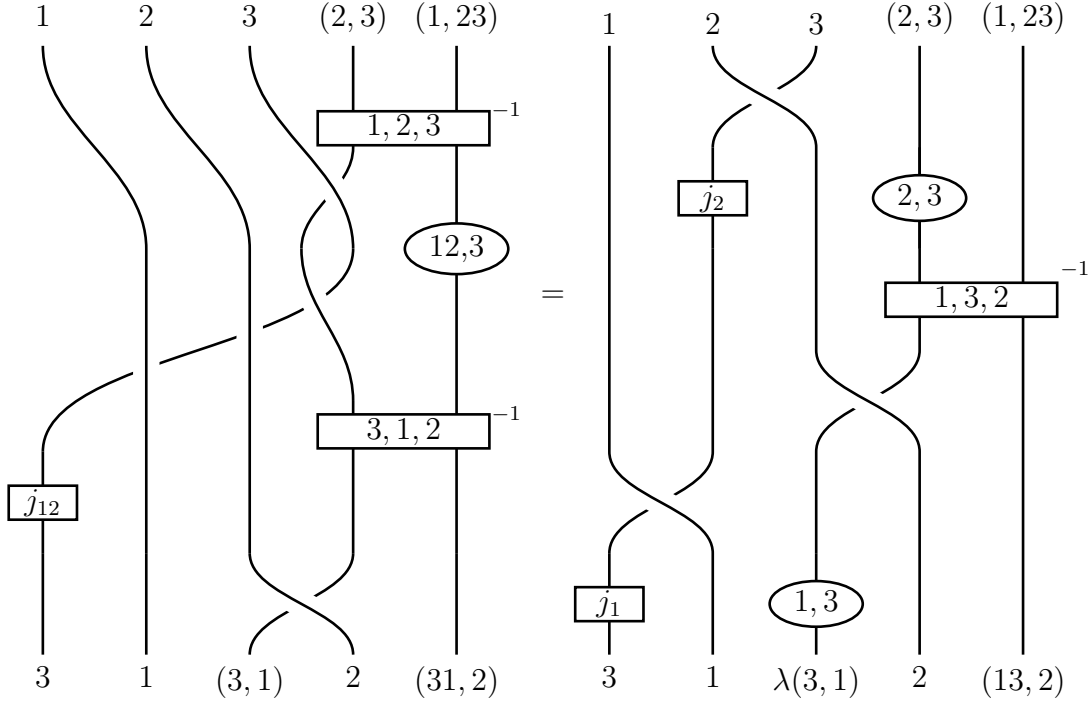
\begin{figure}[h]
$$\hspace{-15pt}\begin{tikzpicture}[line width=1,yscale=.675,xscale=1.375,baseline=95]
\draw (3,0) node[below] {}--(3,3);
\draw (2.5,6) \br (3,8)--(3,10) node[above] {$(2,3)$};
\draw (4,0) node[below] {}--(4,10) node[above] {$(1,23)$};
\draw[fill=white] (4,6) circle (.5) node {12,3};
\onebox{}{}{j_{12}}
\draw (0,2) \br (3,6);
\draw[white, line width=10] (3,6) \br (2,10);
\draw (3,6) \br (2,10)node[above] {$\parcen{3}$};
\draw[white, line width=10] (1,0) --(1,6) \br (0,10);
\draw (1,0) node[below] {} --(1,6) \br (0,10) node[above] {$\parcen{1}$};
\draw[white, line width=10] (2,0) --(2,6) \br (1,10);
\draw (2,0)node[below] {} --(2,6) \br (1,10) node[above] {$\parcen{2}$};
\begin{scope}[xshift=85,yshift=40]
\twoboxnostrands{3,1,2}
   \draw (1.5,1.25) node[] {$^{-1}$};
\end{scope}
\begin{scope}[xshift=85,yshift=210]
\twoboxnostrands{1,2,3}
   \draw (1.5,1.25) node[] {$^{-1}$};
\end{scope}
\draw[white, line width=10] (3,3) \br (2.5,6);
\draw (3,3) \br (2.5,6);
\foreach \x in {0,1,4}{
\draw (\x,-2) --(\x,0);}
\draw (2,-2) \br (3,0);
\draw[white, line width=10] (3,-2) \br (2,0);
\draw (3,-2) \br (2,0);
\draw (0,-2) node[below] {$3$};
\draw (1,-2) node[below] {$1$};
\draw (2,-2) node[below] {$(3,1)$};
\draw (3,-2) node[below] {$2$};
\draw (4,-2) node[below] {$(31,2)$};

\end{tikzpicture}
=
\begin{tikzpicture}[line width=1,yscale=.675,xscale=1.375,baseline=95]
\begin{scope}[yshift=-2cm]
\onebox{3}{}{j_1}
\end{scope}
\draw (0,0) \br (1,2) --(1,6);
\begin{scope}[xshift=1cm,yshift=6cm]
\onebox{}{}{j_2}
\end{scope}
\draw (1,8) \br (2,10) node[above] {$3$};
\draw[white, line width=10] (1,0) \br (0,2);
\draw (1,-2) node[below] {$1$}-- (1,0)\br (0,2)--(0,10) node[above] {$1$};
\begin{scope}[xshift=2cm,yshift=-2cm]
\onecirc{\lambda(3,1)}{}{1,3}
\end{scope}
\draw (2,0)--(2,2) \br (3,4)--(3,6);
\draw (4,-2) node[below] {$(13,2)$} --(4,10) node[above] {$(1,23)$} ;
\begin{scope}[xshift=3cm,yshift=4cm]
\twoboxnostrands{1,3,2}
\end{scope}
\draw[white, line width=10] (3,-2)--(3,2) \br (2,4) --(2,8) \br (1,10) ;
\draw (3,-2) node[below] {$2$}--(3,2) \br (2,4) --(2,8) \br (1,10)node[above] {$2$};
\draw (3,6)--(3,10) node[above] {$(2,3)$};
\begin{scope}[xshift=3cm,yshift=6cm]
\onecirc{}{}{2,3}
\end{scope}
 \draw (4.25,5.5) node[right] {$^{-1}$};
\end{tikzpicture}
$$
    %\centering
    %\includegraphics[width=1\textwidth]{scans/hexagon-cat-2.png}
    \caption{Second Hexagon condition}
   \label{fig:hex-2-cat}
\end{figure}

\begin{remark}
Notice that a braided zesting does not recover \emph{all} braidings that may exist on a given associative zesting.  For example with the trivial associative zesting (i.e. $\lambda(i,j)=\unit$ and $\lambda(i,j,k)=1$) on a braided fusion category we might not recover the reverse braiding, as in general this changes the braiding on the trivial component, while braided zesting does not. 
\end{remark}
\subsection{Equivalence of braided zestings}

Let $A$ be an abelian group and $\cB$ a faithfully $A$-graded fusion category. A bicharacter $\nu:A\times A\to \ku$ will be called alternating if $\nu(a,a)=1$ for all $a\in A$.
\begin{definition}\label{def: braided equiv. zestings}
\begin{itemize}
    \item[(i)] We will say that two braided zestings $(\lambda, j,t)$ and $(\lambda, j',t')$ are \emph{similar} if they define the same braiding on $\cB^\lambda$, that is \[c_{X_a,Y_b}^{(j,t)}=c_{X_a,Y_b}^{(j',t')},\] for all $a,b \in A, X_a\in \cB_a, Y_b\in \cB_b$.
    \item[(ii)] We will say that two braided zestings $(\lambda, j,t)$ and $(\lambda, j',t')$ are \emph{braided equivalent} if there an alternating character $\nu:A\times A\to \ku^\times$ such that $(\lambda, j,t)$ is similar to $(\lambda, j',\nu t')$.
\end{itemize}

\end{definition}
Note that again, braided inequivalent zestings may yield equivalent braided fusion categories--for example one could have a braided automorphism which permutes the simple objects. 
Let $\lambda$ be an associative $A$-zesting of a braided fusion category $\cB$. We define the abelian group

\begin{align}\label{H lambda}
H_\lambda=\left \{ (\kappa,l)\in C^1(A,\Aut_\ot (\Id_\cB))\times C^2(A,\ku^\times):    
\begin{aligned}
\delta(\kappa)_{a,b} \in \Aut_\ot^A(\Id_\cB)\\
\kappa(\lambda(a,b))=\frac{l(a,b)l(a,c)}{l(a,b+c)}\\
\delta(\kappa)_{a,b}(c)=\frac{l(a+b,c)}{l(a,c)l(b,c)}\\
\forall a,b, c\in A
\end{aligned}
\right \}
\end{align}
where $\delta$ is defined in equation \eqref{cochains G}, that is, $\delta(\kappa)_{a,b}=\kappa(b)-\kappa(a+b)-\kappa(a)$. We also define the abelian subgroups
\begin{align}
\label{subgroup H 1}
H_1&=\left \{(\kappa,l)\in H_\lambda: \kappa_a(X_b)=l(a,b)^{-1}, \forall a,b\in A, X_b\in \cB_b\right\}\\
H_2&=\left \{(\kappa,l)\in H_\lambda: \kappa_a=\
\id, \forall a\in A, \text{ and $l$ is an alternating bicharacter}\right\} \label{subgroup H 2}
\end{align}

\begin{proposition}\label{prop: equivalent zesting}
Let $\lambda$ be an associative $A$-zesting of a braided fusion category $\cB$.

\begin{itemize}
\item[(i)] The set of all braided $A$-zestings of the form $(\lambda,j,t)$ with $\lambda$ fixed is a torsor over the abelian group $H_\lambda$ defined in \eqref{H lambda}.

\item[(ii)] The set of all braided $A$-zestings similar to $(\lambda,j,t)$ is a torsor over the abelian group $H_1$ defined in \eqref{subgroup H 1}. 

\item[(iii)] The set of equivalence classes of braided $A$-zestings of the form $(\lambda,j,t)$ with $\lambda$ fixed, under the relation of being braided equivalent as in Definition \ref{def: braided equiv. zestings}(ii), is a torsor over $H_\lambda/H_1H_2$, where $H_2$ was defined in \eqref{subgroup H 2}.
\end{itemize}
\end{proposition}
\begin{proof}
Let $(\lambda, j,t)$ and $(\lambda,j',t')$ be braided $A$-zestings of $\cB$. Then $\kappa_a:=j_a'/j_a$ and $l(a,b)=t'(a,b)/t(a,b)$ for all $a,b\in A$ define an element in $(\kappa,l)\in H_\lambda$. In fact, condition (BZ1) implies that $\delta(\kappa)_{a,b}=\kappa_a\kappa_{ab}^{-1}\kappa_b\in \Aut_\ot^A(\Id_\cB)$ for all $a,b \in A$, condition \ref{BZ2}{(BZ2)} implies $\kappa_a(\lambda(b,c))=\frac{l(a,b)l(a,c)}{l(a,b+c)}$, and $\delta(\kappa)_{a,b}(c)=\frac{l(ab,c)}{l(a,c)l(b,c)}$   for all $a,b,c \in A$.

For item (ii), note that if $c_{X_a,Y_b}^{(j,t)}=c_{X_a,Y_b}^{(j',t')}$ then $\left [ j_a\circ j_a'^{-1}(X_b)\right]\ot \left[ t'(a,b)^{-1}\circ t(a,b)\right]= \id_{X_b\ot \lambda(a,b)}$ for all $a,b\in A, X_b\in \cB_b$. Then if $(\kappa, l)\in H_\lambda$ such that $j=j'\kappa,\, t=tl$, then $\kappa_a(X_b)=l(a,b)^{-1}$ for all $a,b\in A$ and $X_b\in \cB_b$.

Item (iii) follows immediately from (i) and (ii).
\end{proof}

\begin{corollary}\label{corol: triviality of j}
Let $\cB$ be an $A$-graded braided fusion category. A braided $A$-zesting $(\lambda, j, t)$ is similar to an  $A$-braided zesting of the form $(\lambda, \Id_{\cB}, t')$ if and only if

\begin{align}\label{eq: condition j trivial}
\chi_{\lambda(a,b)}|_{\cB_e}=\Id_{\cB_e}, && \quad j_a|_{\cB_e}=\Id_{\cB_e}
\end{align}
for all $a,b \in A$. In particular, if $A$ is the universal grading group, every braided zesting is similar to a braided zesting of the form $(\lambda, \Id_{\cB}, t)$. 
\end{corollary}
\begin{proof}
A braided $A$-zesting $(\lambda, j, t)$  is similar to one of the form $(\lambda, \id, t')$ if and only if $(j^{-1},l_j)\in H_1$, where $l_j(a,b)=j_a(X_b)^{-1}$ with $X_b\in \cB_b$. Hence \begin{align}\label{eq :j in A}
 j_a\in \Aut_\ot^A(\Id_\cB), && \forall a\in A,   
\end{align} and \eqref{eq :j in A} imply that $\chi_{\lambda(a,b)}\in \Aut_\ot^A(\Id_\cB)$ or equivalently $\chi_{\lambda(a,b)}|_{\cB_e}=\Id_{\cB_e}$. Conversely, if conditions \eqref{eq :j in A} holds, then it is is easy to see that $(j^{-1},l_j)\in H_1$.

Now, if $A$ is the universal grading group it follows from Proposition \ref{prop: iso id-g} and Proposition \ref{prop:chi-properties} that any  braided zesting satisfies the condition in \eqref{eq: condition j trivial}.
\end{proof}

\subsection{Obstructions to braided zestings}

%Braided zesting conditions correspond to Figures \ref{fig:unnormalized-1} and Figure \ref{fig:unnormalized-2}.  The second condition is that \[\omega(a,b):=\chi_{\lambda(a,b)}\circ j_{ab}\circ j_a^{-1}\circ j_b^{-1}\in \Aut_\ot^A(\Id_{\cB})\cong \widehat{A},\].  Notice that for arbitrary $j:G\rightarrow \Aut_\ot(\cB)$ the above definition gives $\omega(a,b)\in\Aut_\ot(\cB)$ since it is a composition of such.  The second condition is that $\omega(a,b)$ actually defines a character of $A$--in particular it is constant on graded components so that   $\omega(a,b)(c)=:\omega(a,b;c)\in\ku$ is well-defined.

\emph{Let $\cB$ be a braided fusion category. From now on, using Proposition \ref{prop: iso id-g} we will identify the abelian groups $\Aut_{\ot}^A(\Id_\cB)$ and $\widehat{A}$. Recall that in particular, $\Aut_{\ot}(\Id_\cB)\cong \widehat{U(\cB)}$}.  

\smallbreak
First, we will describe obstructions to the existence of a function $j: A\to \Aut_\otimes(\Id_{\cB})$ satisfying (BZ1).

Let $\cB$ be a braided fusion category graded by a finite abelian group $A$. Let $U(\cB)$ be the universal grading group of $\cB$ and $\pi_1:U(\cB)\to A$ the group epimorphism that defines the $A$-grading on $\cB$. By restriction of the $U(\cB)$-grading, the fusion subcategory $\cB_e$ is $\ker(\pi_1)$-graded with $(\cB_e)_e=\cB_{\operatorname{ad}}$. Then this grading defines a group epimorphism
\begin{equation}\label{eq: definition pi2}
\pi_2:U(\cB_e)\to \ker(\pi_1).
\end{equation}
 
\begin{proposition}\label{prop:condition extension monoidal auto identity}
A tensor natural isomorphism $j\in\Aut_\ot(\Id_{\cB_e})\cong \widehat{U(\cB_e)}$ has an extension to an element in $\Aut_\ot(\Id_\cB)\cong \widehat{U(\cB)}$ if and only if $j \in \Aut_\ot^{\ker(\pi_1)}(\Id_{\cB_e})$. The set of extensions of $j$ is a torsor over $\widehat{A}$.
\end{proposition}
\begin{proof}
We have the exact sequence of abelian groups 
\[0\to \ker(\pi_2)\to U(\cB_e)\overset{\pi_2}{\to} \ker(\pi_1)\to U(\cB) \overset{\pi_1}{\to} A\to 0\] and dualizing 
\[0\to \widehat{A} \overset{\pi_1^*}{\to} \widehat{U(\cB)} \to \widehat{\ker(\pi_1)} \overset{\pi_2^*}{\to} \widehat{U(\cB_e)} \to \widehat{\ker(\pi_2)}\to 0.\]
Hence the image of the restriction map  $\widehat{U(\cB)}\to \widehat{U(\cB_e)}$ is exactly $\widehat{\ker(\pi_1)}$, or equivalently all $\gamma \in U(\cB_e)$ such that $\gamma|_{\ker(\pi_2)}\equiv 1$.
\end{proof}

It follows from Proposition  \ref{prop:condition extension monoidal auto identity} that a \textbf{first partial obstruction} to the existence of a function $j: A\to \Aut_\otimes(\Id_{\cB})$ satisfying (BZ1) is that 
\begin{align}\label{eq:fisrt obstruction}
\chi_{\lambda(a,b)}|_{\cB_e}\in \Aut_{\ot}^{\ker(\pi_1)}(\Id_{\cB_e})\cong \widehat{\ker(\pi_1)}, && \forall a, b \in A.    
\end{align}Proposition \ref{prop: iso id-g} says that condition \eqref{eq:fisrt obstruction} is equivalent to 
\begin{align}
    \chi_{\lambda(a,b)}(X)=\id_X, && \text{ for all  $X\in (\cB_e)_g$ where $g\in \ker(\pi_2)\subset U(\cB_e)$,} 
\end{align}where 
$\pi_2$ was defined in \eqref{eq: definition pi2}.

It follows by condition (BZ1) that $$\chi_{\lambda(a,b)}|_{\cB_e}=\Big (j_a\circ j_b\circ j_{ab}^{-1}\Big)_{\cB_e},$$
where $j$'s are in $\Aut_\ot^{\ker(\pi_1)}(\Id_{\cB_e})$. Hence a \textbf{second partial obstruction} for the existence of $j$ is that the cohomology class of 
\begin{equation}\label{eq:second braided zesting obstruction}
    \chi_{\lambda(-,-)}|_{\cB_e}\in H^2(A,\Aut_\ot^{\ker(\pi_1)}(\Id_{\cB_e}))= H^2(A,\widehat{\ker(\pi_1)}),
\end{equation} must be trivial.

\begin{remark} When $A$ is the universal grading the first and second partial obstructions automatically vanish. If $A$ is the universal grading and $\cB$ is modular then $\ker(\pi_2)=U(\cB_e)$.  Since $\pi_1$ is trivial, the grading on $\cB_e$ is trivial, so $(\cB_e)_g=\cB_e=\cB_{ad}$.  Therefore $\lambda(a,b)\in \cB_{pt}$ centralizes $\cB_e$, $\chi_{\lambda(a,b)}(X)=\id_X$, and the first partial obstruction vanishes. The triviality of $\ker(\pi_1)$ implies that the second partial obstruction vanishes.

\end{remark}

%\begin{lemma}

\subsubsection{Shuffle identities}\label{sec: shuffle hom}
In this section we collect some notation and identities that will be useful later.

We use the following notation, where $A$ is a group:

\begin{itemize}
\item[(a)] $A^p|A^q=\{\bm{x}|\bm{y}=(x_1,\ldots,x_p|y_1,\ldots,y_q), x_i, y_j \in A \},$ \  $p,q\geq 0$. 
\item[(b)] ${\rm{Shuff}}(p,q)$ is the set of  $(p,q)$-shuffles, i.e.  elements $\lambda$ in the symmetric group ${\mathbb{S}}_{p+q}$ such that $\lambda(i) <\lambda(j)$ whenever $1 \leq i < j \leq p $ or $p+1 \leq i < j \leq p+q$.
%\item Any $\pi \in {\rm{Shuff}}(p,q)$ defines a map \begin{align} \pi: X^{p+q}&\to X^{p+q}\\ (x_1,\ldots ,x_{p+q})&\mapsto (x_{\pi(1)}\ldots,x_{\pi(p+q)}) \end{align}
\end{itemize}

Now let $A$ and $N$ be abelian groups.
We define a double complex by $D^{p,q}(A,N)=0$ if $p$ or $q$ is zero and$$D^{p,q}(A ,N):= \Maps(A^p | A^q; N),\  \  p,q > 0$$
 with horizontal and vertical differentials the standard differentials, that is,
  $$\delta_h : D^{p,q}(A, N)=C^p(A,C^q(A,N)) \to
D^{p+1,q}(A, N)=C^{p+1}(A,C^q(A,N))$$ and 
 $$\delta_v : D^{p,q}(A, N)=C^q(A,C^p(A,N)) \to
D^{p,q+1}(A,N)=C^{q+1}(A,C^p(A,N))$$ 
defined by the equations
\begin{align*}(\delta_hF)(g_1,..., g_{p+1} | k_1,...,k_q)
= F(g_2,&...,g_{p+1}| k_1,...,k_q)\\ & + \sum_{i=1}^{p}(-1)^i F(g_1,...,g_ig_{i+1},
..,g_{p+1}| k_1,...,k_q )\\ &+ (-1)^{p+1} F(g_1,...,g_p|
k_1,...,k_q ) \\
(\delta_vF)(g_1,...,g_p| k_1, ..., k_{q+1} ) = 
F(g_1,&...,g_{p}| k_2 ,..., k_{q+1}) \\ &+
\sum_{j=1}^{q}(-1)^j F(g_1,...,g_p| k_1, ...,k_jk_{j+1},..., k_{q+1} )\\
&+ (-1)^{q+1} F(g_1,...,g_p| k_1, ..., k_{q}).
\end{align*}

For any $\alpha \in C^n(A,N)$, for any $1\leq p\leq n-1$ we define the $p$-th shuffle  $\alpha_p\in \Maps(A^{p}| A^{n-p},N)$ as
\begin{align*}
    \alpha_p(a_1,\ldots,a_p|a_{p+1},\ldots,a_n)= \sum_{\pi \in \rm{Shuff}(p,n-p)} (-1)^{\epsilon(\pi)}\alpha(a_{\pi(1)},\ldots, a_{\pi(n)}).
\end{align*}

\begin{proposition}\label{prop: HS}\cite[Proposition 2, page 123]{hochschild1953cohomology}
Let $\alpha\in C^n(A,N)$ then 
\begin{align}
    (\delta \alpha)_p= \delta_h(\alpha_{p-1})+(-1)^p \delta_v(\alpha_{p}).
\end{align}for all $1\leq p\leq n$, where by notation $\alpha_0=\alpha_n=0$, and $\delta$ is the standard differential \eqref{cochains G}.
\end{proposition}

%For future reference it will be useful to describe some of the equations consequence of Proposition \ref{prop: HS}.

%Let $\gamma \in C^3(A,N)$ and $\beta \in C^2(A,N)$ then 

%\begin{align}    \delta_v(\gamma_1)(a_1,|a_2,a_3,a_4)&=-\sum_{\pi \in \rm{Shuff}(1,3)}  (-1)^{\epsilon(\pi)}\delta(\gamma)(a_{\pi(1)},\ldots, a_{\pi(4)}) \label{shuff 1 3}\\    \delta_h(\gamma_1)(a_1,a_2|a_3,a_4)+\delta_v(\gamma_2)(a_1,a_2|a_3,a_4)&=\sum_{\pi \in \rm{Shuff}(2,2)} (-1)^{\epsilon(\pi)}\delta(\gamma)(a_{\pi(1)},\ldots, a_{\pi(4)}) \label{shuff 2 2}\\    \delta_h(\gamma_2)(a_1,a_2,a_3|a_4)&=\sum_{\pi \in \rm{Shuff}(3,1)} (-1)^{\epsilon(\pi)}\delta(\gamma)(a_{\pi(1)},\ldots, a_{\pi(4)}) \label{shuff 3 1}\\    \delta_v(\beta_1)(a_1|a_2,a_3)&= \sum_{\pi \in \rm{Shuff}(1,2)} (-1)^{\epsilon(\pi)}\delta(\gamma)(a_{\pi(1)},a_{\pi(2)}, a_{\pi(3)}) \label{shuff 1 2}\\    \delta_h(\beta_1)(a_1,a_2|a_3)&= \sum_{\pi \in \rm{Shuff}(2,1)} (-1)^{\epsilon(\pi)}\delta(\gamma)(a_{\pi(1)},a_{\pi(2)}, a_{\pi(3)}) \label{shuff 2 1}\end{align}

\subsubsection{General obstructions}
Now let us consider the more general obstruction theory. 

Fix an associative zesting $\lambda$ of $\cB$ (i.e., a $2$-cocycle $\lambda\in H^2(A,\Inv(\cB_e))$ and isomorphisms $\lambda$ satisfying Definition \ref{def:assoczesting}) such that the first and second partial obstructions vanish.

In order to solve the equation in Figure \ref{fig:unnormalized-1} we fix an arbitrary family of isomorphisms $\left \{\nu(a,b):\lambda(a,b)\to \lambda(b,a) \right \}_{a,b\in A}$ (see Figure \ref{fig:map-nu}) and a map $\tilde{j}: A\to \Aut_{\ot}^{\ker(\pi_1)}(\Id_{\cB_e})$ such that $\chi_{\lambda(a,b)}|_{\cB_e}=\tilde{j}_a\circ \tilde{j}_b\circ \tilde{j}_{ab}^{-1}.$ 
 \begin{figure}[h!]
    %\centering
  \begin{align*}
  \begin{tikzpicture}[line width=1,scale=1.75,baseline=35]
    \draw(0,-.5) node[below] {$\lambda(h,g)$}--(0,0);
      \draw(0,2)--(0,2.5)node[above] {$\lambda(g,h)$};
    \onetri{}{}{\nu(g,h)}
    \end{tikzpicture}
    &\hspace{10pt}
    :=\hspace{10pt}
  \begin{tikzpicture}[line width=1,scale=1.75,baseline=35]
    \draw(0,-.5) node[below] {$(h,g)$}--(0,0);
      \draw(0,2)--(0,2.5)node[above] {$(g,h)$};
    \onetri{}{}{g,h}
    \end{tikzpicture}
    \end{align*}
    \caption{Diagram for the isomorphisms $\nu(g,h)$.}
    \label{fig:map-nu}
\end{figure}
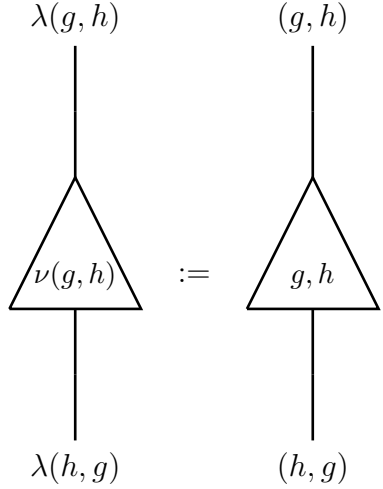
Then we define a map $O_1(\lambda,\nu,\tilde{j})\in \Maps(A^2|A,\ku^\times)$,
\begin{align*}
O_1(\lambda,\nu,\tilde{j}):A\times A\times A &\to \ku^\times\\
 (a_1,a_2,a_3) &\mapsto O_1(a_1|a_2,a_3)
\end{align*}
by  Figure \ref{fig:obs hex 1}.
\begin{figure}[h]
\begin{center}
 $O_1(1;2,3)\id$
  \hspace{10pt}
    :=
    \hspace{10pt}
    \begin{tikzpicture}[yscale=.75,xscale=1.25, line width=1,baseline=135]
    \foreach \x in {0,1}{
    \draw (\x,0)--(\x,12);}
    \draw (1,0) node[below] {$(12,3)$};
    \onetri{(1,2)}{}{1,2}
    \draw (.5,1.5) node[] {$^{-1}$};
    \begin{scope}[yshift=2cm]
    \twoboxnostrands{2,1,3}
      \draw (1.5,1.25) node[] {$^{-1}$};
    \end{scope}
      \begin{scope}[yshift=4cm]
    \onetri{}{}{1,3}
     \draw (.5,1.5) node[] {$^{-1}$};
    \end{scope}
        \begin{scope}[yshift=6cm]
    \twoboxnostrands{2,3,1}
    \end{scope}
       \begin{scope}[yshift=8cm]
       \onebox{}{}{\tilde{j}_1}
       \begin{scope}[xshift=1cm]
   \onetri{}{}{1,23}
       \end{scope}
      \end{scope}
   \begin{scope}[yshift=10cm]
       \twoboxnobottom{(1,2)}{(12,3)}{1,2,3}
    \end{scope}
    \end{tikzpicture}
\end{center}
   \caption{Obstruction to Figure \ref{fig:unnormalized-1}.}
   \label{fig:obs hex 1}
\end{figure}
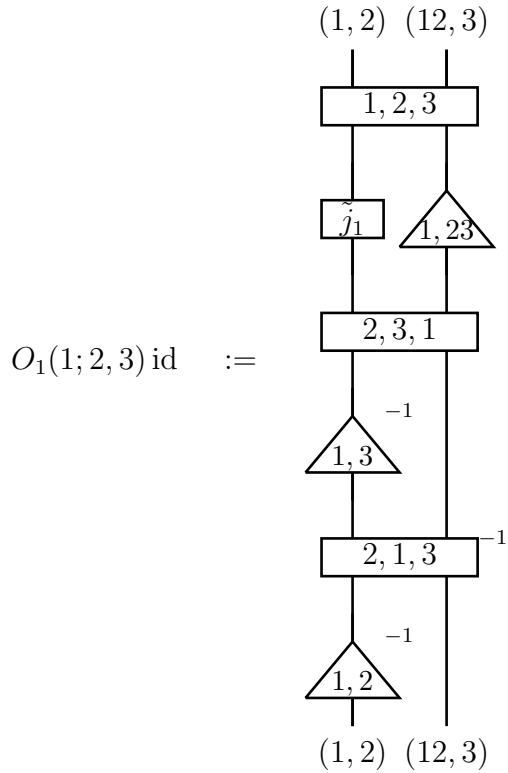

\begin{lemma}\label{prop:obst O1}
\begin{itemize}
    \item[(i)] For any $a\in A$ the function $O_1(\lambda,\nu, \tilde{j})(a|-,-):A\times A\to \ku^\times$ defines a 2-cocycle. 
    \item[(ii)] The cohomology class $O_1(\lambda,\nu,\tilde{j})(a|-,-)$ does not depend on the choice of the family $\nu$ or $\tilde{j}$, and we will be denoted by $O_1(\lambda)(a|-,-)$.
    \item[(iii)] There is a choice of isomorphisms $\nu$ that satisfies the equation in Figure \ref{fig:unnormalized-1} if and only if the cohomology class of $O_1(\lambda)(a|-,-)\in H^2(A,\ku^\times)$ vanishes for each $a\in A$.
\end{itemize}
\end{lemma}
\begin{proof}
To simplify the arguments and notation  we will  assume without serious loss of generality that 
\begin{align*}
\lambda(a_1,a_2)\ot \lambda(a_1a_2,a_3)=\lambda(a_2,a_3)\ot \lambda(a_2a_3,a_1)\\
\lambda(a_1,a_2)=\lambda(a_2,a_1), \quad \lambda(a_1,a_2)\otimes\lambda(a_3,a_4)=\lambda(a_3,a_4)\otimes\lambda(a_1,a_2)
\end{align*}
for all $a_1, a_2, a_3, a_4 \in A.$
Hence the isomorphisms $\lambda(a_1,a_2,a_3)$ and $\nu(a_1,a_2)$ are defined by cochains $\lambda\in C^3(A,\ku^\times)$,  $\nu\in C^2(A,\ku^\times)$ and the associative zesting constraint can be written as 
\begin{align}\label{eq: associative zesting cochain condition}
\delta(\lambda)
=(\lambda \cup_{c'} \lambda),
\end{align} where $(\lambda \cup_{c'} \lambda)(a_1,a_2,a_3,a_4)\in \ku^\times$ is defined by $$(\lambda \cup_{c'} \lambda)(a_1,a_2,a_3,a_4)\id_{\lambda(a_1,a_2)\ot \lambda(a_3,a_4)}=c_{\lambda(a_3,a_4),\lambda(a_1,a_2)}^{-1}.$$
(i) \ \ We have that
 \begin{align*}
O_1(a_1;a_2,a_3)&= \tilde{j}_{a_1}(\lambda(a_2,a_3))\times \frac{\nu(a_1,a_2+a_3)}{\nu(a_1,a_2)\nu(a_1,a_3)}\times \frac{\lambda(a_1,a_2,a_3)\lambda(a_2,a_3,a_1)}{\lambda(a_2,a_1,a_3)}\\
&= \tilde{j}_{a_1}(\lambda(a_2,a_3))\times\delta_v(\nu^{-1})(a_1|a_2,a_3)\lambda_1(a_1|a_2,a_3)
\end{align*}for all $a_1,a_2,a_3 \in A$, where we have been using the notation introduced in Section \ref{sec: shuffle hom}. Note that the 2-cochains $(a_2,a_2)\mapsto \tilde{j}_{a_1}(\lambda(a_2,a_3))$  and $(a_2,a_3)\mapsto  \delta_h(\nu)(a_1|a_2,a_3)$ are  2-cocycles.  Hence, to prove that $O_1(a_1;-,-)$ is a $2$-cocycle we only need to check that

\begin{align*}
\delta_v(\lambda_1)(a_1|a_2,a_3,a_4)=1 && \forall a_1,a_2,a_3,a_4\in A.
\end{align*}
It follows from Proposition \ref{prop: HS} that
\begin{align*}
\delta_v(\lambda_1)(a_1|a_2,a_3,a_4)^{-1}=&\prod_{\pi \in \rm{Shuff}(1,3)}  \delta(\lambda)(a_{\pi(1)},\ldots, a_{\pi(4)})^{\epsilon(\pi)}\\
=&\frac{(\lambda \cup_{c'} \lambda)(a_1,a_2,a_3,a_4)(\lambda \cup_{c'} \lambda)(a_2,a_3,a_4,a_1)}{(\lambda \cup_{c'} \lambda)(a_2,a_1,a_3,a_4)(\lambda \cup_{c'} \lambda)(a_2,a_3,a_1,a_4)}\\
&=1,
\end{align*}where the last equality follows because $\lambda(a,b)=\lambda(b,a)$ for all $a,b\in A$.

(ii)  \ \ Recall that every symmetric 2-cocycle in $Z^2(A,\ku^\times)$ is a coboundary (since $\ku^\times$ is divisible). Then   $\delta_v(\nu^{-1})(a_1|-,-)$  and $\tilde{j}_{a_1}(\lambda(-,-))$ contributes with a 2-coboundary to $O_1(\lambda,\nu,t)(a_1|-,-)$. Hence the cohomology class  $O_1(a_1|-,-)$ only depends on $\lambda$. 

(iii)\   If there is $h\in C^2(A,\ku^\times)$ such that $\delta_v(h)=O_1(\lambda,\nu,\tilde{j})$, then taking $\nu'=\nu h^{-1}$  we have  that $O_1(\lambda,\nu',\tilde{j})=1$.  
\end{proof}

\begin{remark}\label{remark about the t's}
In practice we can take $\lambda(a,b)=\lambda(b,a)$ for all $a,b \in A$. Since $\lambda(a,b)$ is invertible we have that every isomorphism $\nu:\lambda(a,b)\to \lambda(b,a)$ is a multiple of the identity. Hence in order to compute $O_1(\lambda)$, we can start with $\nu(a,b)=\id_{\lambda(a,b)}$ for all $a,b\in A$.
\end{remark}

Assuming that the cohomology classes of $O_1(\lambda)$ vanish, we can find isomorphisms $\nu:\lambda(a,b)\to \lambda(b,a)$ such that $O_1(\lambda,\nu,j)(a_1|a_2,a_3)=1$ for all $a_1,a_2,a_3\in A$. We define the map $O_2(\lambda,\mu,j)(a|b,c)\in \ku^\times $ by  Figure \ref{fig:second-obs}.

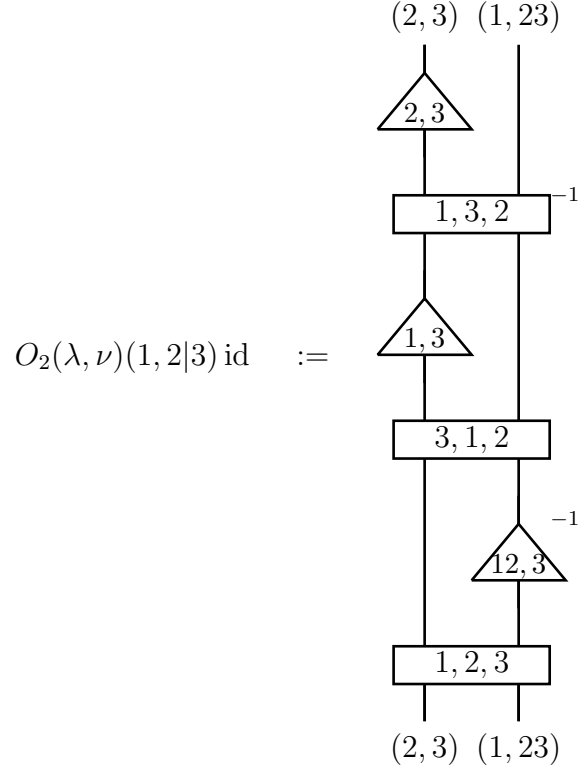
\begin{figure}[h]
\begin{center}  $O_2(\lambda,\nu)(1,2|3)\id$
    \hspace{10pt}
    :=
    \hspace{10pt} 
    \begin{tikzpicture}[yscale=.75,xscale=1.25, line width=1,baseline=135]
    \foreach \x in {0,1}{
    \draw (\x,0)--(\x,12);}
    \draw (1,12) node[above] {$(1,23)$};
    \draw (1,0) node[below] {$(1,23)$};
    \begin{scope}[yshift=10cm]
    \onetri{}{(2,3)}{2,3}
    \end{scope}
    \begin{scope}[yshift=8cm]
    \twoboxnostrands{1,3,2}
       \draw (1.5,1.25) node[] {$^{-1}$};
    \end{scope}
    \begin{scope}[yshift=6cm]
        \onetri{}{}{1,3}
    \end{scope}
      \begin{scope}[yshift=4cm]
        \twoboxnostrands{3,1,2}
    \end{scope}
       \begin{scope}[xshift=1cm,yshift=2cm]
   \onetri{}{}{12,3}
     \draw (.5,1.5) node[] {$^{-1}$};
      \end{scope}
       \twoboxnotop{(2,3)}{}{1,2,3}
    \end{tikzpicture}
   \end{center}
    \caption{Diagrammatic definition of the second obstruction to braided zesting.}
    \label{fig:second-obs}
\end{figure}

\begin{lemma}\label{prop:obst O2}
Let $\tilde{j}:A\to \Aut_\ot^{\ker(\pi_1)}(\Id_\cB)$ be map such that $\delta(\tilde{j})(a,b)=\chi_{\lambda(a,b)}|_{\cB_e}$ for all $a,b\in A$ and let $\nu:\lambda(a,b)\to \lambda(b,a)$ be a family of isomorphisms such that 
\begin{align*}
O_1(\lambda,\nu,\tilde{j})(a_1|a_2,a_3)=1,&& \forall a_1,a_2,a_3\in A.
\end{align*}
Then
\begin{itemize}
\item[(i)]  $O_2(\lambda,\nu)(-|a,b)\in \widehat{A}$ for all $a,b\in A$.

\item[(ii)] The  2-cochain 
    \begin{align*}
        O_2(\lambda,\nu):A\times A &\to \widehat{A}\\
        (a,b)&\mapsto [c\mapsto O_2(\lambda)(c|a,b)]
    \end{align*}
     defines a 2-cocycle $O_2(\lambda,\nu)\in Z^2(A,\widehat{A})$. The cohomology class of $O_2(\lambda,\nu)$ does not depend on the choice of the $\nu$  (under the hypothesis that the $O_1(\lambda,\nu)=1$), and will be denoted by $O_2(\lambda)$.
\end{itemize}
\end{lemma}
\begin{proof}
As in the proof of Lemma \ref{prop:obst O1}, we will assume that $\lambda(a,b,c)$ and $\nu(a,b)$ are defined by cochains $\lambda\in C^3(A,\ku^\times), \nu \in C^2(A,\ku^\times)$.

The condition $O_1(\lambda,\nu,\tilde{j})=1$ can be written as
\begin{align}\label{eq: O_1 trivial coho}
    \delta_v(\nu)(a_1|a_2,a_3)=\lambda_1(a_1|a_2,a_3)\tilde{j}_{a_1}(\lambda(a_2,a_3)), && \forall a_1, a_2, a_3 \in A.
\end{align}

(i) \ \  
We need to check that $\delta_v(O_2(\lambda,\nu,\tilde{j}))(a_1,a_2|a_3,a_4)=1$, where 
\begin{align}
O_2(\lambda,\nu,\tilde{j})(a_1,a_2|a_3)=\delta_h(\nu)(a_1,a_2|a_3)\lambda_2(a_1,a_2|a_3).
\end{align}
First we have that

\begin{align*}
    \delta_v(\delta_h(v))(a_1,a_2|a_3,a_4)=&\delta_h(\delta_v(v))(a_1,a_2|a_3,a_4)\\
    =&\delta_h(\lambda_1)(a_1,a_2|a_3,a_4) \tilde{j}_{a_1}\tilde{j}_{a_1+a_2}^{-1}\tilde{j}_{a_2}(\lambda(a_3,a_4))\\
    =&\delta_h(\lambda_1)(a_1,a_2|a_3,a_4)\chi_{\lambda(a_1,a_2)}(\lambda(a_3,a_4)).
\end{align*}
and
\begin{align*}
\delta_v(O_2(\lambda,\nu,\tilde{j}))(a_1,a_2|a_3,a_4)&= \delta_v(\delta_h(v))\delta_v(\lambda_2)(a_1,a_2|a_3,a_4)\\
&=\big (\delta_h(\lambda_1)\delta_v(\lambda_2)\big )(a_1,a_2|a_3,a_4)\chi_{\lambda(a_1,a_2)}(\lambda(a_3,a_4)).
\end{align*}Using Proposition \ref{prop: HS}   we have

\begin{align*}
  \delta_h(\lambda_1)(a_1,a_2|a_3,a_4)\delta_v(\lambda_2)(a_1,a_2|a_3,a_4)=& \prod_{\pi \in \rm{Shuff}(2,2)}  \delta(\lambda)(a_{\pi(1)},\ldots, a_{\pi(4)})^{\epsilon(\pi)}\\
  =& \lambda\cup_{c'}\lambda(a_1,a_2,a_3,a_4)\lambda\cup_{c'}\lambda(a_3,a_4,a_1,a_2)\\ &\lambda\cup_{c'}\lambda(a_3,a_1,a_2,a_4)\lambda\cup_{c'}\lambda(a_1,a_3,a_2,a_4)\\
  &\lambda\cup_{c'}\lambda(a_3,a_1,a_4,a_2)^{-1}\lambda\cup_{c'}\lambda(a_1,a_3,a_4,a_2)^{-1}\\
  =&\lambda\cup_{c'}\lambda(a_1,a_2,a_3,a_4)\lambda\cup_{c'}\lambda(a_3,a_4,a_1,a_2)\\
  =& \chi_{\lambda(a_1,a_2)}^{-1}(\lambda(a_3,a_4)).
\end{align*}
Then 
\[\delta_v(O_2(\lambda,\nu,\tilde{j}))(a_1,a_2|a_3,a_4)=1,\]as we wanted to check.

(ii) The 2-cocycle condition in this case is 
\[\delta_h(O_2(\lambda,\nu,\tilde{j}))=\delta_h(\lambda_2)=1.\]
Using again Proposition \ref{prop: HS} we have 

\begin{align*}
\delta_h(\lambda_2)=& \prod_{\pi \in \rm{Shuff}(3,1)}  \delta(\gamma)(a_{\pi(1)},\ldots, a_{\pi(4)})^{\epsilon(\pi)}\\
=&\lambda\cup_{c'} \lambda(a_1,a_2,a_3,a_4)\lambda\cup_{c'} \lambda(a_4,a_1,a_2,a_3)\\
&\lambda\cup_{c'}\lambda(a_1,a_4,a_2,a_3)^{-1}\lambda\cup_{c'}\lambda(a_1,a_2,a_4,a_3)^{-1}\\
=&1.
\end{align*}
Finally, if $\nu'$ is another 2-cochain such that $O_1(\lambda,\nu',\tilde{j})=1$, then $\delta_v(\nu/\nu')=1$, that is $\nu/\nu'\in C^1(A,\widehat{A})$, and then $O_2(\lambda,\nu',\tilde{j},)=\delta_h(\nu/\nu')O_2(\lambda,\nu',\tilde{j},)$, that is  $O_2(\lambda,\nu',\tilde{j},)$ and $O_2(\lambda,\nu,\tilde{j},)$ are cohomologous in $H^2(A,\widehat{A})$.
\end{proof}

The short exact sequence \[0\to \widehat{A}\to \widehat{U(\cB)}\to \widehat{\ker(\pi_1)}\to 0,\]induces a long exact sequence in cohomology

\begin{align}\label{eq: long exact sequence}
\cdots\to \Hom(A,\widehat{\ker(\pi_1)})\overset{d_1}{\to} H^2(A,\widehat{A})\to H^2(A,\widehat{U(\cB)})\to  \cdots 
\end{align}
The set  $$S_{\chi}:=\{ \tilde{j}: A\to \Aut_\ot^{\ker(\pi_1)}(\Id_{\cB_e})| \delta(\tilde{j})=\chi_{\lambda}|_{\cB}\}$$ is a torsor over the abelian group $\Hom(A,\widehat{\ker(\pi_1)})$. It follows from Proposition \ref{prop:condition extension monoidal auto identity}  that  for each $\tilde{j}\in S_\chi$ there is $j:A\to \Aut_\ot(\Id_{\cB})$ such that  $\tilde{j}_a=(j_a)|_{\cB_e}$ for each $a\in A$. Hence the natural isomorphisms \[\omega_j(a,b):=\chi_{\lambda(a,b)}\circ j_a^{-1}\circ j_b^{-1}\circ j_{ab}\in \Aut_\ot^{A}(\Id_{\cB})\]
define a 2-cocycle $\omega_j\in Z^2(A,\widehat{A}),$ and again by Proposition \ref{prop:condition extension monoidal auto identity} the cohomology class of $\omega_j$ only depend on $\tilde{j}$, and we will denote by 
\begin{equation}
    \omega_{\tilde{j}}\in H^2(A,\widehat{A}).
\end{equation} Note that if $j', j'' \in S_\chi$, then $\omega_{\tilde{j}/\tilde{j}''}=d_1(\tilde{j}/\tilde{j}'')$, where $d_1$ is defined in \eqref{eq: long exact sequence}.

As a result of the Lemmas \ref{prop:obst O1}, \ref{prop:obst O2} and the previous discussion, we obtain the following result.
\begin{theorem}\label{th: obstruction}
Let $\lambda$ be an associative zesting. Then there is a braided zesting associated if and only if the cohomology classes of $O_1(\lambda)(a,-,-)\in H^2(A,\ku^\times)$  vanish for all $a\in A$,  and there exist $\tilde{j}\in S_\lambda$ such that $O_2(\lambda)=\omega_{\tilde{j}}\in H^2(A,\widehat{A})$. 
\qed
\end{theorem}

\begin{corollary}\label{cor: obstruction zesting sin j}
Let $\lambda$ be an associative zesting such that $\chi_{\lambda(a,b)}|_{\cB_e}=\Id_{\cB_e}$ for all $a,b\in A$. Then there is a braided zesting of the form $(\Id,t)$  if and only if the cohomology classes of $O_1(\lambda)(a,-,-)\in H^2(A,\ku^\times)$  vanish for all $a\in A$,  and $O_2(\lambda)=[\chi_{\lambda}]\in H^2(A,\widehat{A})$.\qed
\end{corollary}

\section{Twist Zesting and its Modular Data}\label{section:twist zesting}

Given a premodular tensor category $\cB$ and a braided zesting $(\lambda,j,t)$ we will denote by $(\cB^\lambda,t)$ the corresponding braided fusion category as constructed in Proposition \ref{prop: braided zesting construction}, suppressing the dependence on $j$. We would like to provide $(\cB^\lambda,t)$ with a ribbon structure. In a customary abuse of notation we denote by $\theta_X$ both the automorphism in $\Aut_\cB(X)$ and the scalar by which it acts when $X$ is a simple object.  Similarly, we will denote the scalar by which $\chi_{a}$ acts on a simple object $X \in \cB$ by $\chi_a(X)$ as well.

\begin{proposition}

Let $A$ be a finite abelian group, $\cB$ be a faithfully  $A$-graded braided tensor category with twist $\theta$ and $(\lambda, j,t)$ a braided zesting. We will denote by \[t^{(2)}:A\times A\to \ku^\times\]
the symmetric function defined by $t(b,a)\circ t(a,b)=:t^{(2)}(a,b)\id_{\lambda(a,b)}$ for all $a,b\in A$.
Let $f:A\to \ku^\times$ be a function and consider the natural isomorphism 
\begin{equation*}
\theta_{X_a}^f:=f(a)\theta_{X_a},\quad a\in A, X_a\in \cB_a,    
\end{equation*} then
\begin{itemize}
    \item[(i)] $\theta^f$ is a twist for $(\cB^\lambda,t)$ if and only if 
    \begin{align}\label{twist-condition}
    f(a+b)\chi_{\lambda(a,b)}(X_a)\chi_{\lambda(a,b)}(Y_b)\theta_{\lambda(a,b)}=f(a)f(b)j_a(Y_b)j_b(X_a)t^{(2)}(a,b), && f(0)=1,
    \end{align}for all $a,b \in A, X_a\in \cB_a, Y_b\in \cB_b$.
    \item[(ii)] If $\theta$ is a ribbon twist for $\cB$ then
   $\theta^f$ is a ribbon twist for  $(\cB^\lambda,t)$ if additionally to equation \ref{twist-condition} we have
    \begin{equation}\label{ribbon condition}
        f(a)=f(-a)\chi_{\lambda(a,-a)}(X_a)\theta_{\lambda(a,-a)}
    \end{equation}
for all $a \in A, X_a\in \cB_a$.

\item[(iii)] If $f',f:A\to \ku^\times$ is a pair of functions satisfying \eqref{twist-condition}, then $f/f':A\to \ku^\times$ is a character. Moreover, the set of all functions  satisfying \eqref{twist-condition} is a torsor over $\widehat{A}$ and the set of all functions satisfying \eqref{twist-condition} and \eqref{ribbon condition} is a torsor over $\widehat{A/2A}$.
\end{itemize}

\end{proposition}\label{prop:twist prop general}
\begin{proof}
Let $X_a\in \cB_a$ and $Y_b\in \cB_b$ simple objects. Then we have that 
\begin{align*}
    \big (\theta^f_{X_a}\otlam \theta^f_{Y_b}\big ) \circ c_{Y_a,X_a}^\lambda \circ c_{X_a,Y_b}^\lambda= f(a)f(b)j_a(Y_b)j_b(X_a)t^{(2)}(a,b)\theta_{X_a\otimes Y_b}\otimes \id_{\lambda(a,b)}, 
\end{align*}
and
\[\theta^f_{X_a\otlam Y_b}= f(a+b)\chi_{\lambda(a,b)}(X_a)\chi_{\lambda(a,b)}(Y_b)\theta_{\lambda(a,b)}\theta_{X_a\otimes Y_b}\id_{\lambda(a,b)}.\]
Hence,  \eqref{twist-condition} holds if and only if \[\big (\theta^f_{X_a}\otlam \theta^f_{Y_b}\big ) \circ c_{Y_a,X_a}^\lambda \circ c^\lambda_{X_a,Y_b}= \theta^f_{X_a\otlam Y_b},\] that is, if $\theta^f$ is a twist.

For the ribbon condition, we have that
\begin{align*}
\theta^f_{\overline{X}_a}&=f(-a)\chi_{\lambda(a,-a)}(X_a)\theta_{X_a^*}\theta_{\lambda(a,-a)^*}\id_{X^*a\otimes \lambda(a,-a)^*} \\
&=f(-a)\chi_{\lambda(a,-a)}(X_a)\theta_{X_a}\theta_{\lambda(a,-a)}\id_{\overline{X}_a}
\end{align*}
and \[\theta^f_{X_a}=f(a)\theta_{X_a}\id_{X_a}\]
for all simple objects $X_a\in \cB_a$. Hence $(\theta_{X_a}^f)^*=\theta_{\overline{X}_a}^f$ if and only if \eqref{ribbon condition} holds.

For (iii), let $f$ and $f^\prime$ both satisfy the conditions in \eqref{twist-condition} and set $\eta(a)=f(a)/f^\prime(a)$.  Since $f(a)f(b)/f(a+b)=f^\prime(a)f^\prime (b)/f^\prime(a+b)$ we find that $\eta(a+b)=\eta(a)\eta(b)$.  A similar argument implies that if $f,f^\prime$ satisfy the condition \eqref{ribbon condition} then $\eta(a)=\eta(-a)$, so $\eta(2a)=1$, that is $\eta\in \widehat{A/2A}$.
\end{proof}
\begin{definition}
A quadruple $(\lambda,j,t,f)$ where $(\lambda,j,t)$ is a braided zesting and $f:A\to \ku^\times$ is a function satisfying equations \eqref{twist-condition}
 and \eqref{ribbon condition} is called a ribbon zesting.
 \end{definition}
 We will denote by $(\cB^\lambda,t,f)$ the twist (ribbon) zesting obtained from $(\lambda,j,t,f)$.
\begin{remark}
We do not know if twists or ribbons for braided zesting always exists. However, in practice, the following condition gives you an easy to check requirement for the existence of them. If a twist exists the symmetric function 
\begin{align*}
s: A\times A \to &\ku^\times\\
(a,b)\mapsto &\frac{\chi_{\lambda(a,b)}(X_a)\chi_{\lambda(a,b)}(Y_b)\theta_{\lambda(a,b)}}{j_a(Y_b)j_b(X_a)t^{(2)}(a,b)} ,
\end{align*}
should be independent of the choice of $X_a \in \cB_a, Y_b\in \cB_b$. If the function $s$ is a 2-cocycle, (an easy condition to check) since $s$ symmetric we can find a function $f$ satisfying condition \eqref{twist-condition}. 
\end{remark}

When zesting with respect to the universal grading, the scalars $\chi_{\lambda(i,j)}(X_k)$ only depend on the graded component of the simple object $X_k\in\cB_k$, so we may denote it $\chi_{\lambda(i,j)}(k)$, and take $j_a=\id$.  In this case the conditions (\ref{twist-condition}) and (\ref{ribbon condition}) for twist zesting reduce to a simpler form:
\begin{corollary}
\label{lem:twist} Suppose $\cB$ is a braided fusion category with a twist and $(\cB^\lambda,t)$ is a braided $A$-zesting where $A=U(\cB)$ is the universal grading group.   Then for $\theta^f\in\Aut(\Id_\cB)$ defined in Proposition \ref{prop:twist prop general} we have:
\begin{itemize}
    \item[(a)] if $f(a+b)\chi_{\lambda(a,b)}(a+b) \theta_{\lambda(a,b)}=f(a)f(b)\,t^{(2)}(a,b)$,

then ${\theta^f}$ defines a twist on $(\cB^\lambda,t)$ and 
    \item[(b)] if $\theta$ is a ribbon twist on $\cB$ and $f(a) = \chi_{\lambda(a,-a)}(a) \theta_{\lambda(a,-a)}f(-a)$ then ${\theta^f}$ defines a ribbon twist on $(\cB^\lambda,t)$.
%    \item[(c)] The set of twists of this form on $(\cB^\lambda,t)$ forms a torsor over $\hat{A}$,  and the set of ribbon twists of this form form a torsor over the subgroup $\Omega_2(\hat{A})$ of characters of $\eta\in\hat{A}$ such that $\eta(a)=\eta(-a)$. 
\end{itemize}
\end{corollary}

\subsection{Modular data of a ribbon zesting}

\begin{proposition}\label{zested trace}
Let $(\lambda,j,t,f)$ be a a ribbon zesting, then quantum trace of an endomorphism of $s:X_a\to X_a$ in $(\cB^\lambda,t,f)$, for $a\in A, X_a\in \cB_a$ is 
\[\operatorname{Tr}^f(s)=\frac{ f(a)}{\dim(\lambda(-a,a)) t(a,a)}\operatorname{Tr}(j_a^{-1}(X_a)\circ s).\]
\end{proposition}

\begin{proof}
In this proof, without loss of generality, we will assume that $\cB$ is a strict pivotal category, which means that the natural isomorphism between an object and its double dual is the identity morphism and also the pivotal structure is trivial.

The trace in $(\cB^\lambda,t,f)$ is given by the formula \[\operatorname{Tr}^f(s)= \frac{1}{\dim(\lambda(-a,a))}\rho_{X^*_a}\circ (\tilde{\psi}\otlam\id)\circ (s\otlam\id)
\circ \phi_{X_a},\] 
where $\tilde{\psi}$ is the pivotal structure in $(\cB^\lambda,t,f)$ and $\rho$ and $\phi$ denote the evaluation and coevaluation maps in $(\cB^\lambda,t,f)$ described in Subsection \ref{section: rigidity}. So, in order to compute $\operatorname{Tr}^f(s)$, we need to compute the pivotal structure $\tilde{\psi}$ first. To do this, we can consider the Drinfeld isomporhism $\tilde{u}$ in $(\cB^\lambda,t,f)$ which is related to the twist and pivotal structure by $\tilde{\psi} = \tilde{u}\circ \tilde{\theta}^f$.

The general formula of the Drinfeld isomorphism $\tilde{u}$ is given by \[\tilde{u} = \frac{1}{\dim(\lambda(-a,a))}(\rho_{X_a}\otlam \id_{X_a})\circ (\id_{X_a^*}\otlam\tilde{c}^{-1}_{X_a, X_a})
\circ (\phi_{X^*_a}\otlam\id_{X_a}).\] 

Applying this formula in our case, and assuming strict pivotality of $\cB$, we get that $\tilde{u}$ can be written in terms of the data of the original category $\cB$ and the zesting structure as

\begin{align}
\label{zdrinfeld1}
   \tilde{u}_x &=  \frac{1}{\dim(\lambda(-a,a))} \hspace{10pt} \begin{tikzpicture}[line width=1,scale=.8,baseline=80] 
       \draw (12.5,6) to [out=90, in=90] (6,6);
   \draw (12.5,2) to [out=-90,in=90] (6,0);
    \foreach \x in {8,9,10}{\draw[white,line width=10] (\x,0) \br (\x-4,6);}
   \foreach \x in {8,9,10}{\draw (\x,-2)--(\x,0) \br (\x-4,6);}
   \begin{scope}[xshift=12cm,yshift=1cm]
   \twoboxnobottom{}{}{-a,a,a}
   \draw (1,2) --(1,4);
   \begin{scope}[yshift=2cm]
     \onecirc{}{}{a,a}
      \draw (.25,1.33) node[right] {$^{-1}$};
     \begin{scope}[yshift=2cm]
        \twoboxnotop{}{}{-a,a,a}
     \end{scope}
   \end{scope}
   \end{scope} 
   \draw[looseness=2] (4,6) to [out=90, in=90] (2,6)--(2,2) to [out=-90, in=-90] (6,-0);
     \draw[looseness=1.5] (5,6) to [out=90, in=90] (1,6)--(1,-1);
      \draw[looseness=1.5,white, line width=10] (1,-1) to [out=-90, in=-90] (3.5,-1);
      \draw[looseness=1.5] (1,-1) to [out=-90, in=-90] (3.5,-1);
      \draw[looseness=.8,white, line width=10] (3.5,-1) \br (12.5,9);
      \draw[looseness=.8] (3.5,-1) \br (12.5,9) node[above] {$X_a$};
          \draw[fill=white] (3.75,5.66) rectangle node {$j_a$} (6.25,6.33);
               \draw (6.25,6.33) node[right] {$^{-1}$};
                 \draw (13.25,2.25) node[right] {$^{-1}$};
                 \draw (13.25,6.25) node[right] {$^{-1}$};
                 \draw (7.5,-2) node[below] {$\phantom{\overline{X_a}}\lambda(a,-a)\phantom{\overline{X_a}}$};
                   \draw (9,-2) node[below] {$\phantom{\overline{\lambda(-a,a)}}X_a\phantom{\overline{\lambda(-a,a)}}$};
                     \draw (10.5,-2) node[below] {$\phantom{\overline{X_a}}\overline{\lambda(-a,a)}\phantom{\overline{X_a}}$};
                      \draw[fill=white] (6-.33, -.66) rectangle node {$\lambda^a$} (6+.33, 0);
    \draw (6.33,0) node[right] {$^{-1}$};
   \end{tikzpicture} 
   \end{align}
 where $\lambda^a := \lambda_{a, -a, a}$ as was defined in Subsection \ref{section: rigidity},
   Applying the standard yoga of graphical calculus and using that  $j_g$ is a tensor autoequivalence and $\cB$ being strict pivotal, i.e. $u = \theta^{-1}$, we get the following expression for the Drinfeld isomorphism of $(\cB^\lambda,t,f)$. 
   \begin{align}
   \label{zdrinfeld2}
   \tilde{u}_x &= \frac{1}{\dim(\lambda(-a,a))t(a,a)}\hspace{10pt} \begin{tikzpicture}[line width=1,baseline=30]
    \draw[looseness=1.5] (-.5,-1) node[below] {$\phantom{\overline{X_a}}\lambda(a,-a)\phantom{\overline{X_a}}$}--(-.5,1) to [out=90,in=90] (2.5,1) --(2.5,-1) node[below] {$\phantom{\overline{X_a}}\overline{\lambda(-a,a)}\phantom{\overline{X_a}}$};
    \draw[fill=white] (-.5-.33, .66-1) rectangle node {$\lambda^a$} (-.5+.33, 1.33-1);
    \draw (-.25,1.25-1) node[right] {$^{-1}$};
    \draw[white, line width=10] (1,-1)--(1,3);
    \begin{scope}[xshift=1cm,yshift=-1cm]
     \onebox{\phantom{\overline{X_a}}X_a\phantom{\overline{X_a}}}{}{j_a}
     \draw (0,2)--(0,4) node[above] {$X_a$};
    \end{scope} 
     \begin{scope}[xshift=1cm,yshift=.5cm]
     \onebox{}{}{\theta_a}
    \end{scope} 
    \draw (1.25,1.25-1) node[right] {$^{-1}$};
      \draw (1.25,2.75-1) node[right] {$^{-1}$};
    \end{tikzpicture} \\
     &= \frac{((\lambda^a)^{-1}\otlam(j^{-1}_a\circ \theta^{-1}_{X_a})\otlam\id_{\lambda(-a,a)^*})\circ (\tilde{c}_{X_a, \lambda(-a,a)}\otlam\id_{\lambda(-a,a)^*})\circ (\id_{X_a}\otlam\rho_{\lambda(-a,a)})}{\dim(\lambda(-a,a))t(a,a)} 
\end{align}
Notice that here, if we weren't assuming $\cB$ is strict pivotal, we would get the Drinfeld isomorphism $u$ of $\cB$ instead of $\theta^{-1}$ in Equality \eqref{zdrinfeld2}.
   
From this, assuming strict pivotality of $\cB$ we get that the pivotal structure $\tilde{\psi}$ in $(\cB^\lambda,t,f)$ is 
\begin{align}
\begin{split}
    \tilde{\psi} & = \frac{f(a)}{\dim(\lambda(-a,a)) t(a,a)} \hspace{10pt} \begin{tikzpicture}[line width=1,baseline=30]
    \draw[looseness=1.5] (-.5,0)--(-.5,1) to [out=90,in=90] (2.5,1) --(2.5,0);
    \draw[fill=white] (-.5-.33, .66) rectangle node {$\lambda^a$} (-.5+.33, 1.33);
    \draw (-.25,1.25) node[right] {$^{-1}$};
    \draw[white, line width=10] (1,0)--(1,3);
    \begin{scope}[xshift=1cm]
     \onebox{}{}{j_a}
     \draw (0,2)--(0,3);
    \end{scope} 
    \draw (1.25,1.25) node[right] {$^{-1}$};
    \end{tikzpicture} \\
    & = \frac{f(a)}{\dim(\lambda(-a,a))t(a,a)} ((\lambda^a)^{-1}\otlam j^{-1}_a\otlam\id_{\lambda(-a,a)^*})\circ (\tilde{c}_{X_a, \lambda(-a,a)}\otlam\id_{\lambda(-a,a)^*})\circ (\id_{X_a}\otlam\rho_{\lambda(-a,a)})
    \end{split}
\end{align}
In this way, we get that the trace $\operatorname{Tr}^f$ in $(\cB^\lambda,t,f)$ of a morphism $s \in \Hom(X_a,X_a)$ can be expressed in the graphical calculus by

\begin{align}
  \operatorname{Tr}^f(s)  &= \frac{f(a)}{\dim(\lambda(a,-a))^{2}\, t(a,a)} \hspace{10pt}  \begin{tikzpicture}[line width=1,baseline=30,scale=.9]
    \draw[looseness=1.5] (-.5,0)--(-.5,1) to [out=90,in=90] (2.5,1) --(2.5,0);
  \draw[fill=white] (-1, .5) rectangle node {$\lambda^{a}$} (0, 1.5);
    \draw (0,1.5) node[right] {$^{-1}$};
    \draw[white, line width=10] (1,0)--(1,3);
    \begin{scope}[xshift=1cm]
     \onebox{}{}{j_a}
     \draw (0,2)--(0,3);
    \end{scope} 
    \draw (1.25,1.25) node[right] {$^{-1}$};
        \begin{scope}[xshift=-.5cm,scale=3]
\draw[looseness=1.5] (1,0) to [out=-90, in=-90] (3,0)  ;
\draw[white, line width=10, looseness=1.5] (0,0) to [out=-90, in=-90] (2,0);
\draw[looseness=1.5] (0,0)to [out=-90, in=-90] (2,0);
    \end{scope}
    \draw[looseness=1.5, white, line width=10] (1,4) to [out=90, in=90] (3,4) -- (3,0) to [out=-90,in=-90] (1,0);
    \draw[looseness=1.5] (1,4) to [out=90, in=90] (3,4) -- (3,0) to [out=-90,in=-90] (1,0);
    \draw[looseness=1.5] (5.5,0) -- (5.5,4) to [out=90,in=90] (8.5,4)--(8.5,0);
    \begin{scope}[xshift=9cm]
    \draw[fill=white] (-1, .5) rectangle node {$\lambda^{-a}$} (0, 1.5);
    \draw (0,1.5) node[right] {$^{-1}$};
    \end{scope}
     \begin{scope}[xshift=1.5cm,yshift=2.5cm]
      \draw[fill=white] (-1, .5) rectangle node {$s$} (0, 1.5);
    \end{scope}
    \end{tikzpicture} .
    \end{align}
    After expressing the closed loop involving morphisms on $X_a$ in terms of the trace in $\cB$, we get
    \begin{align}
   \operatorname{Tr}^f(s) &=  \frac{f(a)}{\dim(\lambda(a,-a))^{2}\, t(a,a)} \operatorname{Tr}(j_a^{-1}(X_a) \circ s)\hspace{10pt} \begin{tikzpicture}[line width=1, baseline=52.5,scale=.75] 
    \draw (0,0)--(0,3) \br (2,5);
    \draw [looseness=1.5] (2,5) to [out=90, in=90] (0,5);
    \draw[white,line width=10] (0,5) to [out=-90, in=90] (2,3);
    \draw (0,5) to [out=-90, in=90] (2,3);
    \draw (2,3)--(2,0);
    \draw[looseness=1.5] (2,0) to [out=-90,in=-90] (0,0);
      \draw[fill=white] (-.5, .25) rectangle node {$\lambda^{a}$} (.5, 1.25);
    \draw (.5,1.25) node[right] {$^{-1}$};
    \begin{scope}[yshift=1.25cm]
        \draw[fill=white] (-.5, .5) rectangle node {$\lambda^{-a}$} (.5, 1.5);
    \draw (.5,1.5) node[right] {$^{-1}$};
    \end{scope}.\end{tikzpicture}
    \end{align}
    Finally, we apply the associative zesting condition of Figure \ref{axiom:asso} with $1=3=a$ and $2=4=-a$ and standard graphical calculus to get the equation
    \begin{align}
    \begin{split}
    \operatorname{Tr}^f(s) &= \frac{f(a)}{\dim(\lambda(a,-a))^{2}\, t(a,a)} \operatorname{Tr}(j_a^{-1}(X_a)  \circ s)\hspace{10pt} \begin{tikzpicture}[line width=1, baseline=52.5,scale=.75] 
    \draw  (2,1) \br (0,3) \br (2,5);
    \draw [looseness=1.5] (2,5) to [out=90, in=90] (0,5);
    \draw[white,line width=10] (0,5) to [out=-90, in=90] (2,3)to [out=-90, in =90] (0,1) ;
    \draw (0,5) to [out=-90, in=90] (2,3) to [out=-90, in =90] (0,1);
    \draw[looseness=1.5] (2,1) to [out=-90,in=-90] (0,1);
 \end{tikzpicture} \\
 &= \frac{f(a)}{\dim(\lambda(a,-a))\, t(a,a)} \operatorname{Tr}(j_a^{-1}(X_a) \circ s)
    \end{split}.
\end{align}
\end{proof}

\begin{lemma}\label{lemma: s tilde}
Let $\lambda=(\lambda,j,t,f)$ be a a ribbon zesting. Then  
\begin{align}
    \label{eq: formula new S tilde}\operatorname{Tr}^f(c_{Y_b,X_a}^\lambda\circ c_{X_a,Y_b}^\lambda)=&\frac{ \dim(\lambda(a+b,-a-b))\dim(\lambda(a,b))t^{(2)}(a,b)f(a+b)}{t(a+b,a+b)}\\ &\times m(j,X_A,Y_b) \operatorname{Tr}(c_{Y,X}\circ c_{X,Y})\notag
\end{align}where 
\begin{equation}\label{ m(j,x,y)}
    m(j,X_a,Y_b)=\frac{j_a(Y_b)j_b(X_a)}{j_{ab}(X_a)j_{ab}(Y_b)j_{ab}(\lambda(a,b))},
\end{equation}for all $ a,b\in A, X_a\in \Irr(\cB_a), Y_b\in \Irr(\cB_b)$.
\end{lemma}
\begin{proof}
Let $X_a\in \Irr(\cB_a), Y_b\in \Irr(\cB_b)$, where
$ a,b\in A$. We have that
\[j_{ab}^{-1}\circ  c_{Y_b,X_a}^\lambda\circ  c_{X_a,Y_b}^\lambda= m(j,X_a,Y_b)t^{(2)}(a,b) \big (c_{Y_b,X_a} \circ  c_{X_a,Y_b} \big ) \otimes \id_{\lambda(a,b)}\]
for all pairs of simple objects $X_a, Y_b$. 

Then using Proposition \ref{zested trace} we have
\begin{align*}
     \operatorname{Tr}^f( c_{Y_b,X_a}^\lambda\circ  c_{X_a,Y_b}^\lambda)
    =&\frac{\dim(\lambda(a+b,-a-b))f(a+b)}{t(a+b,a+b)}m(j,X_A,Y_b)\\
    &\times t^{(2)}(a,b) \operatorname{Tr}\Big ( \big (c_{Y_a,X_a} \circ  c_{X_a,Y_a} \big ) \otimes \id_{\lambda(a,b)}\Big )\\
    =&\frac{\lambda(a+a,-a-a) \dim(\lambda(a,b))f(a+b)}{t(a+b,a+b)}m(j,X_A,Y_b)\\
    &\times t^{(2)}(a,b)  \operatorname{Tr}^f( c_{Y_a,X_a}^\lambda\circ  c_{X_a,Y_a}).
\end{align*}
\end{proof}

Recall that  the \emph{modular data} of a premodular category $\cB$ is the following pair of matrices indexed over $\Irr(\cB)$:
\begin{itemize}
    \item[(i)] \emph{$S$-matrix}. $S_{X,Y}=\operatorname{Tr}(c_{Y^*,X}\circ c_{X,Y^*})$,
    \item[(ii)]  \emph{$T$-matrix.} $T_{X,Y}=\theta_X\delta_{X,Y}$.
\end{itemize}

\begin{theorem}\label{Thm: modular data of zesting}
Let $\lambda=(\lambda,j,t,f)$ be a a ribbon zesting. The $T$-matrix and $S$-matrix of $(\cB^\lambda,t,f)$ are given by the following formulas

\begin{align}
    T_{X_a,X_a}^\lambda&= f(a) T_{X_a,X_a}\label{formula T, zesting}\\
    S^\lambda_{X_a,Y_b}&=\frac{ \dim(\lambda(a-b,-a+b))\dim(\lambda(a,-b))t^{(2)}(a,-b)f(a-b)}{t(a-b,a-b)}\label{formula S, zesting}\\
    &\times m(j,X_a,Y_b^*\otimes \lambda(b,-b)^*)\dim(X_a)^{-1}S_{X_a, Y_b}S_{X_a,\lambda(b,-b)} \notag
\end{align}
for all $ a,b\in A, X_a\in \cB_a, Y_b\in \cB_b$.
\end{theorem}
\begin{proof}
The formula for $T^\lambda$ is a direct consequence of the definition of $\theta^f$. 

For the $S$-matrix of $(\cB^\lambda,t,f)$ we use Lemma \ref{lemma: s tilde} and the fact that $\overline{X}_a=X_a^*\otimes \lambda(a,-a)^*$

\begin{align*}
    S^\lambda_{X_a,Y_b}&=\frac{ \dim(\lambda(a-b,-a+b))\dim(\lambda(a,-b))t^{(2)}(a,-b)f(a-b)}{t(a-b,a-b)}\\
    &\times m(j,X_a,Y_b^*\otimes \lambda(b,-b)^*) S_{X_a,\lambda(b,-b)\otimes Y_b}\\
    =&\frac{ \dim(\lambda(a-b,-a+b))\dim(\lambda(a,-b))t^{(2)}(a,-b)f(a-b)}{t(a-b,a-b)}m(j,X_a,Y_b^*\otimes \lambda(b,-b)^*)\\
    &\times \dim(X_a)^{-1}S_{X_a, Y_b}S_{X_a,\lambda(b,-b)},
\end{align*}where we used in the last equality that $\dim(X)S_{X,Y\otimes a}=S_{X,Y}S_{X,a}$ for every invertible object, see \cite[Proposition 8.13.10]{EGNO}.
\end{proof}

\begin{remark}
The formula \eqref{formula S, zesting} of the $S$-matrix does not look symmetric immediately. For a clearly symmetric formula, we can take the matrix $\tilde{S}_{X_a,Y_a}=\operatorname{Tr}^f(c_{Y_b,X_a}^\lambda\circ c_{X_a,Y_b}^\lambda)$ with formula given in equation \eqref{eq: formula new S tilde} of Lemma \ref{lemma: s tilde}. Now, the $S$-matrix and the $\tilde{S}$-matrix are related in the sense that $S$ is invertible if and only if $\tilde{S}$ is invertible and in that case $S=\tilde{S}^{-1}$, see \cite[Proposition 8.14.2.]{EGNO}.
\end{remark}

%\begin{remark}
%\[\tilde{S}_{X_a,Y_b}=\frac{f(a-b)t^2(a, -b)\dim(\lambda(-b,b)^*)\dim(\lambda(a,-b))\chi_{\lambda(-b,b)^*}(X_a) j_{-b}(X_a) j_{a-b}(X_a\otimes\lambda(a,-b))}{\dim(\lambda(b-a,a-b))t(a-b,a-b)j_{a}(Y_b\otimes \lambda(-b,b))j_{a-b}(Y_b\otimes \lambda(-b,b))}S_{X_a,Y_b}.\]
%\end{remark}
\subsection{M\"uger center}

We want to describe the M\"uger center $\mathcal{Z}_2((\cB^\lambda,t)) = \{X_a\in \cB_a |\, \tilde{c}_{Y_b, X_a}\circ \tilde{c}_{X_a, Y_b} = \id_{X_a\otlam Y_b}\forall\, Y_b\in\cB_b,\, b\in A\}$ of the zesting $(\cB^\lambda,t)$ of $\cB$.
We have that
\[\mathcal{Z}_2((\cB^\lambda,t)) = \{X_a\in \cB_a |\,  c_{Y_b, X_a}\circ c_{X_a, Y_b} =  j_a(Y_b)^{-1} j_b(X_a)^{-1} t^{-2}(a,b)\id_{X_a\otimes Y_b},  \forall\, Y_b\in\cB_b,\, b\in A\}.\]
This means that $X_a\in \cC_a$ is in the M\"uger center of $(\cB^\lambda,t)$ if $X_a$ projectively centralizes $Y_b$ (and the corresponding scalar is $j_a(Y_b)^{-1} j_b(X_a)^{-1} t^{-2}(a,b)$), for all $Y_b\in \cB_b$, $b\in A$.  Recall from \cite{EGNO} that for $\cD\subset\cC$ the centralizer of $\cD$ in $\cC$ is denoted $C_{\cC}(\cD)$.
\begin{lemma}\label{muger center}
Assume that $j$ is trivial. Consider a premodular $A$-graded fusion category $\cB$ and a braided zesting $(\cB^{\lambda}, t)$. If
$C_{\cB}(\cB_{ad})\subseteq\cB_{pt}$ then the M\"uger center of the zested category
\[\mathcal{Z}_2((\cB^\lambda,t)) = \{X_a\in \cB_a\cap \cB_{pt} |\,   t^{2}(a,b)= \frac{\theta_{X_a}\theta_{Y_b}}{\theta_{X_a\otimes Y_b}}, \forall\, Y_b\in\cB_b,\, b\in A\}.\]
\end{lemma}
\begin{proof}
Recall that $t$ is trivial when one of the inputs lives in the trivial component of the $A$-grading. Now, since $j$ is trivial, if $X_a\in \mathcal{Z}_2((\cB^\lambda,t))$ then $c_{Y_0, X_a}\circ c_{X_a, Y_0} = \id_{X_a\otimes Y_0} $ for all $Y_0\in \cB_0$. This means that $X_a\in C_{\cB}(\cB_0)\subseteq C_{\cB}(\cB_{ad})\subseteq \cB_{pt}$. Therefore $\mathcal{Z}_2((\cB^\lambda,t)) \subseteq \cB_{pt}$.

Moreover, since  $X_a\in \mathcal{Z}_2((\cB^\lambda,t)) \subseteq \cB_{pt}$, we have that $\chi_{X_a}(Y_b) = t^{-2}(a,b)$. From the fact that $\chi_{X_a}(Y_b) = \frac{\theta_{X_a\otimes Y_b}}{\theta_{X_a}\theta_{Y_b}}$, we get the desired characterization of the M\"uger center $\mathcal{Z}_2((\cB^\lambda,t))$ of the zested category.
\end{proof}
\begin{remark}
Notice that if $\cB$ is modular or super-modular (i.e. $\mathcal{Z}_2(\cB) = \Vec$ or $\operatorname{sVec}$), we have that $\mathcal{Z}_2((\cB^\lambda,t)) \subseteq C_{\cB}(\cB_{ad}) \subseteq \cB_{pt}$. If $\cB$ is pointed the condition is trivially satisfied too. 
\end{remark}
The following example shows that zesting a non-degenerate category may yield a degenerate one, and conversely, so that M\"uger centers are not zesting invariant.
\begin{example}
The pointed modular category $\cC(\Z/3,Q)$ with quadratic form $Q(a)=e^{2a^2\pi i /3}=\theta_a$ is naturally $\Z/3$-graded.  If we take the trivial associative $\Z/3$-zesting (i.e. trivial $2$- and $3$-cocycles $\lambda(a,b)=\unit$ and $\lambda(a,b,c)=1$) then the braided zestings correspond to a choice of a bicharacter $t:\Z/3\times \Z/3\rightarrow \ku^\times$. We can take $t(a,b)=e^{2\pi iab/3}$ so that $t^2(a,b)=Q(a)Q(b)/Q(a+b)=e^{-2\pi i ab/3}$, which implies $(\cC(\Z/3,Q)^{\lambda},t)$ is symmetric.

The converse is also possible: for the symmetric pointed category $\cC(\Z/3,P)$ with $P(a)=1$ again take the trivial associative zesting.  Now a non-trivial braided zesting corresponds to a non-trivial bicharacter $t$, which yields a non-degenerate braiding.
\end{example}

On the other hand, if we zest with respect to a group $A$ that generates a symmetric pointed subcategory and $A$ does not contain any transparent objects then the M\"uger center does not change:
\begin{proposition}\label{prop: same center}
Let $\cB$ a braided fusion category and $A\subset \Inv(\cB)$ a subgroup such that $\chi_a\neq \id$ for all $a \in A-\{0\}$. Consider the $\widehat{A}$-grading 
\[\cB_\gamma =\{X\in \cB: \chi_a(X)=\gamma(a)\id_X, \quad \forall a\in A\}, \quad \quad \gamma \in \widehat{A}.\]
For any $\widehat{A}$-braided zesting $(\lambda, j, t)$, such that

\begin{itemize}
    \item[(i)] the category generated by $A$ is symmetric,
    \item[(ii)] $j_\gamma=\id$ for all $\gamma \in \widehat{A}$,
\end{itemize}
the Müger center of $\cB$ and $(\cB^\lambda,t)$ coincide. In particular, $\cB$ is non-degenerate if and only if $(\cB^\lambda,t)$ is non-degenerate.
\end{proposition}
\begin{proof}
Since $\cB_1=\{X: c_{X,a}\circ c_{a,X}=\id_{a\otimes X}, \forall a\in A\}$, we have that $\mathcal{Z}_2(\cB)\subset \cB_1$. Hence, if  $X\in \mathcal{Z}_2(\cB)$ and $Y_\gamma\in \cB_\gamma$, then $X\otlam Y_\gamma= X\otimes Y_\gamma$ and $c^\lambda_{X,Y_\gamma}=c_{X,Y_\gamma}$. Hence $\mathcal{Z}_2(\cB)\subset \mathcal{Z}_2((\cB^\lambda,t))$.

Conversely, since the category generated by $A$ is symmetric, we have that $a\in \cB_1$ for all $a\in A$. If
\begin{align*}
    X_\gamma \in \mathcal{Z}_2((\cB^\lambda,t))\cap \cB_\gamma,&& \gamma \in \widehat{A}
\end{align*}then 
\begin{align*}
   \id_{X_{\gamma}\otimes a}=c_{a,X_{\gamma}}^\lambda \circ c_{X_{\gamma},a}^\lambda=c_{a,X_\gamma}\circ c_{X_{\gamma},a}, && \forall a\in A
\end{align*}hence $\gamma=1$. Then,
\begin{align*}
  c_{X,Y}^\gamma=c_{X,Y}, && c_{Y,X}^\lambda=c_{Y,X}  
\end{align*}
 for all $X\in \mathcal{Z}_2((\cB^\lambda,t))$ and $Y\in \cB$. This implies that $\mathcal{Z}_2((\cB^\lambda,t))\subset\mathcal{Z}_2(\cB)$, and then $\mathcal{Z}_2((\cB^\lambda,t))=\mathcal{Z}_2(\cB)$.
\end{proof}

\begin{corollary}\label{cor: same muger}
Let $\cB$ be a non-degenerate braided fusion category such that $\cB_{pt}$ is symmetric.  Then any braided $U(\cB)$-zesting of $\cB$ is non-degenerate. \qed
\end{corollary}
\begin{remark}
The pointed subcategory $\cB_{pt}$ is symmetric for any non-degenerate braided fusion category with $\cB_{pt}\subseteq \cB_{ad}$.
\end{remark}

\subsection{Braid group image}
Let $\mathbb{B}_n$ denote the braid group on $n$ strands and $\sigma_i$ its generators.  Given an object $X$ in a braided fusion category $\cB$ we will denote the associated braid group representation by

\begin{align*}
    \rho_n^X: \mathbb{B}_n& \to \operatorname{Aut}_\cB(X^{\otimes n})\\
    \sigma_i &\mapsto  \boldsymbol{a}^{-1}\circ(\id_{X}^{\otimes i-1}\otimes c_{X,X} \otimes \id^{\otimes n-i-1}_
{X})\circ \boldsymbol{a},
\end{align*} 
where $\boldsymbol{a}$ denotes the appropriate composition of associativity constraints in $\cB$. Note that the fact that this morphism is group homomorphism follows from the hexagon axioms of $\cB$ and Mac Lane's coherence theorem.
The category $\cB$ is said to have \emph{property F} if $\rho_n^X$ has finite image for
all $n$ and all objects $X$ \cite{NRo}.

Consider a braided fusion category $\cB$ with an $A$-grading and $\lambda=(\lambda, j, t)$ a braided $A$-zesting. Here we study how the image of the braid group is modified under the zesting operation.

For any $X_a\in \cB_a, Y_b\in \cB_b, Z_c\in \cB_c$ and $a,b, c\in A$, we have an algebra isomorphism 
\begin{align*}
\psi_{X_a,Y_b,Z_c}:\End_{\cB}(X_a\ot Y_b\ot Z_c)&\to \End_{\cB^\lambda}((X_a\otlam Y_b)\otlam Z_c)\\
f&\mapsto w^{-1} \big (f\ot \id_{\lambda(a,b)\ot \lambda(a+b,c)} \big )w
\end{align*}where $w= \id_{X_a\ot Y_b}\ot c_{\lambda(a,b),Z_c}\ot \id_{\lambda(a+b,c)}$.
For  $X_a\in \cB_a^\lambda$ and $n\in \Z_{>0}$, we define inductively the algebra isomorphism
\begin{align}\label{psi n}
    \psi_n^{X_a}: \End_\cB(X_a^{\otimes n}) \to \End_{\cB^\lambda}(X_a^{\otlam n})
\end{align}
by $\psi^{X_a}_{n}=\big (\psi^{X_a}_{n-1}\otlam\id_{X_a}\big )\circ \psi_{X_a^{\otlam n-2},X_a,X_a}$, where $X^{\otlam n}= X^{\otlam n-1}\otlam X$.
%$X^{\otlam n}=X^{\otlam n-1}\otlam X$, where $X^{\otlam 1}=X$. For any $n\in \Z^{>0}$ we have an algebra isomorphism
%\begin{align}\label{psi n}    \psi_n: \End_\cB(X^{\otimes n}) \to \End_{\cB^\lambda}(X^{\otlam n})\end{align} given by
Graphically, given for $f\in \End_\cB(X_a^{\otimes n})$, we have that $\psi_n^{X_a}(f)$ is given by
$$
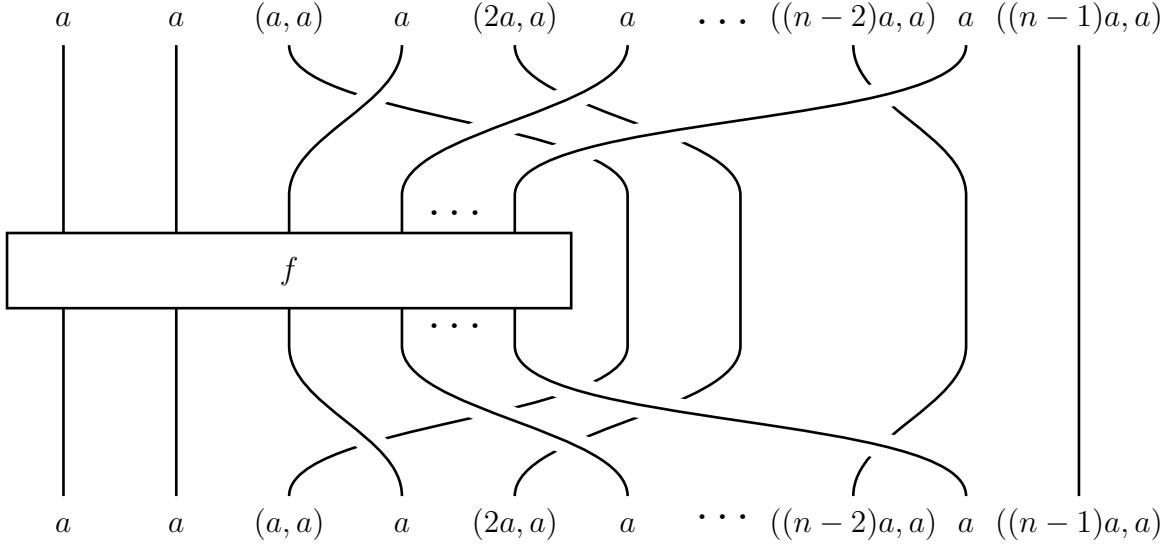
\begin{tikzpicture}[line width=1,yscale=.5, xscale=1.5]
\draw (0,0) node[below] {$\parcen{a}$} --(0,12) node[above] {$\parcen{a}$};
\draw (1,0) node[below] {$\parcen{a}$} --(1,12) node[above] {$\parcen{a}$};
\draw (2,0) node[below] {$ (a,a)$}  \br  (5,4) -- (5,8)  \br (2,12) node[above] {$ (a,a)$};
\draw (4,0) node[below] {$ (2a,a)$} \br (6,4) --(6,8) \br (4,12) node[above] {$ (2a,a)$};
\draw[white, line width=10] (3,0) \br (2,4) -- (2,8) \br (3,12) ;
\draw (3,0) node[below] {$\parcen{a}$} \br (2,4) -- (2,8) \br (3,12) node[above] {$\parcen{a}$};
\draw[white, line width=10] (5,0) \br (3,4)--(3,8) \br (5,12);
\draw (5,0) node[below] {$\parcen{a}$} \br (3,4)--(3,8) \br (5,12) node[above] {$\parcen{a}$};
\draw (7,0) node[below] {$ ((n-2)a,a)$} \br (8,4) --(8,8) \br (7,12) node[above] {$ ((n-2)a,a)$};
\draw[white, line width=10] (8,0) \br (4,4)-- (4,8) \br (8,12) ;
\draw (8,0) node[below] {$\parcen{a}$} \br (4,4) --(4,8) \br (8,12) node[above] {$\parcen{a}$};
\draw[fill=white] (-.5,5) rectangle node {$f$} (4.5,7); 
\draw (3.5,4.5) node {\Large $\cdots$};
\draw (3.5,7.5) node {\Large $\cdots$};
\draw (9,0) node[below] {$ ((n-1)a,a)$} --(9,12) node[above] {$ ((n-1)a,a)$};
\draw (5.875,0) node[below] {\Large $\cdots$};
\draw (5.875,12) node[above] {\Large $\cdots$};
\end{tikzpicture}$$
%, that is how is the braid group image under a similar homomorphism using the braid $\tilde{c}$ on $\cB^{(\lambda, j, t)}$ for a zesting of $\cB$. In particular, we prove that $\cB$ has property F if and only if any/every braided zesting $\cB^{(\lambda, j, t)}$ has property F.

\begin{theorem}
Let $\cB$ be a braided fusion category and  $\lambda=(\lambda, j, t)$ a braided $A$-zesting. Then for all $n>0$ and $X_a\in \cB_a$  simple homogeneous object, we have
\begin{align*}
    \rho^\lambda_n(\sigma_i)=j_a(X_a)t(a,a)\psi_n(\rho_n(\sigma_i)),&& 1\leq i<n,
\end{align*}
where $\rho^\lambda$ is the braid group representation associated to $(\cB^\lambda, t)$, $\psi_n$ was defined in \ref{psi n}, and $j_a(X_a), t(a,a)\in \ku^\times$ are the scalars associated with the braid $c^\lambda_{X_a,X_a}.$
%with $j_a(X_a)$ and $t(a,a)$.
%Let $\cB^{(\lambda, j, t)}$ be a braided zesting of the braided category $\cB$. The image of the braidgroup representation $\tilde{\rho}: \mathbb{C} B_n\to\operatorname{End}_{\cB^{(\lambda, j, t)}}(X_a^{\tilde{\otimes} n})$ is projectively equivalent to $\rho: \mathbb{C} B_n\to\operatorname{End}_{\cB}(X_a^{\otimes n})$. Moreover, the scalar is given by $j_a(X_a) t(a,a)$, for all $a\in A$.
\end{theorem}

\begin{proof}
If $X_a\in \cB_a$ is a simple object, we have that $c^\lambda_{X_a,X_a}=j_a(X_a)t(a,a)c_{X_a,X_a}\ot \id_{\lambda(a,a)}$. A simple graphical computation shows that 

\begin{align*}
    \rho^\lambda_n(\sigma_{n-1})=&(\boldsymbol{a}^{\lambda}_{X_a^{\otlam n-2},X_a,X_a})^{-1}\circ (\id_{X^{\otlam n-2}}\otlam c^\lambda_{X_a,X_a})\circ \boldsymbol{a}^{\lambda}_{X_a^{\otlam n-2},X_a,X_a}\\
    =& j_a(X_a)t(a,a)(\boldsymbol{a}^{\lambda}_{X_a^{\otlam n-2},X_a,X_a})^{-1}\circ (\id_{X^{\otlam n-2}}\ot c_{X_a,X_a}\ot \id_{\lambda(a^{n -1},a)})\circ \boldsymbol{a}^{\lambda}_{X_a^{\otlam n-2},X_a,X_a}\\
    =&j_a(X_a)t(a,a)\psi_n(\rho_n(\sigma_{n-1})).
\end{align*}

%\begin{align*}
%    \rho^\lambda_n(\sigma_{n-1})=j_a(X_a)t(a,a)\psi_n(\rho_n(\sigma_{n-1})),&& n\in \Z^{>0},
%\end{align*}
For $i<n-1$, we have

\begin{align*}
\rho_{n}^\lambda(\sigma_i)&=\rho_{i+1}^\lambda(\sigma_i)\otlam \id_{X^{\otlam (n-(i+1))}}\\
&= j(X_a)t(a,a)\psi_{i+1}(\rho_{i+1}(\sigma_i))\otlam \id_{X^{\otlam (n-1-i)}}\\
&=j_{a}(X_a)t(a,a)\psi_n(\rho_{n}(\sigma_i)).
\end{align*}
\end{proof}

\begin{corollary}
A premodular category $\cB$ has property F if and only if any of its ribbon zestings $(\cB^{\lambda},t)$ has property F.
\end{corollary}
\begin{proof}
Any premodular category can be included in a modular category, e.g., in its Drinfeld center.
Then,  it follows from Vafa's theorem \cite[Theorem 3.1.19]{BK}, \cite{VAFA1988421} that the double braiding in any premodular category has finite order. Then $j_a(X_a)t(a,a)$ is a root of unity, and $\rho^\lambda_n$ has finite image if and only if $\rho_n$ has finite image. 
\end{proof}

\section{Applications}\label{section: applications}

We now apply the theory of zesting to a number of familiar examples.

\subsection{Braided zesting of modular tensor categories}\label{braided zesting section}
We will apply the obstruction developed in the last section to our main case of interest. Let $\cB$ be a modular category and $A=U(\cB)$ the universal grading group. 
The maximal pointed fusion category in $\cB_{ad}=\cB_e$ is  centralized by the subcategory $\cB_{pt}$ generated by invertible objects \cite[Corollary 8.22.7]{EGNO}. In particular $\cB_{ad}\cap\cB_{pt}$ is a  symmetric pointed fusion category.

Recall that every symmetric pointed fusion category  has the form  $\Vec_S^\nu$, where $S$ is an abelian group, $\nu:S\to \Z/2$ is an additive group homomorphism and the braiding is given by $$c(a,b)=(-1)^{\nu(a)\nu(b)}\id_{a\otimes b}.$$  In particular, $\Vec_S^\nu$ is super-Tannakian in general, and Tannakian if $\nu$ is trivial.

\begin{proposition}
Let $\cB$ be a braided fusion category graded by $A$. If $\Vec_S^{\nu}\subset (\cB_e)_{pt}$, then for every $\lambda\in Z^2(A,S)$ the obstruction to the existence of an associative zesting is given by \[O_4(a_1,a_2,a_3,a_4)=(\lambda\cup_\nu\lambda) (a_1,a_2,a_3,a_4)=(-1)^{\nu(\lambda(a_1,a_2)\nu(\lambda(a_3,a_4))}.\]
In particular, if $S$ has odd order the obstruction automatically vanishes.
\end{proposition}
\begin{remark}
In practice, we overcome this obstruction by finding a $\lambda\in C^3(A,\ku^\times)$ so that 
$\delta(\lambda)(a_1,a_2,a_3,a_4)=O_4(a_1,a_2,a_3,a_4)$.  If $|S|$ is odd $O_4(a_1,a_2,a_3,a_4)\equiv 1$ and we may choose any $3$-cocycle $\lambda$ (e.g. $\lambda\equiv 1$), whereas if this obstruction does not vanish identically we must solve the linear system coming from $\delta(\lambda)(a_1,a_2,a_3,a_4)=O_4(a_1,a_2,a_3,a_4)$. Here the $3$-cochain $\lambda$ corresponds to the scalar associated with the map $\lambda_{a_1,a_2,a_3}:\lambda(a_1,a_2)\lambda(a_1a_2,a_3)\rightarrow\lambda(a_1,a_2a_3)\lambda(a_2,a_3)$.
\end{remark}

\subsection{Cyclic zesting}

\subsubsection{Cohomology of cyclic groups}

Let $C=\langle g\rangle$ be a cyclic  group of order $N$ and $M$ an abelian group. We will identify $C$ with $\Z/N$, via the isomorphism $\Z/N \to C, a \mapsto g^a$. We define some cochains associated with any $\nu \in M$ that will be useful later.
\begin{align}
\beta_\nu(i)&=i\nu, \label{def beta}\\
\gamma_\nu(i,j)&=\begin{cases} 0 & \text{ if } i+j< N\\ \nu & \text{ if } i+j\geq  N \end{cases},\label{def gamma mu}\\
\lambda_\nu(i,j,k)&=\begin{cases} 0 & \text{ if } i+j< N\\ k\nu & \text{ if } i+j\geq  N \end{cases}\label{def lambda mu},
\end{align}where $0\leq i,j,k \leq N$.  By a straightforward computation by cases, we have that

\begin{align}
\delta(\beta_\nu)&=\gamma_{N\nu}\label{delta beta}\\
\delta(\gamma_\nu)&=0\\
 \delta(\lambda_\nu)(i,j,k,l)&=\begin{cases} N\nu & \text{ if } i+j, k+l \geq N , \\ 0 & \text{ otherwise.} \end{cases} \label{delta lambda}
\end{align}

Hence, $\gamma_\nu \in Z^2(\Z/N,M)$, and  $\gamma_{N\nu}\in B^2(\Z/N,M)$.  Moreover for all $\nu \in M_N:=\{m\in M: Nm=0\}$ we have $\lambda_\nu\in Z^3(\Z/N,M)$.  It is well known that the induced group homomorphisms 
\begin{align}\label{coho cyclic}
\gamma: M/NM\to H^2(\Z/N,M), && \lambda:M_N\to H^3(\Z/N,M)
\end{align}are in fact group isomorphisms, see for more details \cite{weibel_1994}.

\subsubsection{Braided pointed fusion categories from cyclic groups}\label{subsec: cyclic braid }

Let $\cC(C,\Theta)$ be a braided pointed fusion category with $\Inv(\cB)=\langle g\rangle=C$ a cyclic group of order $N$ and ribbon structure $\Theta(g)$ such that $\dim(a)=1$ for all $a\in C$. 
We have that $\Inv(\mathcal{Z}_2(\cC(C,\Theta)))=\langle g^m\rangle$, where $m=\operatorname{Ord}(\Theta_g^2)$, and

\begin{itemize}
\item[(i)] $\cC(C,\Theta)$ is modular if and only if $\Theta_g^2 \in \ku^\times$ has order $N$.
\item[(ii)] $\cC(C,\Theta)$ is symmetric if and only if $\Theta_g^2=1$, and this case
\begin{itemize}
\item[(a)] $\cC(C,\Theta)$ is Tannakian if  and only if $\Theta_g=1$
\item[(b)] $\cC(C,\Theta)$ is super-Tannakian if and only if $\Theta_g=-1$.
\end{itemize}
\end{itemize}In the symmetric case the ribbon is a character (trivial in the Tannakian case), and the braiding can be described as 
\begin{equation}\label{eq: braid symmetric cat}
 c_{g_1,g_2}=\begin{cases} -\id_{g_1\otimes g_2}, & \text{ if } \Theta_{g_1}=\Theta_{g_2}=-1,\\ \id_{g_1\otimes g_2}, & \text{ otherwise}. \end{cases}  
\end{equation}

\subsubsection{Cyclic braided zestings}\label{section: notation cyclic zesting}

In this section we fix $\cB$ a braided fusion category and $C\subset \Inv(\cB)$ a cyclic group of order $N$ such that $\chi: C\to \Aut_\ot(\Id_\cB)$ is injective (equivalently, $C$ contains no non-trivial transparent objects). By Corollary \ref{corol grading}, $\cB$ has a faithful $\widehat{C}$-grading 
\begin{align}\label{equ: grading cyclic}
\cB_\gamma=\{X: \chi_{a}(X)=\gamma(a)\id_X, \quad \forall a\in C\}, && \gamma \in \widehat{C}.
\end{align}
Then 
\begin{equation}\label{eq: definition C perp}
\cB_1\cap C=\ker(\chi_{C,C})= C^{\perp},
\end{equation}
where $1\in \widehat{C}$ is the trivial character and $\ker(\chi_{C,C})=\{a\in C: \chi_a(b)=1, \forall b\in C\}$. Notice that while we usually denote the trivial component of our grading by $\cB_e$, here the grading is by the dual group $\widehat{C}$ of characters so we denote the trivial component by $\cB_1$ to emphasize this. We are interested in describing the braided $\widehat{C}$-zesting  induced by elements in $H^2(\widehat{C}, C^{\perp})$.

Let $g\in C$ be a generator. Thus $\widehat{C}$ can canonically be identified  with $\{\mu\in \ku^\times : \mu^N=1\}$, via $\gamma\mapsto \gamma (g)$. Hence, in order fix a generator of $\widehat{C}$, from now we fix a primitive $N$th root of unity $q\in \ku^\times$. Under the isomorphism $\Z/N \to \widehat{C}, a\mapsto [g^b\mapsto q^{ab}]$, we have that $\cB$ is $\Z/N$-graded with

\begin{align}\label{cyclic grading}
\cB_a=\{X: \chi_{g}(X)=q^a\id_X\}, &&  a\in \Z/N.
\end{align}

Let 
\begin{equation}\label{eq: generator Z_2}
h:=g^{[C:C^{\perp}]}
\end{equation}be a generator of $C^{\perp}$ and  $\Theta:C^{\perp}\to \ku^\times$ be the \emph{canonical} ribbon twist. 

\subsubsection{Associative zestings}

With the aim of unifying our results and formulas we define  $\epsilon\in \{1,0\}$ depending on an integer $a\in \Z$ and $\Theta_h \in \{\pm 1\}$ as
\begin{align}\label{def epsilon}
\epsilon= \begin{cases} 1 & \text{ if } \Theta_{h^a}=-1,\\ 0 & \text{ if } \Theta_{h^a}=1. \end{cases}
\end{align}
Additionally, we fix $\zeta \in \ku^\times$ a primitive root of unity of order $2N$ such that $\zeta^2=q$. 

In the next proposition we will follows the notation introduced in Section \ref{section: notation cyclic zesting}. 

\begin{proposition}\label{prop: assoc zestings cyclic}
Let $\cB$ a braided fusion category and $C=\langle g\rangle \subset \Inv(\cB)$ a cyclic group of order $N$ such that $\chi: C\to \Aut_\ot(\Id_\cB)$ is injective.  The equivalence classes of associative zestings of $\cB$ with respect of the grading given in \eqref{equ: grading cyclic} and associated  2-cocycle in $Z^2(\widehat{C},C^{\perp})$ are parametrized by  $\Z/m \times \Z/N$. The associative zesting corresponding to a pair $(a,b)\in \Z/m \times \Z/N$ is  given by 
\begin{align}
 \label{lambda 1} \lambda_a(i,j)=\begin{cases} \unit & \text{ if } i+j< N\\ h^a & \text{ if } i+j\geq  N \end{cases}\\ 
\label{lambda 2}
\lambda_b(i,j,k)=\begin{cases} 1 & \text{ if } i+j< N\\ \zeta^{k(\epsilon+2b)} & \text{ if } i+j\geq  N \end{cases}
\end{align}
where $0\leq i,j,k < N$, $C^{\perp}$ was defined in \eqref{eq: definition C perp},  $m=|C^{\perp}|$, $\zeta^2=q$, $h=g^{[C:C^{\perp}]}$, and $\epsilon$ is defined in \eqref{def epsilon}.
\end{proposition}
\begin{proof}
Since $C^{\perp}$ has order $m$, the 2-cocycles $\lambda_a$ with $a\in \Z/m$ form a set of representatives of $H^2(\Z/N,C^{\perp})$. 

If $\Theta_h=1$, we have that $C^{\perp}$ is Tannakian and then the 4-cocycle obstruction  $O_4(\lambda_a)$ automatically vanishes. Hence associative zestings are parametrized by pair $(a,b)\in \Z/m \times \Z/N$, with corresponding zesting

\begin{align}
\lambda_a(i,j)=\begin{cases} \unit & \text{ if } i+j< N\\ h^a & \text{ if } i+j\geq  N \end{cases}\\ 
\lambda_b(i,j,k)=\begin{cases} 1 & \text{ if } i+j< N\\ q^{kb} & \text{ if } i+j\geq  N \end{cases}
\end{align}
where $0\leq i,j,k < N$.

If $\Theta_h=-1$, we have that  $C^{\perp}$ is super-Tannakian. In particular $C^{\perp}$ has even order, and let $v:C^{\perp}\to \Z/2$ the non-trivial group homomorphism. The 4-cocycle obstruction $O_4(\lambda_a)$ 
is given by 
\begin{align*}
O_4(i,j,k,l)&=(-1)^{v(\lambda_a(i,j))v(\lambda_a(k,l))}\\
&=\begin{cases} (-1)^a & \text{ if } i+j, k+l \geq  N\\ 1 & \text{ otherwise. } \end{cases}
\end{align*}Note that the set of all solutions to the equation $Q^N=-1$ can be parametrized as $\zeta^{1+2b}$, with $b\in \Z/N$. Hence, for $a$ odd  using  \eqref{delta lambda} we have that the associative zestings is parametrized by pairs $(a,b)\in \Z/m\times \Z/N$, with 
\begin{align}
 \lambda_a(i,j)=\begin{cases} \unit & \text{ if } i+j< N\\ h^a & \text{ if } i+j\geq  N \end{cases}\\ 
\lambda_b(i,j,k)=\begin{cases} 1 & \text{ if } i+j< N\\ \zeta^{k(1+2b)} & \text{ if } i+j\geq  N \end{cases}
\end{align}
where $0\leq i,j,k < N$. 
\end{proof}

\subsubsection{Braided zestings}

\begin{proposition}\label{prop: braided cyclic zesting}
Let $(a,b)\in 	\Z/m\times \Z/N$ and  $(\lambda_a,\lambda_b)$ the associative zesting constructed in Proposition  \ref{prop: assoc zestings cyclic}. Then  
\begin{itemize}
\item[(i)] $(\lambda_a,\lambda_b)$ admits a braided zesting with $j=\id$ if and only if $a\frac{N}{m}=\epsilon+2b \mod N$. 
\item[(ii)] If $a,b$ satisfy the conditions in (i), there are $N$ different braided zestings, parametrized by a choice of an element in $\{ s\in \ku^\times: s^N=\zeta^{-(\epsilon+2b)} \}.$  
\item[(iii)] Explicitly, the braided zesting associated with $s$ as in (ii) is given by  $j_y=\Id_\cB$ for all $y \in \Z/N$ and 
\begin{align*}
t_s(i,j)=s^{-ij}\id_{\lambda_{a}(i,j)}, && 	0\leq i,j < N.
\end{align*}
\end{itemize}
\end{proposition}

\begin{proof}
In order to compute the 2-cocycles $O_1(\lambda_b)$, we can take initially $t(i,j)=\id_{\lambda_a(i,j)}$. Then
\[O_1(\lambda_b)(i|j,k)=\frac{\lambda_b(i,j,k)\lambda_b(j,k,i)}{\lambda_b(j,i,k)}=\lambda_b(j,k,i) \]
Since $H^2(\Z/N,\ku^\times)=0$, we can redefine the isomorphisms
\begin{align*}
t(i,j)=l(i,j)\id_{\lambda(i,j)}, && l(i,j)\in \ku^\times. && i,j \in\Z/N,    
\end{align*}
to that satisfy the equation in Figure \ref{fig:unnormalized-1}. We need to choose $s$ so that $s^N=\zeta^{-(\epsilon+2b)}$. Then define 
\begin{equation}
    l(i,j)= s^{-ij}
\end{equation}where $0\leq i,j,k < N$. In fact, by \eqref{delta beta}

\begin{align}
    \frac{l(i,k)l(j,k)}{l(i+j,k)}=\delta(\beta_s)^k(i,j)=\lambda_b(i,j,k)^{-1},
\end{align}
%\begin{align}   \frac{l(i+j,k)}{l(i,k)l(j,k)}=\frac{s^{-([i+j])k}}{s^{-ik}s^{-jk}}=s^{-k\big ( [i+j]-i-j\big )}=s^{kNu}=q^{-ku b}=\lambda_b(i,j,k)^{-1}\end{align}bwhere $[i+j]$ is the remainder of $i+j$ upon division  by $N$ so that $i+j=uN+[i+j]$ with $u =0$ if $i+j<N$ and $u=1$ if $i+j\geq N$.
and since $l(i,j)=l(j,i)$, we have that
\[\frac{l(i,k)l(i,j)}{l(i,j+k)}=\lambda_b(j,k,i)=O_1(\lambda_b)(i|j,k).\]
Now,

\begin{align*}
O_2(b,s)(i,j|k)=&\frac{l(i,k)l(j,k)}{l(i+j,k)}\lambda_b(i,j,k)\\
=&\lambda_b(i,j,k)^{2}
\end{align*}
%where $$hence $[O_2(b,s)]=-2b\in \Z/N\cong  H^2(\Z/N,\Z/N)$.
Finally, since $\chi_{h^a}(X_k)=q^{a\frac{N}{m}k}$, we get that

\begin{align}
 \chi_{\lambda_a(i,j)(X_1)}O_2(\lambda_a,\lambda_b)(i,j|1)=\begin{cases} 1 & \text{ if } i+j< N\\ q^{a\frac{N}{m}-(\epsilon+2b)} & \text{ if } i+j\geq  N, \end{cases}
\end{align}
that is $[\chi_{\lambda_a}/O_2]= a\frac{N}{m}-(\epsilon+2b) \in \Z/N\cong  H^2(\Z/N,\widehat{\Z/N})$. Then if $a\frac{N}{m}\neq \epsilon+2b \mod N$ the associative zesting $(\lambda_a,\lambda_b)$ does not admits a braided zesting with $j=\id$ and if $a\frac{N}{m}= \epsilon+2b \mod N$, $[O_2(\lambda)]=[\chi_{\lambda}]$, so $j_t=\id$ and $t_s(i,j)=s^{-ij}\id_{\lambda_a(i,j)}$ define a braided zesting. \end{proof}

\subsubsection{Ribbon zesting and its modular data}

\begin{proposition}\label{lemma: ribbon cyclic}
Let $(\lambda_a,\lambda_b,\id,t_s)$ be a braided zesting constructed in Proposition \ref{prop: braided cyclic zesting}.
If $\cB$ has a ribbon twist $\theta$ such that $\theta_{h^a}=\Theta_{h^a}$ then a ribbon zesting $f:\Z/N\to \ku^\times$ is defined by
\begin{align}
f(i)= & s^{-i^2},  && 0\leq i< N,
\end{align} 
and its  modular data is given by

\begin{align}\label{T-matrix cyclic}
T_{X_i,X_i}^\lambda= s^{-i^2}T_{X_i,X_i}, &&  0\leq i < N, X_i\in \cB_i,
\end{align}
\begin{align}\label{S-matrix cyclic}
S_{X_i,Y_j}^\lambda=s^{2ij}S_{X_i,Y_j}, &&  0\leq i ,j < N,  X_i\in \cB_i, Y_j\in \cB_j.
\end{align}
\end{proposition}

\begin{proof}
Equation \eqref{twist-condition}  is 
\[f(i)f(i+j)^{-1}f(j)=s^{2ij}\chi_{\lambda(i,j)}(i+j)\theta_{\lambda(i,j)}=\begin{cases} s^{2ij} & \text{ if } i+j< N\\ s^{2ij}q^{a\frac{N}{m}(i+j)}\theta_{h^a} & \text{ if } i+j\geq  N. \end{cases}\]

Let $0\leq i,j<N-1$. If $i+j< N$, then 
\begin{align*}
f(i)f(i+j)^{-1}f(j)=s^{-i^2}s^{(i+j)^2}s^{-j^2}=s^{2ij}.
\end{align*}

If $i+j\geq N$, then $i+j=[i+j]+N$ where $0\leq[i+j] <N$, and $(i+j)^2=i^2+j^2+2ij-2N(i+j)+N^2$, thus

\begin{align*}
f(i)f(i+j)^{-1}f(j)&=s^{2ij-2N(i+j)}s^{N^2}\\
&=  s^{2ij} q^{(\epsilon+2b)(i+j)}(-1)^{\epsilon+2b} \\
&= (-1)^{\epsilon} q^{a\frac{N}{m}(i+j)}s^{2ij} \\
&= \theta_{h^a} q^{a\frac{N}{m}(i+j)}s^{2ij} .
\end{align*}

Equation \eqref{ribbon condition}  for $0<i<N-1<$ is
\[f(i)f(N-i)^{-1}=\chi_{h^a}(i)\theta_{h^i}=q^{a\frac{N}{m}i}\theta_{h^a},\] and 
\begin{align*}
f(i)f(N-i)=& s^{-i^2}s^{(N-i)^2}\\
 =& s^{-2Ni}s^{N^2}=(-1)^{\epsilon}q^{(\epsilon+2b)i}\\
 =& \theta_{h_a}q^{a\frac{N}{m}i}. 
\end{align*}

The formulas for modular data follow from Theorem \ref{Thm: modular data of zesting}. In fact, the $T$-matrix follows directly from the definition of $\theta^f$.  First note that since $\theta_h=\Theta_h$, then $\dim(h^i)=1$, in particular $\dim(\lambda_a(i,j))=1$ for all $i, j\in \Z/N$, second 
\begin{align*}
    \frac{S_{X_i,\lambda(j,-j)}}{\dim(X_i)}=\chi_{h^a}(X_i^*)=q^{-a\frac{N}{m}i}, && 0<i<N,
\end{align*}and third
\begin{align*}
    s^{-i(N-j)}=s^{2ij}q^{(\epsilon+2b)i},
\end{align*}
hence

\begin{align*}
    S_{X_i,Y_j}^{\lambda}&= \frac{f(i-j)t^{(2)}(i,-j)}{t(i-j,i-j)}\frac{S_{X_i,\lambda(j,-j)}}{\dim(X_a)}S_{X_i,Y_j}\\
    &=s^{-i(N-j)}\frac{S_{X_i,\lambda(j,-j)}}{\dim(X_a)}S_{X_i,Y_j}\\
    &= s^{2ij}S_{X_i,Y_j}.
\end{align*}
\end{proof}

\begin{remark}
If $N$ is odd $f(i)=s^{-i^2}$ is the unique ribbon zesting and if $N$ is even $f(i)$ and $(-1)^if(i)$ are all the ribbon zestings.  From the proof of the above we see that this changes the $S$-matrix by a factor of $(-1)^{i-j}$ on the $(i,j)$-graded block.
\end{remark}

As a particular case of Proposition \ref{prop: braided cyclic zesting} we obtain the following result on fermion zesting (cf. \cite{16-fold}):

\begin{corollary}
Let $\cB$ be a braided fusion category and $f\in \cB$ a simple object such that \begin{itemize}
\item[(i)] $f\otimes f\cong \unit$,
    \item[(ii)]$\Theta_f=-1$,
    \item[(ii)] $\chi_f$ is not the identity.
\end{itemize}Then $\cB$ has eight different braided $\Z/2$-settings parameterized  with modular data 
\begin{align}
T_{X_i,X_i}^\lambda= s^{-i^2}T_{X_i,X_i}, &&
S_{X_i,Y_j}^\lambda=s^{2ij}S_{X_i,Y_j}
\end{align}where $s$ is a root of unity of order eight.
\end{corollary}
\qed

%If  $C\subset \cB$ is non-degenerate, it follows from \cite{??} that $\cB=C\otimes \cB_1$. It follows from Proposition \ref{prop: braided cyclic zesting} that the braided zesting correspond to $(b,s)\in Z/N\times \ku^\ti,es$ such that $2b=-\epsilon \mod N$ and $s^N=1$. Hence, $(\theta^f_{g})^2=s^2\Theta_g^2$

\subsection{Quantum Group categories of type $A$}
A large class of examples of modular categories satisfying the hypotheses of section \ref{subsec: cyclic braid }  can be obtained from quantum groups (see \cite[Section 3.3]{BK}).
Of particular interest are the
 modular categories $SU(N)_k$ obtained from quantum groups $U_Q\mathfrak{sl}_N$ for $Q=e^{\pi i/(N+k)}$ (eschewing $q$ to avoid notation clashes).  Two references for this construction are \cite{brug2,KazWen}, where much of what follows can be found.  For any $N,k$, the category $SU(N)_k$ has a (maximally) pointed subcategory $\cP(N,k)$ with fusion rules like $\bZ/N$. In particular $SU(N)_k$ is (universally) $\Z/N$-graded with trivial component $PSU(N)_k:=[SU(N)_k]_e$. Labelling the fundamental weights of the root system of type $A_{N-1}$ by $\varpi_i$ for $i=1,\ldots,N-1$ (we follow  \cite[Planches, Chapters IV,V,VI]{Bourbaki} for notation), we find that the simple objects in $\cP(N,k)$ correspond to weights $0$ and $k\varpi_i$ for $i=1,\ldots,N-1$. For notational convenience we define $g=X_{k\varpi_1}$ so that $g^0=\unit$ and $g^t:=X_{k\varpi_t}$. In this notation we have $g^i\otimes g^j=g^{i+j}$. To determine the nature of this subcategory we must compute the twists $\theta_i:=\theta_{g^i}$, for which we employ standard techniques (see \cite{BK,Row-survey}, for example). The key computation is that the twist of the simple object labeled by highest weight $\mu$ is $\theta_{\mu}=Q^{\mathfrak{c}_\mu}$ where $\mathfrak{c}:=\langle \mu+2\rho,\mu\rangle$.   We find:

\begin{equation}\label{equation:twist A_{N-1}}
\theta_j=\zeta_{2N}^{kj(N-j)}, \quad \text{ where } \zeta_{2N}=e^{\pi i/N}.
\end{equation}  Thus we may identify $\cP(N,k)$ with the pointed ribbon fusion category $\cC(\Z/N,\eta)$ where $\eta$ is the quadratic form given by $\eta(j)=\theta_j$.

By the twist equation we obtain the formula for the double braiding in $\cP(N,k)$ as: $$c_{g^t,g^s}\circ c_{g^s,g^t}=
\frac{\theta_{s+t}}{\theta_s\theta_t}=
\zeta_{2N}^{-2stk}\id_{g^{s+t}}$$ so that the M\"uger center of $\cP(N,k)$ is generated by $g^{\frac{N}{(N,k)}}$. In particular, $\cP(N,k)$ is modular
if and only if $(k,N)=1$ in which case we have the factorization $SU(N)_k\cong PSU(N)_k\boxtimes \cP(N,k)$ as modular categories, and $\Inv(PSU(N)_k)$ is trivial.

On the other hand if $N\mid sk$ then $g^s$ centralizes $\cP(N,k)$ and hence lies in the trivial component $\cB_0=PSU(N)_k$ (under the universal $\Z/N$-grading). Thus $\cP(N,k)\cap PSU(N)_k$ is a non-trivial symmetric pointed subcategory whenever $(N,k)\neq 1$. Indeed, $\cP(N,k)$ is symmetric if and only if $N\mid k$.  Furthermore, by the form of the twists calculated above we can determine when $\cP(N,k)$ for $k=\alpha N$ is Tannakian or super-Tannakian. If $N$ is odd, we only have Tannakian categories, but if $N$ is even, we have that $\cP(N,\alpha N)$ is Tannakian if and only if $\alpha$ is even, and super-Tannakian otherwise.

%Now let us assume that $N\mid k$ so that $\cP(N,k) \subset PSU(N)_k$, and hence symmetric.
The object $X_1$ labelled by the highest weight $\varpi_1$ is a tensor generator for $SU(N)_k$.  We will assume that $X_1$ is in the $1$-graded component of the universal $\Z/N$-grading $\cB_1$. 
Applying Proposition \ref{prop:chi-properties}(iv) we first compute $\chi:U(\cB)=\Z/N\rightarrow \widehat{\Z/N}$. Now $\chi_{g^a}$ is determined by $\chi_{g^a}(1)$ since $\chi_{g^a}(m)=\chi_{g^a}(1)^m$, and $\chi_{g^a}(1)=(\chi_g(1))^a$ since the operation on $\widehat{\Z/N}$ is pointwise.  Thus we reduce to computing the scalar associated with the double braiding $c_{g,X_1}\circ c_{X_1,g}$ where $X_1=X_{\varpi_1}$. 
As $g\otimes X_1=X_{(k-1)\varpi_1+\varpi_{2}}$ is simple we need only compute $$c_{g,X_1}\circ c_{X_1,g}=\frac{\theta_{(k-1)\varpi_1+\varpi_2}}{\theta_1\theta_{\varpi_1}}=\zeta_N^{-1}$$ where $\zeta_N:=e^{2\pi i /N}$.
Thus we see that $\chi_{g^a}(m)=q^{am}$ where $q:=\zeta_N^{-1}$.  In this way the $\Z/N$-grading is given by $\cB_i=\{X: \chi_g(X)=q^i\id_X\}$ as in (\ref{cyclic grading}).  Notice that $q$ is determined once we declare that $X_1\in\cB_1$ and pick our generator $g$ of $U(\cB)$--fixing any two choices among the grading, $q$ and generator of $U(\cB)$ determine the third.

Now we may apply the results of subsection \ref{braided zesting section} to $SU(N)_{k}$.  We will consider several cases to illustrate the subtleties:

\begin{enumerate}
    \item For $SU(N)_{\alpha N}$ for $N$ odd the pointed subcategory $\cP(N,\alpha N)$ is Tannakian.  We will zest with respect to the universal $\Z/N$-grading so that $m=N$ and $\epsilon=0$ in the notation of Proposition \ref{prop: assoc zestings cyclic}.  Thus there are $N^2$  associative zestings $(\lambda_a,\lambda_b)$ where $(a,b)\in\Z/N\times\Z/N$. The $N$ associative zestings for pairs $(2b,b)$ each admit $N$ braided zestings, which in turn admit $N$ twist braided zestings.  Thus there are \emph{at most} a total of $N^2$ distinct ribbon twist braided zestings, all of which are modular by Proposition \ref{prop: same center}.
    \item For $SU(N)_{\alpha N}$ with $N$ even and $\alpha$ odd the pointed subcategory $\cP(N,\alpha N)$ is super-Tannakian. We will zest with respect to the universal $\Z/N$-grading so that $m=N$ in the notation of Proposition \ref{prop: assoc zestings cyclic}. Thus there are $N^2$ distinct associative zestings $(\lambda_a,\lambda_b)$ where $(a,b)\in\Z/N\times\Z/N$. For $a$ even we have $\epsilon=0$ and the situation is similar as above: the $N$ associative zestings $(\lambda_{a},\lambda_{a/2})$ and $(\lambda_a,\lambda_{(a+N)/2})$ each admit $N$ braided zestings and $N^2$ twist braided zestings. Among these $2N$ of them are ribbon twist zestings, all of which are modular by Proposition \ref{prop: same center}. Now for $a$ odd we have $\epsilon=1$ so that the $N$ pairs $(a,\frac{a-1}{2})$ and $(a,\frac{a-1+N}{2})$ admit $N$ braided zestings each.  All told there are \emph{at most} $4N^2$ ribbon braided zestings, each of which is modular.
\end{enumerate}

We hasten to point out that, in practice, there can be equivalences among braided zestings.  We will see some examples of this below.

\subsubsection{$SU(3)_3$}
Consider the non-group-theoretical \cite{NRo} integral modular category $SU(3)_3$ of rank 10 and dimension $36$ (this example inspired the notion of zesting in  \cite{AIM2012}).
We define $q=e^{-2\pi i/3}$ and $\zeta=e^{2\pi i/18}$, and order the simple objects as follows $[\unit, g, g^2, Y, X_1,X_2,X_3,Z_1,Z_2,Z_3]$.  In the correspondence with the $SU(3)$ highest weights we have $X_1$ labeled by $\varpi_1$ and $g$ by $3\varpi_1$ as above. Then $X_2=g\otimes X_1$ and $X_3=g^2\otimes X_1$, with $Z_i=X_i^{*}$ in $SU(3)_3$. We have $\dim(Y)=3$, $\dim(X_i)=\dim(Z_i)=2$ for all $i$, and the $\Z/3$-grading is given by $X_i\in \cB_1$ and $Z_i\in\cB_2$. The twists of the simple objects ordered as above are $[1,1,1,-1,\zeta^4,\zeta^{16},\zeta^{10},\zeta^4,\zeta^{16},\zeta^{10}].$  We note that there are two inequivalent sets of modular data $(S,T)$ (and hence, presumably, modular categories) with these fusion rules: the above and its complex conjugate.  Notice that the twists of the $X_i$ are primitive $9$th roots of unity, so that there are $6$ Galois conjugates.  However, the relabeling symmetry among the pairs $(X_i,Z_i)$ allows us to recognize these $6$ conjugates as belonging to just $2$ inequivalent classes.  The (unnormalized) modular $S$-matrix has the block $3\times 3$ form:
$S=\begin{pmatrix} A& B& \overline{B}\\ B^T & C &D\\\overline{B}^T& D^T& C\end{pmatrix} $

where $$A=\begin{pmatrix}
 1 & 1 & 1 & 3\\
 1 & 1 & 1 & 3\\
 1 & 1 & 1 & 3\\
 3 & 3 & 3 & -3 \end{pmatrix}
,\quad B=2\begin{pmatrix}1&1&1\\ \zeta ^6 &  \zeta ^6 &  \zeta ^6\\
- \zeta ^6 & - \zeta ^6 &  -\zeta ^6\\
0&0&0\end{pmatrix},$$

$$C=2\begin{pmatrix}\zeta&\zeta^7&-\zeta^4\\
\zeta^7& -\zeta^4&\zeta\\
-\zeta^4&\zeta&\zeta^7\end{pmatrix},\quad \textrm{and}\quad D=2\begin{pmatrix}\zeta^2-\zeta^5&\zeta^8-\zeta^5&\zeta^8+\zeta^2\\
\zeta^8-\zeta^5&\zeta^8+\zeta^2&\zeta^2-\zeta^5\\
\zeta^8+\zeta^2&\zeta^2-\zeta^5&\zeta^8-\zeta^5
\end{pmatrix}.$$

The $9$ associative zestings of $SU(3)_3$ are parameterized by $(a,b)\in\Z/3\times\Z/3$ where $$\lambda_a(i,j)=\begin{cases} \unit, & i+j<3\\ g^a &i+j\geq 3\end{cases}\quad\text{and}\quad\lambda_b(i,j,k)=\begin{cases} 1, & i+j<3\\ q^{bk}& i+j\geq 3\end{cases}.$$  The fusion rules for $a=2$ and $a=0$ are the isomorphic: reordering the simple objects as $[\unit,g^2,g,X_1,X_3,X_2,Z_3,Z_2,Z_1]$ gives us the fusion rule isomorphism.  By results of \cite{KazWen} the $6$ fusion categories corresponding to $a=0$ and $a=2$ are obtained from $SU(3)_3$ by changing the quantum parameter $Q$ and/or changing the associativity by a $3$-cocycle. 
On the other hand, for $a=1$ we find that $\unit\not\subset X_1^{\ot_1 3}$ yet $\unit\subset X^{\ot 3}$ for all simple objects $X \in SU(3)_3$ so that these fusion rules are not isomorphic to those of $SU(3)_3$.\footnote{It seems to be the case that the trivial representation appears in $X_\mu^{\ot N}$  for any object $X_\mu\in\Rep(\mathfrak{sl}_N)$, but we could not find a proof in the literature.}

Now for each pair $(a,b)\in\{(0,0),(1,2),(2,1)\}$ we obtain $3$ braided zestings by choosing an $s$ so that $s^3=q^{-b}$, by Proposition \ref{prop: braided cyclic zesting}.  Moreover, by Proposition \ref{lemma: ribbon cyclic} these each have a unique ribbon zesting, given by multiplying the $SU(3)_3$ twists in the component $\cB_i$ by $s^{-i^2}$, and all are modular with these choices.  Thus there are at most $9$ modular categories obtained from $\Z/3$-zesting of $SU(3)_3$.  In fact, we will see that there are only $3$ inequivalent sets of modular data, and presumably only $3$ inequivalent modular categories (this is not immediate as modular data is not a complete invariant, see \cite{mignard2017modular}.

For $b=0$ we have $s^3=1$ and for $b=1$ we have $s^3=q^{-1}=e^{2\pi i/3}$.  The twist on $\cB_1$ and $\cB_2$ are primitive 9th roots of unity $\{\zeta^4,\zeta^{10},\zeta^{16}\}$, so that rescaling by $s^{-1}=s^{-4}$ with $s=e^{2\pi i n/3}$ simply permutes them.  Similarly for the case $b=2$ we have $s=\zeta^{6x+2}$ for $0\leq x\leq 2$ so that rescaling these twists by $s^{-1}$ and $s^{-4}$ conjugates the set of values, and permutes them in a way consistent with the fusion rule isomorphism above.  Now since the fusion rules and dimensions are the same, the $S$-matrices are determined by the twist (via the balancing equation).  Thus we obtain two sets of modular data from the pairs $(0,0)$ and $(1,2)$: those of $SU(3)_3$ and the complex conjugate.  Indeed, it is easily checked that adjusting the $SU(3)_3$ $S$-matrix above by a factor of $s^{2ij}$ on the $(i+1,j+1)$ block has the effect of permuting the rows/columns and possibly complex conjugating the entries.

Now for $b=2$ we have $s^3=q^{-2}=e^{-2\pi i/3}$, with solutions $s=e^{-(6x+2)\pi i/9}=\zeta_9^{-3x-1}$ for $0\leq x\leq 2$, where $\zeta_9:=e^{2\pi i/9}$.  Rescaling $\{\zeta^4=\zeta_9^2,\zeta^{10}=\zeta_9^5,\zeta^{16}=\zeta_9^8\}$ by $s^{-1}=\zeta_9^{3x+1}$ and $s^{-4}=\zeta_9^{12x+4}$ both yield $\{1,q,q^{-1}\}$ for any $x$ which is invariant under complex conjugation.  Again, the $S$-matrix is determined by the twists and the fusion rules by the balancing equation so that we find that there is exactly one set of modular data $(S,T)$ corresponding to the modular zesting of $SU(3)_3$ when $a=1$.  For completeness we provide explicit modular data $(\tilde{S},\tilde{T})$ (cf. \cite[Section 4.2]{AIM2012}): taking $x=0$, the twists are given by 
$\tilde{T}:=[1,1,1,-1,q^{-1},1,q,q,q^{-1},1]$ and $\tilde{S}=\begin{pmatrix} A& B& \overline{B}\\ B^T & \tilde{C} &\tilde{D}\\\overline{B}^T& \tilde{D}^T& \tilde{C}\end{pmatrix} $ where $A$ and $B$ are the same as for $SU(3)_3$ above and $$\tilde{C}=2\begin{pmatrix}\zeta^{-3}&\zeta^3&-1\\
\zeta^3& -1&\zeta^{-3}\\
-1&\zeta^{-3}&\zeta^3\end{pmatrix}, \quad \text{and} \quad \tilde{D}=2\begin{pmatrix}-1&\zeta^3&\zeta^{-3}\\
\zeta^3&\zeta^{-3}&-1\\
\zeta^{-3}&-1&\zeta^3 \end{pmatrix}.$$
Note that $\zeta^{\pm 3}=-q^{\pm 1}$ so that the entries of $\tilde{S}$ lie in the field $\Q(q)$.

Let us compare the zesting of $SU(3)_3$ to gauging constructions. Clearly $SU(3)_3$, its complex conjugate and its Grothendieck inequivalent zesting each contain $\Rep(\Z/3)$ as a Tannakian subcategory.  If we take the corresponding $\Z/3$-condensation \cite{DGNO} we obtain a modular category $\mathcal{L}$ of dimension $4=36/3^2$ that has a gaugable symmetry $\phi:\Z/3\rightarrow \Aut_\ot^{br}(\mathcal{L})$ \cite{SCJZ}.  It is not difficult to see that $\mathcal{L}$ must be the so-called $3$ fermion modular category $3F$, with fusion rules  like $\Z/2\times\Z/2$, and the   $\Z/3$ action cyclically permutes the fermions. Thus we should be able to recover the three zestings of $SU(3)_3$ by gauging this symmetry.  The obstructions to gauging vanish as they lie in $H^3(\Z/3,\Z/2\times\Z/2)=0$ and $H^4(\Z/3,U(1))=0$.  The gaugings of $3F$ are parameterized by $H^2(\Z/3,\Z/2\times\Z/2)=0$ and $H^3(\Z/3,U(1))\cong\Z_3$.  Thus we obtain $3$ such gaugings, consistent with the zesting calculation above.

\subsubsection{$SU(4)_4$}
The rank $35$ modular category $\cB:=SU(4)_4$ has pointed subcategory $\cP(4,4)$ with fusion rules like $\Z/4$, but is non-Tannakian: the generator $g$ of the group $\Inv(\cB)$ has twist $\theta_g=-1$, and $c_{g,g}=-\id_{g^2}$.
 We write

$$SU(4)_4=\cC_0 \oplus \cC_1\oplus \cC_2 \oplus \cC_3$$
to decompose the category into its $\Z/4$ universally-graded components, which have the following ranks
\begin{center}
\begin{tabular}{c|cccc}
    Component &  $\cC_0$ & $\cC_1$ & $\cC_2$ & $\cC_3$\\
    \hline
    Rank & 10 & 8 & 9 & 8 
\end{tabular}.
\end{center} We set $q=e^{-2\pi i/4}$ so that $\chi_g(X_i)=q^i$ for $X_i\in\cC_i$. Table \ref{su44 data} summarizes the last few pages of analysis and records the parametrization of simple objects in $SU(4)_4$ along with their universal grading, quantum dimensions, and twists.

We first consider the $\Z/4$-zestings.   To conform with the notation of Propositions \ref{prop: assoc zestings cyclic} and \ref{prop: braided cyclic zesting}, we set $\zeta=e^{-2\pi i/8}$.  The associative $\Z/4$-zestings are parameterized by $(a,b)\in\Z/4\times\Z/4$ as above.  When $a$ is odd we have $\epsilon=1$ and otherwise $\epsilon=0$.  Braided zestings exist for the $8$ pairs $$(a,b)\in\{(0,0),(0,2),(1,0),(1,2),(2,1),(2,3),(3,1),(3,3)\},$$ and are parameterized by solutions to $s^4=\zeta^{-\epsilon(a)-2b}$.  Each of these, in turn have a unique ribbon structure that gives positive dimensions, and each of these are modular by Lemma \ref{cor: same muger}.  Thus there are at most $32$ distinct modular categories obtained as $\Z/4$-zestings of $SU(4)_4$. For any triple $(a,b,s)$ the  central charge of the corresponding modular categories are the same for any of the 4 choices of $s$, giving us (at least) $8$ distinct modular categories see Table \ref{table: su4 level 4}.  As can be seen from the data in Table \ref{su44 data}, $SU(4)_4$ has a high degree of symmetry there are many objects of the same dimension giving rise to labeling ambiguities. Moreover, \cite{cain2002.03220} shows that the group of (not necessarily braided) monoidal autoequivalences is isomorphic to $\Z/2\times\Z/4$.   In particular, distinguishing or identifying the the modular categories with the same underlying fusion category (i.e. the same $(a,b)$ but different $s$) is a subtle problem.  

\renewcommand{\arraystretch}{1.2}
\begin{table}[h!]
   
    \begin{tabular}{|c|c|c|c|}\hline
        $(a,b)$ & $X_1^{\ot 4}\supset\unit$? &$s$  & central charge \\
        \hline\hline
       $(0,0)$  & yes& $\gamma$ & $\zeta_{16}^{-1}$\\
         \hline
         $(0,2)$  & yes & $\gamma\cdot \left(\zeta_{32}\right)^4$ & $-\zeta_{16}^{-1}$\\
         \hline$(1,0)$  & no  &$\gamma\cdot \left(\zeta_{32}\right)$  & $-i\cdot\zeta_{16}$\\
         \hline
         $(1,2)$  & no & $\gamma\cdot \left(\zeta_{32}\right)^5$ &$i\cdot\zeta_{16}$\\
         \hline
         $(2,1)$  & no & $\gamma\cdot \left(\zeta_{32}\right)^2$ &$-i\cdot\zeta_{16}^{-1}$\\
         \hline
         $(2,3)$  & no & $\gamma\cdot \left(\zeta_{32}\right)^6$&$i\cdot\zeta_{16}^{-1}$\\
         \hline
         $(3,1)$  & yes & $\gamma\cdot \left(\zeta_{32}\right)^3$&$-\zeta_{16}$\\
         \hline
          $(3,3)$  & yes & $\gamma\cdot \left(\zeta_{32}\right)^7$ &  $\zeta_{16}$\\
         \hline
    \end{tabular}
    \caption{$SU(4)_4$ ribbon zesting data: $\zeta_{N}:=e^{2\pi i /N}$ and $\gamma^4=1$}\label{table: su4 level 4}
    \end{table}

We may also consider the $\Z/2$-zestings.  We can define a $\Z/2$-grading on $\cB:=SU(4)_4$ by $\cB_0=\cC_0\oplus\cC_2$ and $\cB_1=\cC_1\oplus\cC_3$ where $\cC_i$ are the components of the universal grading above.  This corresponds to the grading by the subgroup $\Z/2\cong\langle g^2\rangle<\Inv(\cB)$. If we choose a $2$-cocycle $\lambda_a\in H^2(\Z/2,\Inv(\cB))$ with values in $\langle g^2\rangle$ then $\cZ_2(\langle g^2\rangle)=\langle g^2\rangle$ so we may apply the results of Propositions \ref{prop: assoc zestings cyclic} and \ref{prop: braided cyclic zesting} to obtain braided zestings and ribbon twists as above.

On the other hand, we may define a $2$-cocycle $H^2(\Z/2, \Inv(\cB))$ by $\lambda(1,1)=g$ and $\lambda(0,1)=\lambda(1,0)=\unit$.  The normalized $3$-cochains $\lambda_{\pm}(1,1,1)=\pm i$ provide associative zestings.  We claim that these fusion categories do not admit braided zestings.

We have that $\chi_g|\cB_0=\Id_{\cC_0}-\Id_{\cC_2}$.  In particular $\chi_{g}$ is not in $\Aut^{\Z/2}_\otimes(\Id_\cB)$, so that $j_a=\id$ does not satisfy condition (BZ1).  What is required is a function $j:\Z/2\rightarrow \Aut_\ot(\Id_\cB)$ such that $$\chi_{\lambda(a,b)}\circ j_{ab}\circ j_a^{-1}\circ j_b^{-1}\in \Aut_\ot^{\Z/2}(\Id_{\cB}).$$  In particular taking $a=b=1\in\Z/2$ we seek a $j_1$ such that $\chi_{g}\circ j_1^{-2}\in\Aut_\ot^{\Z/2}(\Id_{\cB})$.  This means that $\chi_{g}\circ j_1^{-2}$ must be the identity on $\cB_0$, so that $j_1^{-2}|_{\cC_2}=-\Id_{\cC_2}$ and $j_1^{-2}|_{\cC_0}=\Id_{\cC_0}$.  But since $j_1\in\Aut_\ot(\Id_\cB)$ we see that $j_1(X)\otimes j_1(Y)=j_1(X\otimes Y)=\id_{X\otimes Y}$ for $X,Y\in\Irr(\cC_0)$, and $j_1(X)=k\id_X$ and $j_1(Y)=k\id_Y$ since $j_1$ must act by a constant scalar on the simple objects in the universally-graded components.  Thus $k=\pm 1$.  But now $j_1^{-2}(X)=k^2\id_X=\id_X$, contradicting $j_1^{-2}|_{\cC_2}=-\Id_{\cC_2}$.  Alternatively we see that the second obstruction (\ref{eq:second braided zesting obstruction}) implies that $(a,b)\mapsto\chi_{\lambda(a,b)}$ should define a coboundary $\Z/2\times\Z/2\rightarrow \widehat{\Z/2}$.  But in this case $(1,1)\mapsto \psi_{\Z/2}$ where $\psi_{\Z/2}(1)=-1$ is a non-trivial character, and hence the corresponding cocycle is non-trivial.
We conclude that this associative zesting does not admit a braided zesting. 
\subsubsection{$SU(4)_2$}
The metaplectic \cite{BGPR} $\Z/4$-graded modular category $\cB=SU(4)_2\cong SO(6)_2$ has rank $10$ and dimension $24$, with pointed subcategory $\cP(4,2)\cong \cC(\Z/4,\eta)$ where $\eta(j)=e^{2\pi ij^2/4}$.  As above we will denote the  generator of $\Inv(\cB)$ by $g$ and in this case $h=g^2$ generates $\Inv(\cB_{ad})$.   We set $q=e^{-2\pi i/4}$ so that the grading is of the form $\cB_0=\{\unit,g^2,Y_1\}$, $\cB_2=\{g,g^3,Y_2\}$, $\cB_1=\{X_1,X_2\}$ and $\cB_3=\{Z_1,Z_2\}$ where $X_1$ is labeled by $\varpi_1$ and $Z_i\cong X_i^{*}$.  The dimensions are $\dim(Y_i)=2$ and $\dim(X_i)=\dim(Z_i)=\sqrt{3}$.
 We have two choices of zesting $2$-cocycle $\lambda_a\in H^2(\Z/4,\Z/2)\cong\Z_2$ given by $\lambda_a(i,j)=h^a=g^{2a}$ for $i+j\geq 4$ and $\unit$ otherwise, where $a=0,1$.  In the notation of Propositions \ref{prop: assoc zestings cyclic} and \ref{prop: braided cyclic zesting} we have $m=2$ and $\epsilon=0$, and $N=4$ so that there are $8$ associative zestings, taking $\lambda_b\in H^3(\Z/4,\ku^\times)$ as in Eqn. (\ref{lambda 2}).  By Proposition \ref{prop: braided cyclic zesting}(i) the 4 pairs $(\lambda_a,\lambda_b)$ that admit braidings correspond to $\{(0,0),(0,2),(1,1),(1,3)\}$ so that we have \emph{at most} 16 braided zestings, depending on $a,b$ and $s$ with $s^4=q^{-b}$.  In fact, we have one such braided zesting for each $s=e^{2\pi ix/16}$, since each $16$th root of unity appears. Each of these braided zestings admits $2$ ribbon twists, one of which is unitary. We will spare the reader the full details, but there are a few interesting things to note:
 \begin{enumerate}
     \item When $a=0$ and $b=0,2$ the $8$ unitary ribbon braided zestings are modular and in fact remain metaplectic \cite{BGPR}.  Since zesting leaves the trivial component unchanged, and the $2$ dimension object in the trivial component has twist $e^{2\pi i/3}$, we cannot obtain the complex conjugate category by zesting.  Indeed, we can check directly that we get 2 distinct sets of twists among the 4 choices of $s^4=1$ (for $b=0$) and similarly for the 4 choices of $s^4=-1$ (for $b=2$).  The central charges for $b=0$ are all the same, as are the central charges for $b=2$, and they are complex conjugates of each other (the cases $b=0$ and $b=2$).  
     \item For $a=1$ and $b=1,3$ we see that $g\ot_1 g=g\ot g\ot g^2=\unit$, so that $g$ is self-dual in the zested theory $\cB(1,b)$.  Moreover, $s^{-4}=q^b=\pm i$ so that the twist of $g$ in $\cB(1,b)$ is $\pm 1$ and is thus a boson or fermion.  Lemma \ref{muger center} show that, in fact $g$ is in the M\"uger center of $\cB(1,b)$, so that $\cB(1,b)\cong\Rep(\Z/2,z)\boxtimes \mathcal{A}$ where either $z=0$ or $z=1$ and $\mathcal{A}$ is a Galois conjugate of $SU(2)_4$.  In particular, $\cB(1,b)$ is not modular.
 \end{enumerate}

\begin{table}[h!]
$$
\begin{array}{cccc}
 \text{Label} & \text{Grading} & \text{Dimension} & \text{Twist}\\
\hline
 1 & 0 & 1 & 1 \\
 g & 0 & 1 & -1 \\
 g^2 & 0 & 1 & 1 \\
 g^3 & 0 & 1 & -1 \\
 Y & 0 & 2d-1 & -1 \\
 gY & 0 & 2d-1 & 1 \\
 g^2 & 0 & 2d-1 & -1 \\
 g^3 Y & 0 & 2d-1 & 1 \\
 Z & 0 & 2d-3 & -i \\
 gZ & 0 & 2d-3 & i \\
\hline
 X & 1 & \sqrt{2d} & \zeta_{64}^{15} \\
 gX & 1 & \sqrt{2d} & \zeta_{64}^{31} \\
 g^2 X & 1 & \sqrt{2d} & -\zeta_{64}^{15} \\
 g^3 X & 1 & \sqrt{2d} & -\zeta_{64}^{31} \\
 \tilde{X} & 1 & \sqrt{14d-8} & -\zeta_{64}^{7} \\
 g \tilde{X} & 1 & \sqrt{14d-8} & -\zeta_{64}^{23} \\
 g^2 \tilde{X} & 1 & \sqrt{14d-8} & \zeta_{64}^{7} \\
 g^3 \tilde{X} & 1 & \sqrt{14d-8} & \zeta_{64}^{23} \\
\hline
 X' & 2 & d & \zeta_{16}^{5} \\
 gX' & 2 & d & \zeta_{16}^{5} \\
 g^2 X' & 2 & d & \zeta_{16}^{5} \\
 g^3 X' & 2 & d & \zeta_{16}^{5} \\
 X'' & 2 & d & -\zeta_{16} \\
 gX'' & 2 & d & -\zeta_{16} \\
 g^2 X'' & 2 & d & -\zeta_{16} \\
 g^3 X'' & 2 & d & -\zeta_{16} \\
 W & 2 & 4d-4 & -\zeta_{16}^7 \\
 \hline
 X^* & 3 & \sqrt{2d} & \zeta_{64}^{15} \\
 gX^* & 3 & \sqrt{2d} & -\zeta_{64}^{31} \\
 g^2 X^* & 3 & \sqrt{2d} & -\zeta_{64}^{15} \\
 g^3 X^* & 3 & \sqrt{2d} & \zeta_{64}^{31} \\
 \tilde{X}^* & 3 & \sqrt{14d-8} & -\zeta_{64}^{7} \\
 g \tilde{X}^* & 3 & \sqrt{14d-8} & \zeta_{64}^{23} \\
 g^2 \tilde{X}^* & 3 & \sqrt{14d-8} & \zeta_{64}^{7} \\
 g^3 \tilde{X}^* & 3 & \sqrt{14d-8} & -\zeta_{64}^{23} \\
\end{array}$$
\caption{Basic data for isomorphism classes of simple objects in $SU(4)_4$. Here $d=\sqrt{2}+2$ and $\zeta_{16}=e^{2\pi i/16}$, $\zeta_{64}=e^{2\pi i/64}$ are primitive 16th and 64th roots of unity, respectively.}\label{su44 data}
\end{table}

\section{Zesting obstructions and  Eilenberg-MacLane  cohomology}\label{section: davydov nikshych}

After we posted this paper A. Davydov and D. Nikshych posted  \cite{DN} containing related results. In \cite{DN} some particular braided zestings are interpreted as deformations of a braided monoidal 2-functor and its obstruction as an element in an Eilenberg-MacLane cohomology group.

In this section, we will briefly explain the connection between some of the results in \cite{DN} and some of ours. Primarily, we want to analyze the apparent differences between the obstructions in this paper and \cite{DN}.  Essentially the differences come down to this: in \cite{DN} the cohomology class in the Elinberg-MacLane  cohomology $H^2(K(A,2),\ku^\times)$ is the obstruction for a symmetric 2-cocycle $Z^2_{\operatorname{Sym}}(A,\Inv(\cB_e))$ to admit a braided zesting, while our cohomological  obstructions  in Theorem  \ref{th: obstruction} and Corollary \ref{cor: obstruction zesting sin j} are the obstructions that a fixed associative zesting admits a braided zesting.  In practice, to compute the EM-obstruction or to explicitly describe the braided zesting (assuming it exists) one would need to go though the step-by-step process we have presented: describe the associative zestings (a 3-cochain) and check that our braided zesting obstructions vanish.

In \cite{DN} it was proved that braided extensions of a braided fusion category $\cB_e$ by a finite abelian group $A$ correspond to braided monoidal 2-functors from $A$ (seen as discrete braided monoidal 2-category) to $\operatorname{Pic}_{br}(\cB_e)$ the braided 2-categorical Picard group of $\cB_e$ (consisting of invertible central $\cB_e$-module categories), see \textit{loc. cit.} for details.  

Given a faithfully $A$-graded fusion category $\cB=\bigoplus_{a\in A}\cB_a$ there is an associated group homomorphism $f:A\to \pi_0(\operatorname{Pic}_{br}(\cB_e)), a\mapsto [\cB_a]$. In \cite[Section 8.6]{DN} braided $A$-zestings of $\cB$ 
with $j=1$ are  interpreted as liftings of $f$ (called deformations of $f$ in \textit{loc. cit.}), that is   braided monoidal 2-functors  $\mathcal{F}:A\to  \operatorname{Pic}_{br}(\cB_e)$ 
such that  $[\mathcal{F}(a)]=[\cB_a]$ for all $a$. Since $\pi_1(\operatorname{Pic}_{br}(\cB_e))=\Inv(\cZ_2(\cB_e))$, liftings of $f:A\to \pi_0(\operatorname{Pic}_{br}(\cB_e))$ are associated with elements in $H^2_{\operatorname{sym}}(A,\Inv(\cZ_2(\cB_e)))=\operatorname{Ext}_{\Z}(A,\Inv(\cZ_2(\cB_e)))$ (cohomology classes of  symmetric 2-cocycles). Now, in \cite[Proposition 8.32]{DN} they proved that the obstruction to the existence of a braided zesting associated to an element in $\lambda \in H^2_{\operatorname{sym}}(A,\Inv(\cZ_2(\cB_e)))$ is given by an element $PW
^2_{\cB_e}(\lambda)\in H^5(K(A,2),\ku^\times)$.

In order to describe $PW
^2_{\cB_e}(\lambda)$ we recall the cocycle description of $H^5(K(A,2),\ku^\times)$ for abelian groups $A$ and $M$. Let $Z^5(K(A,2),\ku^\times)$ be the abelian subgroup of $C^4(A,M)\oplus C^3(A,M)\oplus C^3(A,M)=\{a(-,-,-,-),a(-,-|-),a(-|-,-)\}$ such that 

\begin{align}
   a(xy,z,w,u)+a(x,y,zw,u)+a(x,y,z,w) =&  a(y,z,w,u)   + a(x,yz,w,u)\\ \notag &  + a(x,y,z,wu),  \\
a(x|z,w) - a(x|yz,w) + a(x|y,zw)- a(x|y,z) =&  a(x,y,z,w) - a(y,x,z,w) \label{4-br-cocy2} \\ & \notag+ a(y,z,x,w) - a(y,z,w,x)\\
a(y,z|w) - a(xy,z|w) + a(x,yz|w)- a(x,y|w) =& a(x,y,z,w) - a(x,y,w,z) \label{4-br-cocy} \\ \notag & + a(x,w,y,z) - a(w,x,y,z)\\
\label{4-br-cocy-5} a(y|z, w)-a(xy|z, w) + a(x|z, w) \  \  \  \ \  \ \ &\\  +a(x, y|zw)-a(x, y|w)-a(x, y|z) 
=&-a(x, y, z, w)+a(x, z, y, w)\notag\\& +a(x, z, w, y)-a(z, x, y, w) \notag\\ & +a(z, w, x, y)-a(z, x, w, y)  \notag\\
a(x,y|z) - a(y,x|z)=  a(x|z,y)- a(x|y,z),\label{4-br-cocy-2}
\end{align}
for all $x,y,z,w,u \in A$. Let $B^5(K(A,2),M)\subset B^5(K(A,2),M)$ the subgroup of abelian cocycles of the form

\begin{align*}
a(x,y,z,w)=& b(y, z, w)-b(xy, z, w) + b(x, yz, w)-b(x, y, zw) +b(x, y, z),\\
a(x,y|z)=& b(x|y)-b(x|yz) + b(x|z)-b(x, y, z) + b(y, x, z)-b(y, z, x),\\
a(x|y,z)=& b(y|z)-b(xy|z) + b(x|z) + b(x, y, z)-b(x, z, y) + b(z, x, y),
\end{align*} 
for some
$(b(-,-,-),b(-,-))\in C^3(A,M)\oplus C^2(A,M)$. 
The group $H^5(K(A,2),M)$ is 
by definition $Z^5(K(A,2),M)/Z^5(K(A,2),M)$. The elements in $Z^5(K(A,2),M)$ are  called \emph{abelian} cocycles.

The obstruction $PW
^2_{\cB_e}(\lambda)\in H^5(K(A,2),\ku^\times)$ is described as the cohomology class of the abelian cocycle 
\begin{align*}
PW^2_{\cB_e}(\lambda)(x, y, z, w)=& c_{\lambda(x,y),\lambda(z,w)}\\
PW^2_{\cB_e}(\lambda)(x, y|z) =& 1\\
PW^2_{\cB_e}(\lambda)(x|y,z)=&\chi_{\lambda(y,z)}^{-1}(x),
\end{align*}
where the number $\chi_{\lambda(y,z)}(x)\in \ku^\times$ corresponds to $[\cB_x,\lambda(y,z)]^{-1}$ in \cite{DN}. 

Recall that for a braided fusion category $\cB$, we denote $\cB^{rev}$ the same fusion category but with braiding $c'_{X,Y}:=c_{Y,X}^{-1}$ for all $X,Y\in \cB$.

In the following proposition we prove that checking the vanishing of $PW
^2_{\cB_e}(\lambda)$ in $H^5(K(A,2),\ku^\times)$ is basically the same as checking the condition in Corollary \ref{cor: obstruction zesting sin j}.

\begin{proposition}
Let $\cB=\bigoplus_{a\in A}\cB_a$ be a graded fusion category and $\lambda \in Z^2(A,\Inv(\cZ_2(\cB_e)))$ a symmetric 2-cocycle. Then $PW
^2_{\cB_e}(\lambda)\in H^5(K(A,2),\ku^\times)$ vanishes if and only if there is an associative zesting of $\cB^{rev}$ corresponding to $\lambda$ for which the conditions in Corollary \ref{cor: obstruction zesting sin j} hold.
\end{proposition}
\begin{proof}
Since $\lambda$ takes values in the Müger center of $\cB_e$, the cocycle $\nu_\lambda$ defined in  Figure \ref{fig:obstr} for $\cB^{rev}$ is exactly the (standard) 4-cocycle $$\nu_\lambda(x,y,z,w)=PW^2_{\cB_e}(\lambda)(x,y,z,w)=c_{\lambda(x,y),\lambda(z,w)} \quad \forall x,y,z,w\in A.$$ If the cohomology class of  $PW^2_{\cB_e}(\lambda)(-, -, -, -)$ is trivial, it follows that there is $\lambda\in C^3(A,\ku^\times)$ such that  $PW^2_{\cB_e}(\lambda)\in Z^5(K(A,2),\ku^\times)$ is cohomologous to

\begin{align*}
a(x,y,z,w)=&1 \\
a(x,y|z)=&\frac{\lambda(y, x, z)}{\lambda(x, y, z)\lambda(y, z, x)}\\
a(x|y,z)=&  \frac{\lambda(x, y, z)\lambda(z, x, y)}{\lambda(x, z, y)}\chi_{\lambda(y,z)}^{-1}(x),
\end{align*}
and $(\lambda(-,-), \lambda(-,-,-))$ is an associative zesting for $\cB^{rev}$.

Now, if follows from the equation \eqref{4-br-cocy} that $a(-,-|z)\in Z^2(A,\ku^\times)$ for all $z\in A$. The cohomology classes of $a(-,-|z)^{-1}$ correspond to the obstructions $O_1(\lambda)(a,-,-)$ defined in Lemma \ref{prop:obst O1}. If this obstructions vanish there is $b(-|-)\in C^2(A,\ku^\times)$ such  $PW^2_{\cB_e}\in Z^5(K(A,2),\ku^\times)$ is cohomologous to

\begin{align}
a'(x,y,z,w)=&1 \\
a'(x,y|z)=&1\\
a'(x|y,z)=& \frac{b(y|z)b(x|z)}{b(xy|z)} \frac{\lambda(x, y, z)\lambda(z, x, y)}{\lambda(x, z, y)}\chi_{\lambda(y,z)}^{-1}(x).
\end{align}
It follows from equations \eqref{4-br-cocy-2}, \eqref{4-br-cocy-5}, and  \eqref{4-br-cocy2} that $a'(-|-,-)\in Z^2_{sym}(A,\widehat{A})$ (symmetric 2-cocycles). Hence the cohomology of $PW^2_{\cB_e} \in Z^5(K(A,2),\ku^\times)$ vanishes if and only if $[\chi_{\lambda(y,z)}(x)]=\left [\frac{b(y|z)b(x|z)}{b(xy|z)} \frac{\lambda(x, y, z)\lambda(z, x, y)}{\lambda(x, z, y)} \right ]$ in $H^2(A,\widehat{A})$. Since the cohomology class of  $\frac{b(y|z)b(x|z)}{b(xy|z)} \frac{\lambda(x, y, z)\lambda(z, x, y)}{\lambda(x, z, y)}$ agree with $O_2(\lambda,b)
^{-1}$ defined in Figure \ref{fig:second-obs}, the conditions in Corollary \ref{cor: obstruction zesting sin j} holds for some associative zesting of $\cB^{rev}$ if and only if $PW^2_{\cB_e}$ vanishes. 
\end{proof}
\section{Conclusions and Future Directions}
We have developed the general theory of associative zesting for fusion categories and a further theory of braided, twist and ribbon zestings for categories with these additional structures and properties.  We have illustrated their utility with a few examples, notably establishing the existence of a modular category of rank $10$ and dimension $36$ obtained by zesting $SU(3)_3$ that was conjectured in \cite{AIM2012}.  Moreover we have shown that braided zesting preserves property $F$, and given explicit computations of the modular data for braided zestings of modular categories.  While zesting shares some similarities with symmetry gauging, the explicit nature of zesting is a distinct advantage.  

This work suggests several interesting directions for future applications.  Note that we have mostly applied our theory to zesting modular categories with respect to the universal grading.  While these are perhaps the most interesting and most transparent examples, it would be interesting to apply associative zesting to fusion categories that do not admit a braiding and braided zesting with respect to non-universal grading groups and non-cyclic grading groups.  Finally we point out that symmetry gauging has a physical interpretation as phase transitions of topological phases of matter.  We do not know if zesting has a meaningful physical interpretation.

%\bibliographystyle{plain}
%\bibliography{biblio}

\end{document}